\documentclass[reqno, a4paper]{amsart}
\usepackage{fixltx2e}
\usepackage{bookmark}
\usepackage{etoolbox}
\usepackage[usenames, dvipsnames]{color}

\definecolor{l-gray}{gray}{0.5}
\usepackage[english]{babel}
\usepackage{enumitem}

\usepackage[utf8]{inputenc}  
\usepackage[T1]{fontenc}
\usepackage{pifont}
\usepackage{gensymb}

\usepackage{rotating}
\usepackage{scalerel}

\usepackage{amsmath}
\usepackage{amssymb}
\usepackage{amsthm, remreset}
\usepackage{amscd}
\usepackage{centernot}
\usepackage{mathdots}
\usepackage[all,color,cmtip]{xy}

\makeatletter
\@namedef{subjclassname@2020}{%
  \textup{2020} Mathematics Subject Classification}
\makeatother

\newdir{ >}{{}*!/-7pt/\dir{>}}
\newdir{->-}{\dir{-}*\dir{>}*\dir{-}}
\newdir{  |}{{}*!/-20pt/\dir{|}}
\newdir{ |}{{}*!/-7pt/\dir{|}}
\newdir{|- }{!<2.5pt,0pt>\dir{|}*\dir{-}*{}}
\newdir{> }{!/7pt/\dir{>}*{}}
\newdir{>}{{}*:(1,-.2)@^{>}*:(1,+.2)@_{>}}
\newdir{<}{{}*:(1,+.2)@^{<}*:(1,-.2)@_{<}}

\newdir{>>}{{}*!/3.5pt/:(1,-.2)@^{>}*!/3.5pt/:(1,+.2)@_{>}*!/7pt/:(1,-.2)@^{>}*!/7pt/:(1,+.2)@_{>}}
\newdir{ >>}{{}*!/8pt/@{|}*!/3.5pt/:(1,-.2)@^{>}*!/3.5pt/:(1,+.2)@_{>}}
\newdir{ |>}{{}*!/-3.5pt/@{|}*!/-8pt/:(1,-.2)@^{>}*!/-8pt/:(1,+.2)@_{>}}
\newdir{ >}{{}*!/-8pt/@{>}}

\def\pullback{
 \ar@{-}[]+R+<3pt,-1pt>;[]+RD+<3pt,-3pt>%
 \ar@{-}[]+D+<1pt,-3pt>;[]+RD+<3pt,-3pt>}
 
 \def\lowpullback{
 \ar@{-}[]+R+<6pt,-6pt>;[]+D+<15pt,-16pt>;[]+RD+<6pt,-6pt>%
 \ar@{-}[]+D+<4pt,-10pt>;[]+RD+<6pt,-16pt>}

\def\pullbackdots{%
 \ar@{.}[]+R+<6pt,-1pt>;[]+RD+<6pt,-6pt>%
 \ar@{.}[]+D+<1pt,-6pt>;[]+RD+<6pt,-6pt>}
 
\def\pushout{%
 \ar@{-}[]+L+<-5pt,1pt>;[]+LU+<-5pt,5pt>%
 \ar@{-}[]+U+<-1pt,5pt>;[]+LU+<-5pt,5pt>}
 
\def\uppushout{%
 \ar@[Gray]@{-}[]+L+<-5pt,-1pt>;[]+LD+<-5pt,-5pt>%
 \ar@[Gray]@{-}[]+D+<-1pt,-5pt>;[]+LD+<-5pt,-5pt>}
 
 \def\lowuppushout{%
 \ar@{-}[]+L+<-10pt,-2pt>;[]+LD+<-10pt,-10pt>%
 \ar@{-}[]+D+<-2pt,-10pt>;[]+LD+<-10pt,-10pt>}

\def\splitpullback{%
 \ar@{-}[]+R+<6pt,-.51ex>;[]+RD+<6pt,-6pt>%
 \ar@{-}[]+D+<.51ex,-6pt>;[]+RD+<6pt,-6pt>}

\def\skewpullback{%
 \ar@{-}[]+R+<6pt,-1pt>;[]+RD+<-7pt,-8pt>%
 \ar@{-}[]+D+<-12pt,-8pt>;[]+RD+<-7pt,-8pt>}

\def\dottedpullback{%
 \ar@{.}[]+R+<6pt,-1pt>;[]+RD+<6pt,-6pt>%
 \ar@{.}[]+D+<1pt,-6pt>;[]+RD+<6pt,-6pt>}

\usepackage{geometry}
\let\origappendix\appendix 
\renewcommand\appendix{\clearpage\pagenumbering{roman}\origappendix}

\usepackage{graphicx}
\usepackage{epstopdf}
\newtheorem{theorem}[subsubsection]{Theorem}
	
\newtheorem{observation}[subsubsection]{Observation}	
\newtheorem{corollary}[subsubsection]{Corollary}	
\newtheorem{proposition}[subsubsection]{Proposition}	
\newtheorem{notation}[subsubsection]{Notation}	
\newtheorem{lemma}[subsubsection]{Lemma}
\newtheorem{definition}[subsubsection]{Definition}

\newtheorem{example}[subsubsection]{Example}

\newtheorem{remark}[subsubsection]{Remark}
\newtheorem{convention}[subsubsection]{Convention}

\newcommand{\Z}{\mathbb{Z}}

\newcommand{\N}{\mathbb{N}}
\newcommand{\EE}{\mathcal{E}}

\newcommand{\C}{\mathcal{C}}

\newcommand{\X}{\mathcal{X}}

\newcommand{\SET}{\mathsf{Set}}

\newcommand{\QND}{\mathsf{Qnd}}

\newcommand{\RCK}{\mathsf{Rck}}
\newcommand{\EXT}{\mathsf{Ext}}

\newcommand{\CEXT}{\mathsf{CExt}}

\newcommand{\ARR}{\mathsf{Arr}}

\newcommand{\GRP}{\mathsf{Grp}}

\newcommand{\AB}{\mathsf{Ab}}
\newcommand{\TOP}{\mathsf{Top}}

\newcommand*{\defeq}{\mathrel{\vcenter{\baselineskip0.5ex \lineskiplimit0pt
                     \hbox{\scriptsize.}\hbox{\scriptsize.}}}%
                     =}
\newcommand*{\eqdef}{\mathrel{=\vcenter{\baselineskip0.5ex \lineskiplimit0pt
                     \hbox{\scriptsize.}\hbox{\scriptsize.}}}%
                     }
\newcommand*{\qndiop}{\triangleleft^{-1}}
\newcommand*{\qndop}{\triangleleft}

\newcommand\ttop{\scaleobj{0.6}{\top}}
\newcommand\pperp{\scaleobj{0.6}{\perp}}

\DeclareMathOperator{\Co}{C_{0}}
\DeclareMathOperator{\Ci}{C_{1}}
\DeclareMathOperator{\Cii}{C_{2}}

\DeclareMathOperator{\Ii}{I_{1}}
\DeclareMathOperator{\Fi}{F_{1}}
\DeclareMathOperator{\Fii}{F_{1}^i}

\DeclareMathOperator{\FIi}{F_{2}}
\DeclareMathOperator{\FIil}{F_{\hat{2}}}

\DeclareMathOperator{\FII}{F_{2}^i}

\DeclareMathOperator{\Fa}{F_{ab}}
\DeclareMathOperator{\Fg}{F_g}

\DeclareMathOperator{\Fq}{F_q}
\DeclareMathOperator{\Fr}{F_r}

\DeclareMathOperator{\id}{id}

\DeclareMathOperator{\ab}{ab}

\DeclareMathOperator{\ER}{ER}

\DeclareMathOperator{\Eq}{Eq}

\DeclareMathOperator{\Adj}{Adj} 

\DeclareMathOperator{\Conj}{Conj}
\DeclareMathOperator{\Pth}{Pth}

\DeclareMathOperator{\pth}{pth}
\DeclareMathOperator{\U}{U}
\DeclareMathOperator{\I}{I}

\DeclareMathOperator{\Ker}{Ker}

\usepackage{comment}
\includecomment{comment}
\definecolor{britishracinggreen}{rgb}{0.0, 0.26, 0.15}
\definecolor{darkpink}{rgb}{0.91, 0.33, 0.5}

\newcommand{\fr}[1]{{\color{red}#1}}
\newcommand{\db}[1]{{\color{Blue}#1}}

\newcommand{\gr}[1]{{\color{darkpink}{\underline{#1}}}}
\newcommand{\grl}[1]{{\color{darkpink}{#1}}}
\newcommand{\red}[1]{{\color{Red}#1}}

\newcommand{\R}[1]{{\rm(R#1)}}

\newcommand{\K}[1]{{\rm K}_{#1}}

\begin{document}

\title[Higher coverings of racks and quandles -- Part II]{Higher coverings of racks and quandles -- Part II}
\author{Fran\c{c}ois Renaud}

\email[Fran\c{c}ois Renaud]{francois.renaud@uclouvain.be}

\address{Institut de Recherche en Math\'ematique et Physique, Universit\'e catholique de Louvain, che\-min du cyclotron~2 bte~L7.01.02, B--1348 Louvain-la-Neuve, Belgium}

\thanks{The author is a Ph.D.\ student funded by \textit{Formation à la Recherche dans l'Industrie et dans l'Agriculture} (\textit{FRIA}) as part of \textit{Fonds de la Recherche Scientifique - FNRS}}

\begin{abstract}
This article is the second part of a series of three articles, in which we develop a higher covering theory of racks and quandles. This project is rooted in M.~Eisermann's work on quandle coverings, and the categorical perspective brought to the subject by V.~Even, who characterizes coverings as those surjections which are categorically central, relatively to trivial quandles. We extend this work by applying the techniques from higher categorical Galois theory, in the sense of G.~Janelidze, and in particular we identify meaningful higher-dimensional centrality conditions defining our higher coverings of racks and quandles.

In this second article (Part II), we show that categorical Galois theory applies to the inclusion of the category of coverings into the category of surjective morphisms of racks and quandles. We characterise the induced Galois theoretic concepts of trivial covering, normal covering and covering in this two-dimensional context. The latter is described via our definition and study of double coverings, also called algebraically central double extensions of racks and quandles. We define a suitable and well-behaved commutator which captures the zero, one and two-dimensional concepts of centralization in the category of quandles. We keep track of the links with the corresponding concepts in the category of groups.
\end{abstract}
\maketitle
\section{Introduction}
This article (\emph{Part II}) is the continuation of \cite{Ren2020} which we refer to as \emph{Part I}. In the following paragraphs we recall enough material from Part I, and the first order covering theory of racks and quandles, in order to contextualize and motivate the content of the present paper, in which, based on the firm theoretical groundings of categorical Galois theory, we identify the \emph{second order coverings} of racks and quandles, and the \emph{relative} concept of \emph{centralization}, together with the definition of a suitable \emph{commutator} in this context. We refer to Part I for more details, further motivations, references, and historical remarks about the material in this introduction.
\subsection{Recalls and notations}
We like to describe racks as sets equipped with a self-distributive system of symmetries, attached to each point. More precisely, a \emph{rack} is a set $A$ equipped with a binary operation $\qndop\colon A \times A \to A$ such that for each $a\in A$, the function $- \qndop a \colon A \to A$ (which is called the \emph{symmetry attached to $a$}) admits an inverse (denoted $- \qndiop a \colon A \to A$) and is \emph{compatible} with the operation $\qndop$ (self-distributivity), i.e. for all $x$, $a$ and $b$ in $A$:
\begin{enumerate}
\item[\R1] $ (x \qndop a) \qndiop a = x = (x \qndiop a) \qndop a$; 
\item[\R2]$(x \qndop a) \qndop b =(x \qndop b) \qndop (a \qndop b)$.
\end{enumerate}
A \emph{morphism of racks} is a function between racks which preserves the operation $\qndop$. The thus defined \emph{category of racks} ($\RCK$) is a variety of algebras in which the \emph{subvariety of quandles} ($\QND$) is the full subcategory of $\RCK$ whose objects $Q$ are such that $a \qndop a = a$ for each $a \in Q$. It is the category of those racks whose symmetries are required to fix the point they are attached to. We write $\Fr$ (\emph{free rack functor}) and $\Fq$ (\emph{free quandle functor}) for the left adjoints to the forgetful functors $\U\colon \RCK \to \SET$ and $\U \colon \QND \to \SET$ to the category of sets. The units of the corresponding adjunctions are denoted by $\eta^r$ and $\eta^q$ respectively. Similarly $\Fg\colon \SET \to \GRP$ is the \emph{free group functor} and the unit of the adjunction $\Fg \dashv \U$ is denoted $\eta^g$.

One of the earliest motivations for the study of racks and quandles arises from their link with group conjugation (Paragraph \ref{ParagraphTightRelationshipWithGroups}). Main applications are to be found in knot theory and subsequently in physics or even computer science -- see for instance \cite{Joy1982,Bri1988,FenRou1992,Dri1992,FeRoSa1995,CaKaSa2001,CJKLS2003,Kau2012,Eis2014} and references there. Another important line of work concerns the concept of \emph{symmetric space} -- see \cite{Loos1969,Tak1943,Ber2008,Hel2012}. The categories of racks and quandles are also an interesting context for applying and developing categorical algebra \cite{Eve2014,EveGr2014,EveGr2016,EveGrMo2016,Bou2015,BouMon2018,Bou2020}.

\subsubsection{The connected component adjunction} 
In topology, it is convenient to view the category of sets ($\SET$) as the category of \emph{discrete topological spaces}. The functor sending a topological space to its set of connected components is then obtained as the left adjoint to the inclusion of $\SET$ in the category of topological spaces ($\TOP$).

For our purposes, we observe that $\SET$ is isomorphic to the subvariety of  \emph{trivial quandles} defined by the additional axiom ``$x \qndop a = x$''. In other words, each set $X$ can be viewed in a unique way as a quandle whose ``system of symmetries'' is obtained by attaching the identity function $ - \qndop a \defeq \id_X$ to each element $a$ in $X$. The left adjoint of the inclusion $\I \colon \SET \to \RCK$ (or $\I \colon \SET \to \QND$) defines the \emph{connected component functor} $\pi_0\colon \RCK \to \SET$ (respectively $\pi_0\colon \QND \to \SET$), with unit $\eta \colon A \to \I (\pi_0(A))$, where $\pi_0(A)$ is the set of ($\Co A$)-equivalence classes of elements in $A$. Here $\Co A$ is the congruence describing when two elements of $A$ are considered to be \emph{connected} \cite{Joy1979}. Two elements $x$ and $y$ of a rack (or quandle) $A$ are connected if one, e.g. $y$, is the result of the successive action of certain symmetries (or their inverses) on the other, here $x$. We define (see \cite{Joy1982}) that $(x,y) \in \Co A$ if and only if there exists $n \in \N$ and elements $a_1$, $a_2$, $\ldots$, $a_n$ in $A$ such that 
\begin{equation}\label{EquationEndpoint}
y= x \qndop^{\delta_1}  a_1 \qndop^{\delta_2} a_2  \cdots\qndop^{\delta_n} a_n,
\end{equation}
for some $\delta_i \in \lbrace -1,\, 1 \rbrace$ for $1 \leq i \leq n$; where by convention the left-most operation is computed first. The data of such a formal sequence $(a_i^{\delta_i})_{1\leq i \leq n}$ is called a \emph{primitive path}. Together with the data of a \emph{head} $x$, these describe a \emph{primitive trail}, whose \emph{endpoint} is then given by $y$. Equation \eqref{EquationEndpoint} describes how the primitive path $(a_i^{\delta_i})_{1\leq i \leq n}$ \emph{acts} on a head $x$ in order to produce the endpoint $y$. Note that a same primitive path $(a_i^{\delta_i})_{1\leq i \leq n}$ can act on different heads, yielding different endpoints.

This connected component adjunction is the \emph{base} of the covering theory of interest, in the same way that, from the perspective of categorical Galois theory, the connected component adjunction in topology is at the base of the classical covering theory of topological spaces. The key definitions and developments of this article (e.g. double coverings of racks and quandles, our commutator and the relative concept of centrality in this context) merely require a minimalist introduction to the relevant concepts from categorical Galois theory, such as provided below, or in Part I (see for instance \cite{JK1994} for more details).

\subsubsection{Basic categorical Galois theory}\label{ParagraphBasicCategoryTheory}
In summary, we consider a convenient particular instance of the general theory which was developed in \cite{Jan1990}. The axiomatic framework in which categorical Galois theory takes place is that of a \emph{Galois structure}. For our purposes, a Galois structure, say $\Gamma$, mainly consists in the data of a category $\C$ (for instance $\QND$), a subcategory $\X$ in $\C$ (for instance $\SET$), together with a reflection of $\C$ on $\X$, i.e.~the data of a left adjoint $F\colon \C \to \X$ to the inclusion $\I\colon \X \to \C$ (e.g. $\pi_0 \colon \QND \to \SET$; note that we often omit $\I$ from our notations in what follows). The ``bigger'' context $\C$ is understood to be more ``sophisticated'', more difficult to study, whereas the ``smaller'' context $\X$ is supposedly more ``primitive'', or merely better understood. In order to obtain a Galois structure from such a reflection, we also need to specify a class of morphisms in $\C$, whose ``elements'' will be called \emph{extensions} (which are the surjective morphisms of quandles in our example).

The purpose of Galois theory is then to study special classes of extensions in $\C$ which are naturally associated to those extensions which lie in the subcategory $\X$. In this work, we call an extension which is a morphisms in $\X$ a \emph{primitive extension}. These special classes of extensions in $\C$ which are associated to primitive extensions then measure a sphere of influence of $\X$ in $\C$ (with respect to the chosen concept of extension). In particular, the most important special class of extensions is the class of \emph{$\Gamma$-coverings}, or simply \emph{coverings}, also called \emph{central extensions}, defined below. In our example with $\QND$, the induced concept of $\Gamma$-covering coincides with the concept of \emph{covering} defined by M.~Eisermann in \cite{Eis2014}, as it was first proved in \cite{Eve2014}. A covering, or central extension of quandles (or racks), is a surjective morphism of quandles (or racks) $f\colon A \to B$ such that $x \qndop a = x \qndop b$ whenever $a$, $b$ and $x$ in $A$ are such that $(a,b)$ is in the kernel pair $\Eq(f)$ of $f$, i.e.~such that $f(a) = f(b)$. A surjective morphism of racks or quandles $f\colon A \to B$ can be \emph{centralized}, using a quotient of its domain, given by the \emph{centralization congruence} $\Ci A$ over $A$ generated by the pairs $(x,y)$ such that there exists $(a,b) \in \Eq(f)$ such that $y = x \qndop a \qndiop b$ (see \cite{DuEvMo2017} and Part I).

In general, given a suitable Galois structure $\Gamma$ one defines three different special classes of extensions, the simplest of which is the class of \emph{trivial $\Gamma$-coverings}, or simply \emph{trivial coverings} (also referred to as \emph{trivial extensions}). A trivial covering is defined (in a suitable Galois structure) as an extension $t\colon T \to E$ which is the pullback of a primitive extension $p\colon X \to F(E)$ in $\X$, along the unit morphism $\eta_{E}\colon E \to F(E)$ (left-hand square in Diagram \eqref{DiagramDefinitionTrivialCovering}). In a suitable Galois structure, the category of trivial coverings above an object $E \in \C$ is equivalent to the category of primitive extensions above $F(E)$ in $\X$. In topology, this Galois-theoretic definition of trivial covering coincides with the classical definition of trivial covering \cite[Section 1.3]{Hat2001} (for a suitable choice of Galois structure $\Gamma_T$, whose reflection is the connected component functor to $\SET$ \cite[Section 6.3]{BorJan1994}). Note that both in the example from topology, and in the example in $\QND$, a primitive extension is just a surjective function in $\SET$.
\begin{equation}\label{DiagramDefinitionTrivialCovering}
\vcenter{\xymatrix@C=13pt@R=8pt{ X \ar[d]_-{p} & T \ar[d]^-{t} \ar[l] \ar[r] & A \ar[d]^-{c} \\
F(E) & E \ar[l]^-{\eta_{E}} \ar[r]_-{e} & B
}}
\end{equation}
The class of \emph{$\Gamma$-coverings}, or simply \emph{coverings} is then defined (in a suitable Galois structure $\Gamma$) as the class of those extensions $c\colon A \to B$ in $\C$, for which there exists another extension $e\colon E \to B$, which is said to \emph{split} $c$, i.e.~such that the pullback $t$ of $c$ along $e$ is a trivial extension (right-hand square in Diagram \eqref{DiagramDefinitionTrivialCovering}). In certain contexts, coverings are also refferred to as \emph{central extensions}, such as in \cite{JK1994}, in reference to the example from group theory described below. Again, the classical concept of a topological covering arises as the concept of a $\Gamma_T$-covering defined in that same suitable Galois structure for topological spaces. The remaining special class of extensions is the class of \emph{normal $\Gamma$-coverings}, or simply \emph{normal covering} (or \emph{normal extension}), which are those extensions which are split by themselves.

As it is the case in topology (under some suitable assumptions) the coverings above a given object of $\C$ can be classified using data which is \emph{internal} to $\X$. Informally we say that the ``behaviour of coverings is tractable using information which remains in the simpler context described by $\X$''. The fundamental theorem of categorical Galois theory \cite[Theorem 3.7]{Jan1990} formalises this idea in such a way that the classification of topological coverings and other examples of so-called ``Galois correspondences'' appear as particular instances of it \cite{BorJan1994}. One of these particular instances is the classical Galois theory of field extensions, and both of its generalizations by A.~Grothendieck and A.~R.~Magid. Besides the example from topology, we are also interested in the theory of central extensions from group theory, which is another instance of categorical Galois theory \cite{Jan1990,JK1994}, using the Galois structure (say $\Gamma_G$) obtained from the \emph{abelianization functor} and the class of surjective group homomorphisms (see Paragraph \ref{ParagraphTightRelationshipWithGroups}). The $\Gamma_G$-coverings are the classically called \emph{central extensions of groups}. As mentioned before, and generalizing from this example, the terminology \emph{central extension} is used in certain contexts, such as in \cite{JK1994}, to describe what is more generally called a $\Gamma$-covering. Finally, the classification results for quandle coverings obtained in \cite{Eis2014} also derive from this fundamental theorem of categorical Galois theory, as we partially explained in Part I.

Note that in this article, we further develop the theory of quandle and rack coverings by using higher dimensional categorical Galois theory. In particular we identify a suitable commutator for the study of coverings and \emph{double coverings} in this context, together with a \emph{relative higher centrality condition} which is \emph{compatible} with the zero-dimensional centrality $\Co$ and one-dimensional centrality $\Ci$ described before. The analogy is to be made with the corresponding developments in group theory as we describe in Part I and below. Because of the similarities between the corresponding Galois structures, many aspects of the covering theory of racks and quandles can also be interpreted using the covering theory of topological spaces (Part I and \cite{Eis2014}).

We precisely define what a ``suitable'' Galois structure is for our purposes in Convention \ref{ConventionGaloisStructure}. Again, we describe and explain the concepts at play to the extend which is necessary for appreciating the content of this article. As it is stated in \cite{EveGr2014} and Part I, the adjunction $\pi_0 \dashv \I$ between the categories $\RCK$ (or similarly for $\QND$) and $\SET$ is part of a \emph{strongly Birkhoff} \emph{Galois structure} denoted $\Gamma \defeq (\RCK,\SET,\pi_0,\I,\eta,\epsilon,\EE)$ (see Convention \ref{ConventionGaloisStructure} and Section \ref{SectionAdmissibility}), where the \emph{class of extensions} $\EE$ is the class of surjective morphisms of racks.

\begin{convention}\label{ConventionGaloisStructure} For our purposes, a \emph{Galois structure} $\Gamma \defeq (\C,\X,F,\I,\eta,\epsilon,\EE)$ (see \cite{JanComo1991}), is the data of an inclusion $\I$, of a \emph{full} (\emph{replete}) subcategory $\X$ in a category $\C$, with left adjoint $F\colon{\C \to \X}$, unit $\eta$, counit $\epsilon$ and a chosen class of \emph{extensions} $\EE$ within the morphisms of $\C$. The class $\EE$ is subject to the following conditions:
\begin{enumerate}
\item $\EE$ contains all isomorphisms, and $\EE$ is closed under composition;
\item the reflection (by $F$) of an extension yields an extension;
\item pullbacks along extensions exist, and the pullback of an extension is an extension.
\end{enumerate} 
For our purposes, $\EE$ will always be a class of \emph{regular epimorphisms}. Moreover, we require the components of the unit $\eta$ to be extensions, i.e.~for each object $X$ in $\C$, $\eta_X\colon X \to \I F X$ is an extension. Such a category $\X$ is said to be $\EE$-reflective in $\C$. Finally, taking pullbacks along extensions should be a ``well behaved'' operation i.e.~we require our extensions to be of \emph{effective $\EE$-descent} in $\C$ (see Section \ref{SectionBeyondBarrExactness} below).

As mentioned before, we call \emph{primitive extensions}, those extensions $p$ which lie in $\X$. A \emph{trivial $\Gamma$-covering} or \emph{trivial covering} is an extension $t$ which is the pullback of a primitive extension $p$ along a unit morphism $\eta_X$ (provided that $\Gamma$ is \emph{admissible} or \emph{strongly Birkhoff}, see Section \ref{SectionAdmissibility}). A \emph{normal $\Gamma$-covering} or \emph{normal covering} is such that the projections of its kernel pair are trivial coverings. A \emph{$\Gamma$-covering} or \emph{covering}, sometimes called \emph{central extension}, is an extension $c\colon A \to B$ such that there is another extension $e\colon E \to B$ such that the pullback $t$ of $c$ along $e$ is a trivial covering.
\end{convention}

\subsubsection{Strong connections with groups}\label{ParagraphTightRelationshipWithGroups}
There is an enlightening parallel to be made between our work and the corresponding developments in group theory. Firstly because the algebraic structures of racks and quandles are intimately related to group conjugation, and secondly because such a parallel helps to navigate the possible outcomes for the development of a higher order covering theory of racks and quandles (see Part I).

The term \emph{wrack} was first used by J.C.~Conway and G.C.~Wraith, in an unpublished correspondence of 1959, to describe the ``wreckage'' of a group, whose multiplication operation has been forgotten, and only the operation of conjugation remains. The functor $\Conj\colon \GRP \to \RCK$ (or its restriction $\Conj\colon \GRP \to \QND$) sends a group $G$ to the quandle with underlying set $G$ and operation defined by conjugation $x \qndop a \defeq a^{-1} x a$. Interestingly, the left adjoint of $\Conj$ plays an important role in the covering theory of racks and quandles, as it was first noticed by M.~Eisermann, and further investigated in Part I. The functor $\Conj$ is right adjoint to the functor $\Pth\colon \RCK \to \GRP$ (first defined by D.E.~Joyce as $Adconj$, see also $\Adj$ in \cite{Eis2014}) which sends a rack $A$ to the group $\Pth(A)$ obtained as the following quotient in $\GRP$: \[ \xymatrix{ \Fg(\U(A)) \ar[r]^-{q_A} & \Pth(A) \defeq  \Fg(\U(A))/\langle\langle \gr{c}^{-1}\gr{a}^{-1}\gr{x}\,\gr{a} \mid a,x,c \in A \text{ and } c= x \qndop a \rangle\rangle,}\] 
where $\langle\langle z \mid [...] \rangle\rangle$ stands for the normal subgroup generated by the elements $z$ such that $[...]$. The unit $\pth$ of this adjunction $\Pth \dashv \Conj$ is given by $\pth_A \defeq q_A \eta^g_{\U(A)} \colon A \to \Conj(\Pth(A)) \  \colon \ a \mapsto \gr a$. We write $\vec{f} \defeq \Pth(f)$ for the image by $\Pth$ of a morphism of racks or quandles. We often omit $\Conj$, $\U$ and $\I$ from our notations.

We explained in Part I how to derive the definition of $\Pth$ from looking at equivalence classes of terms in the ``language of racks''. Subsequently we showed that an element $g$ in $\Pth(A)$, which we call a \emph{path}, represents an equivalence class of \emph{homotopically equivalent} primitive paths -- in the sense of the covering theory of racks. We recall that the \emph{action} $x \cdot g$ of such a path $g \in \Pth(A)$ on an element $x \in A$ is defined using the action of generators $\gr a \in \Pth(A)$ for which $x \cdot \gr a \defeq x \qndop a$ and $x \cdot \gr a^{-1} \defeq x \qndiop a$ (as we saw for primitive paths). A \emph{trail} $(x,g)$ in $A$ is the data of a \emph{head} $x \in A$ and a \emph{path} $g\in \Pth(A)$, and the \emph{endpoint} of $(x,g)$ is defined by $x \cdot g$. For each $g,h \in \Pth(A)$ and $x\in A$ we have that: if $e$ is the neutral element in $\Pth(A)$, then $x \cdot e = x$; moreover, $x \cdot (gh) = (x \cdot g) \cdot h$, and  $(x\qndop y) \cdot g = (x \cdot g) \qndop (y \cdot g)$; finally $\gr{(x \cdot g)} \defeq \pth_A(x \cdot g) = g^{-1} \gr{x}  g$ (see \emph{augmented quandle} \cite{Joy1979}).

Given a morphism of racks $f\colon A \to B$, we have $f(x \cdot g) = f(x) \cdot \vec{f}(g)$. If $A = \Fr(X)$ for some set $X$, then $\Pth(A) = \Fg(X)$ acts \emph{freely} on $A$, which plays an important role in the characterization of coverings (see Section \ref{SectionTowardsHigherCoveringTheory}). For the covering theory of quandles, homotopy-equivalence classes of primitive paths are represented by the elements of $\Pth\degree(Q)$, which is the normal subgroup of $\Pth(Q)$ generated by pairs of generators $\gr a\, \gr b^{-1}$ for $a$ and $b$ in $Q$. This divergence is easily understood via the comparison of free racks and free quandles, as a consequence of the additional \emph{idempotency axiom} in $\QND$ (see Part I). It is also the case that if $Q = \Fq(X)$ for some set $X$, then $\Pth\degree(Q) = \Fg\degree(X)$ acts \emph{freely} on $Q$, where $\Fg\degree(X)$ is $\langle \langle ab^{-1} \mid a,b \in \Fg(X)\rangle \rangle$. Moreover, the foundational concepts of interest for the covering theory (relative centrality) in $\RCK$ and $\QND$ coincide in the appropriate sense (and in every dimension); see Part I for more explanations.

If the abelianization functor $\ab \colon \GRP \to \AB$, into the category of abelian groups $\AB$, is the left adjoint to the inclusion of $\AB$ in $\GRP$, and $\Fa$ is the \emph{free abelian group functor}; then we have the following square of adjunctions (Diagram \eqref{DiagramSquareAdjD1}) describing the relationship between $\pi_0 \dashv \I$ in $\RCK$ (or $\QND$) and the \emph{abelianization} in $\GRP$. All but one of the four possible squares of functors below commute -- ($\pi_0$, $\Conj$, $\U$, $\ab$) doesn't. Moreover, a group is abelian if and only if its conjugation operation is trivial. Similarly, a surjective group homomorphism ($f\colon G \to H$) is central (in the sense of group theory, i.e., its \emph{kernel} $\Ker(f) \leq Z(G)$ is in the \emph{center} of $G$) if and only if $\Conj(f)$ is a covering of quandles. The image by $\Pth$ of a covering in $\RCK$ is a central extension of groups.
\begin{equation}\label{DiagramSquareAdjD1}
\vcenter{\xymatrix@C=22pt@R=18pt{ \RCK \ar@{{}{}{}}[dd]|-{\dashv}  \ar@<-3pt>@/_7pt/[rr]_-{\pi_0}  \ar@<-3pt>@/_7pt/[dd]_-{\Pth} \ar@{{}{}{}}[rr]|-{\top}  & & \SET \ar@{{}{}{}}[dd]|-{\dashv}   \ar@<-3pt>@/_7pt/[ll]_-{\I} \ar@<-3pt>@/_7pt/[dd]_-{\Fa} \\
\\
 \GRP  \ar@{{}{}{}}[rr]|-{\top}  \ar@<-3pt>@/_7pt/[uu]_-{\Conj} \ar@<-3pt>@/_7pt/[rr]_-{\ab}   & &  \AB  \ar@<-3pt>@/_7pt/[ll]_-{\I} \ar@<-3pt>@/_7pt/[uu]_-{\U}
}}
\end{equation}

\subsection{Towards higher dimensions}\label{SectionTowardsHigherExtensions}
In order to extend the covering theory of racks and quandles to higher dimensions, we first look at the arrow category $\EXT\RCK$ (or $\EXT\QND$). Given any category $\C$ with a chosen class of extensions $\EE$ (Convention \ref{ConventionGaloisStructure}), $\EXT\C$ refers to the full subcategory of extensions within the category of morphisms (arrow category) $\ARR\C$. A morphism $\alpha\colon{f_A \to f_B}$ in such a category of morphisms is given by a pair of morphisms in $\C$, which we denote $\alpha = (\alpha_{\ttop},\alpha_{\pperp})$ (the \emph{top} and \emph{bottom components} of $\alpha$) as in the commutative Diagram~\eqref{EquationDoubleExtension}. 
\begin{equation}\label{EquationDoubleExtension}
\vcenter{\xymatrix @R=0.6pt @ C=8pt{
A_{\ttop} \ar[rrr]^-{\alpha_{\ttop}} \ar[rd] |-{p}  \ar[dddd]_{f_A} & & &  B_{\ttop}  \ar[dddd]^-{f_B} \\
 & A_{\pperp} \times_{B_{\pperp}} B_{\ttop}  \ar[rru]|-{\pi_2} \ar[lddd]|-{\pi_1}  \\
\\
\\
A_{\pperp}  \ar[rrr]_-{\alpha_{\pperp}} &  & & B_{\pperp}
}} 
\end{equation} 
If all morphisms in this commutative diagram are in $\EE$, including $\alpha$'s so-called \emph{comparison map} $p$, then $\alpha$ is said to be a \emph{double extension} \cite{Jan1991}. For our purposes, the class of double extensions $\EE^1$ is the appropriate induced class of (two-dimensional) extensions in $\EXT\C$.


The inclusion $\I \colon \CEXT\RCK \to \EXT\RCK$ (and similarly for $\I \colon \CEXT\QND \to \EXT\QND$), of the full subcategory of coverings in the category of extensions, admits a left adjoint $\Fi \colon \EXT\RCK \to \CEXT\RCK$. The functor $\Fi$ universally makes an extension into a covering (or central extension). It is said to \emph{universally centralize} an extension (one-dimensional centralization) in the same way that $\pi_0$ \emph{universally trivializes} objects (zero-dimensional centralization). The unit $\eta^1$ of the adjunction $\Fi \dashv \I$ is defined for an extension $f\colon A \to B$ by $\eta^1_f \defeq (\eta^1_A, \id_B)$, where the kernel pair $\Eq(\eta^1_A)$ of the quotient $\eta^1_A$ is denoted $\Ci f$. As mentioned before, it is generated by the pairs $(x\qndop a \qndiop b, x)$ for $x$, $a$, and $b$ in $A$ such that $f(a) = f(b)$. Then $\eta^1_f \in \EE^1$ is a double extension making $\CEXT\RCK$ into an $\EE^1$-reflective subcategory of $\EXT\RCK$ (Convention \ref{ConventionGaloisStructure}). This data then fits into the square of adjunctions
\begin{equation}\label{DiagramSquareAdmAdjD2}
\vcenter{\xymatrix@C=20pt@R=15pt{ \EXT\RCK \ar@{{}{}{}}[dd]|-{\dashv}  \ar@<-3pt>@/_7pt/[rr]_-{\Fi}  \ar@<-3pt>@/_7pt/[dd]_-{\Pth^1} \ar@{{}{}{}}[rr]|-{\top}  & & \CEXT\RCK \ar@{{}{}{}}[dd]|-{\dashv}   \ar@<-3pt>@/_7pt/[ll]_-{\I} \ar@<-3pt>@/_7pt/[dd]_-{\Pth^1} \\
\\
 \EXT\GRP  \ar@{{}{}{}}[rr]|-{\top}  \ar@<-3pt>@/_7pt/[uu]_-{\Conj^1} \ar@<-3pt>@/_7pt/[rr]_-{\ab^1}   & &  \CEXT\GRP \ar@<-3pt>@/_7pt/[ll]_-{\I} \ar@<-3pt>@/_7pt/[uu]_-{\Conj^1}
}}
\end{equation}
where the functors $\Pth^1$ and $\Conj^1$ are the appropriate restrictions of the adjoint pairs induced by $\Pth \dashv \Conj$ between the categories of morphisms above $\RCK$ and $\GRP$. The functor $\ab^1$ is the \emph{centralization functor} in $\GRP$ sending a surjective group homomorphism $f\colon G \to H$ to the central extension of groups $\ab^1(f)\colon G/[\Ker f,G]_{\GRP} \to H$ obtained from the quotient of the domain of $f$: $G/[\Ker f,G]_{\GRP}$ where $\Ker f$ is the kernel of $f$ and for any subgroups $X$ and $Y \leq G$, the subgroup $[X,Y]_{\GRP} \defeq \langle xyx^{-1}y^{-1} \mid x \in X,\ y\in Y \rangle \leq X \cap Y \leq G$ defines the classical commutator from group theory. As before, all squares of functors in Diagram \eqref{DiagramSquareAdmAdjD2} commute, but for the square ($\Fi$, $\Conj^1$, $\ab^1$, $\U$) which doesn't (see Example \ref{ExampleCentralizationsDontCoincide}).

\subsection{Content}
In Section \ref{SectionGaloisStructureInDimTwo}, we first show that categorical Galois theory applies to the adjunction $\Fi \dashv \I$ on the top line of Diagram \eqref{DiagramSquareAdmAdjD2}, which we fit into a \emph{strongly Birkhoff Galois structure} $\Gamma^1$ (Section \ref{SectionAdmissibility}). Alongside the results of Part I, this mainly consists in the recollection of classical properties of double extensions, including a bit of \emph{descent theory} (see \cite{JanTho1994,JanSoTho2004} and references therein). The rest of the article is then aimed at the characterization and ``visualization'' of the induced notion of \emph{$\Gamma^1$-covering}, or \emph{double central extension of racks and quandles}, as it was previously done for groups \cite{Jan1991}, leading to the developments of \cite{EvGrVd2008,Ever2010}. Note that the study of the more technical categorical aspects of Section \ref{SectionGaloisStructureInDimTwo} is not necessary for the readers' understanding of what follows. Section \ref{SectionTowardsHigherCoveringTheory} provides a useful transition to the rest of the article, as we recall our general method for the characterization of coverings, and produce a first visual representation of trivial $\Gamma^1$-coverings. We define and study the concept of \emph{double covering}, also called \emph{algebraically central double extension of racks and quandles} in Section \ref{SectionDoubleCoverings}. We provide examples, and the definition of a meaningful and well-behaved notion of commutator, which captures (in the usual sense) the centralization congruence for objects, extensions and double extensions in the category of quandles. We illustrate our methods and definitions via the characterization of normal $\Gamma^1$-coverings, leading to a better understanding of two-dimensional centrality. The concept of double covering (or algebraically central double extension of racks and quandles) and the concept of $\Gamma^1$-covering (or double central extension of racks and quandles) are then shown to coincide in Section \ref{SectionCharacterization}. Section \ref{SectionCentralization} is dedicated to the centralization of double extensions (i.e. the reflection of the category of double extensions on the category of double coverings) leading to the next step of the covering theory. Finally, in Section \ref{SectionFurtherDevelopments} we hint at further research and we adapt the concept of \emph{Galois structure with (abstract) commutators} in such a way that fits to our context and remains compatible with the developments in \cite{Jan2008}.

\section{An admissible Galois structure in dimension 2}\label{SectionGaloisStructureInDimTwo}
In order to apply categorical Galois theory (see Paragraph \ref{ParagraphBasicCategoryTheory}) to the inclusion $\I\colon \CEXT\C \to \EXT\C$ (where $\C$ stands for $\RCK$ or $\QND$) of the category of coverings of racks (or quandles) in the category of extensions, we first fit the reflection $\Fi \dashv \I$ into a Galois structure $\Gamma^1 \defeq (\EXT\C, \CEXT\C, \Fi, \I, \eta^1, \epsilon^1, \EE^1)$, satisfying the conditions of Convention \ref{ConventionGaloisStructure}. In order to do so, we need an appropriate class of extensions in dimension 2. In dimension 1, the base category $\C = \RCK$ or $\C = \QND$ is finitely \emph{cocomplete} and \emph{Barr-exact} \cite{Bar1971}, like any variety of algebras. In short this means that $\C$ has finite limits and colimits, it is \emph{regular} (i.e.~every morphism factors uniquely, up to isomorphism, into a regular epimorphism, followed by a monomorphism, and these factorizations are stable under pullbacks) and, moreover, every \emph{equivalence relation} is the kernel pair of its \emph{coequalizer}. In such a context, a fruitful class of extensions is given by the class of \emph{regular epimorphisms}. However, for a general Barr-exact $\C$, the category $\EXT\C$ is not necessarily Barr-exact (see comment preceding Definition 3.4 \cite{EvGrVd2008}); $\EXT\RCK$ and $\EXT\QND$ even fail to be regular categories (see below). The class of regular epimorphisms is then not appropriate for applying Galois theory. As mentioned before, the appropriate class of extensions (in the category of extensions) is given by the class of double extensions ($\EE^1$).


Given $\C$ and $\EE$ as above, let us briefly recall some well known basic properties of the category $\EXT\C$, full subcategory of extensions within $\ARR\C$. Limits in $\ARR\C$ are \emph{computed component-wise}. Given a diagram $D$ in $\ARR\C$, compute the limits $L_{\ttop}$ and $L_{\pperp}$ in $\C$ of the diagrams obtained as the top component of $D$ and the bottom component of $D$ respectively. The limit $l\colon L_{\ttop} \to L_{\pperp}$ of $D$ in $\ARR\C$ is given by the induced comparison map between $L_{\ttop}$ and $L_{\pperp}$. Using the regularity of $\C$, limits can be computed in $\EXT\C$ as the regular epic part $e$ of the regular epi-mono factorization $l=me$ of the limit in $\ARR\C$ (precompose the legs of the limit $L_{\pperp}$ with the mono part $m$ to obtain the bottom legs of the limit $e$ in $\EXT\C$). Pushouts in $\EXT\C$ are computed component-wise in $\C$. The initial object is the identity on the initial object of $\C$. The coequalizer of a parallel pair of morphisms in $\EXT\C$ is computed component-wise in $\C$, and the resulting commutative square is a pushout square of regular epimorphisms. Given a morphism $\alpha$ in $\EXT\C$ which is a pushout-square of regular epimorphisms in $\C$, it is the coequalizer of its kernel pair computed in $\EXT\C$. Regular epimorphisms in $\EXT\C$ are thus the same as (oriented) pushout squares of regular epimorphisms in $\C$. Monomorphisms are morphisms for which the top component is a monomorphism in $\C$. Regular epi-mono factorizations exist, and are unique in $\EXT\C$, however these might not be pullback stable in general.

\begin{remark} When $\C$ is $\RCK$ or $\QND$, regularity of $\EXT\C$ would imply that the category of surjective functions $\EXT\SET$ is regular (since $\EXT\SET$ is closed under regular quotients and limits in $\EXT\C$). We recall that not all regular epimorphisms in $\EXT\SET$ are pullback stable (see also \cite[Remark 3.1]{JanSo2011} and Remark \ref{RemarkEffectiveDescentRegC}). Since $\C$ is Barr exact, $\EXT\SET$ is equivalent to the category $\mathsf{ERSet}$ of (internal) equivalence relations over $\SET$. Using the arguments from \cite[Section 2]{JanSo2002}, a regular epimorphism in $\mathsf{ERSet}$ is given by a morphism $\bar{\alpha}\colon \Eq(f_A) \to \Eq(f_B)$ and a surjective morphism $\alpha_{\ttop}\colon A_{\ttop} \to B_{\ttop}$ that commute with the projections of the equivalence relations $\Eq(f_A) \rightrightarrows A_{\ttop}$ and $\Eq(f_B) \rightrightarrows B_{\ttop}$ (as in the top-right corner of the commutative Diagram \eqref{Diagram3by3}) and such that $(b,b') \in \Eq(f_B)$ if and only if there exists a finite sequence $(a_1,a_1')$, $\ldots$, $(a_n,a_n')\in \Eq(f_A)$ with $b = \alpha_{\ttop} (a_1)$, $\alpha_{\ttop} (a_i')=\alpha_{\ttop} (a_{i+1})$ for $i \in \lbrace 1, \ldots, n-1 \rbrace$ and $\alpha_{\ttop} (a_n')=b'$ \cite[Proposition 2.2]{JanSo2002}. Such a morphism is a pullback stable regular epimorphism if and only if it is a regular epimorphism such that $(b,b') \in \Eq(f_B)$ if and only if there exists $(a,a')\in \Eq(f_A)$ with $b = \alpha_{\ttop} (a_1)$ and $\alpha_{\ttop} (a')=b'$ \cite[Proposition 2.3(b)]{JanSo2002}. We adapt \cite[Example 2.4]{Lor2020} to this context: define $A_{\ttop} = \lbrace (0,a), (0,b),(1,a),(1,b) \rbrace$, $B_{\ttop} = \lbrace 0, 1, 2 \rbrace$ and $\alpha_{\ttop}$ such that $\alpha_{\ttop}(0,a) = 0$, $\alpha_{\ttop}(1,b) = 2$ and $\alpha_{\ttop}(0,b)= \alpha_{\ttop}(1,b)=1 \in B_{\ttop}$. If $\Eq(f_A)$ is the equivalence relation generated by the pairs $((0,a),(1,a))$ and $((0,b),(1,b))$; and $\Eq(f_B)$ is $B_{\ttop} \times B_{\ttop}$, then the pair $(\bar{\alpha},\alpha_{\ttop})$ defines a regular epimorphism in $\mathsf{ERSet}$, but it is not pullback stable. Indeed, its pullback along the inclusion of $ \lbrace 0, 2 \rbrace \times  \lbrace 0, 2 \rbrace \rightrightarrows \lbrace 0, 2 \rbrace$ in $B_{\ttop} \times B_{\ttop} \rightrightarrows B_{\ttop}$ is not a regular epimorphism.
\end{remark}

It is convenient to bring back a problem or computation in $\EXT\C$ to a couple of problems and computations in $\C$, using the projections on the top and bottom components (this component-wise decrease in dimension is essential for the inductive approach to higher covering theory \cite{Ever2010}). ``From an engineering perspective'', our interest in the concept of a double extension lies in the fact that pullbacks of such, and subsequently many other constructions involving double extensions, can be computed component-wise in $\C$.

\subsection{Basic properties of double extensions}\label{SectionBasicPropertiesOfDoubleExtensions}

In short, we hope for the class of double extensions to have as many good properties in $\EXT\C$ as the class of regular epimorphisms has in the Barr-exact category $\C$. Note that since double extensions are regular epimorphisms in $\EXT\C$, a double extension which is a monomorphism in $\EXT\C$ is an isomorphism. Then observe that Proposition 3.5 and Lemma 3.8  from \cite{EvGrVd2008} easily generalize to our context as it was observed in \cite{EvGoeVdl2012} (Example 1.11, Proposition 1.6 and Remark 1.7). For any \emph{regular} \cite{Bar1971} category $\C$:
\begin{lemma}\label{LemmaDoubleExtensionsCalculus}
\begin{enumerate}
\item If a morphism $\alpha = (\alpha_{\ttop},\alpha_{\pperp})$ in $\EXT\C$ is such that $\alpha_{\ttop}$ is an extension and $\alpha_{\pperp}$ an isomorphism, then $\alpha$ is a double extension.
\item Double extensions are closed under composition.
\item Pullbacks along double extensions (exist in $\EXT\C$ and) are computed \emph{component-wise}. Moreover the pullback of a double extension is a double extension.
\end{enumerate}
\end{lemma}

Given a pair of morphisms $\alpha = (\alpha_{\ttop},\alpha_{\pperp})\colon{f_A \to f_B}$ and $\gamma = (\gamma_{\ttop},\gamma_{\pperp})\colon{f_C \to f_B}$ in $\EXT\C$, their \emph{component-wise pullback} is given by the following commutative diagram in $\C$ \[\vcenter{\xymatrix@1@!0@=32pt{
& C_{\ttop} \times_{B_{\ttop}} A_{\ttop} \pullback \ar[rr]^-{\pi_{A_{\ttop}}} \ar[dd]|(.3){\pi_{C_{\ttop}}}|-{\hole} \ar@{{}{..}{>}}[ld]_-{f} && A_{\ttop} \ar[dd]^-{\alpha_{\ttop}} \ar[ld]|-{f_A} \\
C_{\pperp} \times_{B_{\pperp}} A_{\pperp} \pullback \ar[rr]^(.75){\pi_{A_{\pperp}}} \ar[dd]_-{\pi_{C_{\pperp}}} && A_{\pperp} \ar[dd]^(.25){\alpha_{\pperp}} \\
& C_{\ttop} \ar[ld]_-{f_C} \ar[rr]^(.25){\gamma_{\ttop}}|-{\hole} && B_{\ttop} \ar[ld]^-{f_B} \\
C_{\pperp} \ar[rr]_-{\gamma_{\pperp}} && B_{\pperp}}} \] where the front and back faces are pullbacks -- i.e.~it is the pullback of $\alpha$ and $\gamma$ in the arrow category $\ARR\C$. Provided that $\alpha$ is a double extension, Lemma \ref{LemmaDoubleExtensionsCalculus} above says that $f$ is an extension and the pullback of $\alpha$ and $\beta$ in $\EXT\C$ is given by $f$ together with the projections $(\pi_{A_{\ttop}},\pi_{A_{\pperp}})$ and $(\pi_{C_{\ttop}},\pi_{C_{\pperp}})$, where the latter is actually a double extension. In particular, the kernel pair of a double extension $\alpha=(\alpha_{\ttop},\alpha_{\pperp})$ exists in $\EXT\C$ and is given by the kernel pairs of $\alpha_{\ttop}$ and $\alpha_{\pperp}$ in each component (together with the induced morphism between those -- see Notation \ref{NotationDoubleExtAndKernelPair}). Moreover, the legs of such kernel pairs are also double extensions (see Lemma \ref{LemmaBarrKock}).

Lemma \ref{LemmaDoubleExtensionsCalculus} is important for what follows, if only because pullbacks along extensions appear everywhere in categorical Galois theory. As we mentioned earlier, if neither $\alpha$ or $\beta$ is known to be a double extension, their pullback in $\EXT\C$ still exists, but it is not necessarily computed component-wise and thus it is badly behaved. As we move to higher dimensions, these general pullbacks are no longer convenient for our purposes.

Note that in the context of Barr-exact \emph{Mal'tsev} categories \cite{CKP1993}, double extensions are the same as pushout squares of extensions, and, as a rule, higher extensions are easier to identify -- primarily using split epimorphisms. Note that the lack of such arguments is a challenge in our more general context where categories are not Mal'tsev categories.

As we may conclude from \cite[Proposition 3.3]{EvGoeVdl2012} and the fact that our categories are not \emph{Mal'tsev}, the axiom (E4) of ``right-cancellation'' considered there (see also  \cite[Lemma 3.8]{EvGrVd2008}) cannot hold in our context. We have the following weaker version: 

\begin{lemma}\label{LemmaWeakRightCancellation}
If the composite $\alpha  \beta$ is a double extension in $\EXT\C$, and $\beta$ is a commutative square of extensions in $\C$, then $\alpha$ is a double extension.
\end{lemma}
\begin{proof}
Since $\alpha_{\ttop} \beta_{\ttop}$ and $\alpha_{\pperp} \beta_{\pperp}$ are regular epimorphisms, $\alpha$ is a square of extensions in $\C$. The rest of the proof is an easy exercise.
\end{proof}

In particular we deduce that pullbacks along double extensions reflect double extensions. 

\begin{corollary}\label{CorollaryPullbacksReflectDExtensions}
Given a morphism $\alpha = (\alpha_{\ttop},\alpha_{\pperp})$ in $\EXT\C$, if its pullback along a double extension $\beta$ yields a double extension $\alpha'$, then $\alpha$ is itself a double extension.
\end{corollary}

\begin{observation}\label{ObservationCompositeOfDoubleExtensions} Finally we note that if a composite of two double extensions $\alpha \beta$ is a pullback square, then both $\alpha$ and $\beta$ are easily shown to be pullback squares.
\end{observation}


\subsection{Beyond Barr exactness, effective descent along double extensions}\label{SectionBeyondBarrExactness}
Given a Galois structure $\Gamma$ as in Convention \ref{ConventionGaloisStructure}, $\Gamma$-coverings, which are the key concept of study, are defined as those extensions $c\colon A \to B$ for which there is another extension $e\colon E \to B$ such that the pullback $t$ of $c$ along $e$ is a trivial $\Gamma$-covering. In most references \cite{BorJan1994,GraRo2004,Jan2008}, $e$ is further required to be of \emph{effective descent} or \emph{effective $\EE$-descent} (see \cite{JanTho1994,JanSoTho2004} and references therein). Such extensions are sometimes also called \emph{monadic extensions} \cite{JanComo1991,Ever2014}. In the contexts of interest for this work, we shall always have that all our extensions are of effective $\EE$-descent, which is why we use this simplified definition of covering. The idea is to ask that ``pulling back along $e$ is an \emph{algebraic} operation'', which is necessary for the ``information about coverings to be \emph{tractable} in $\X$'' in the sense of the fundamental theorem of categorical Galois theory (see for instance \cite[Corollary 5.4]{JanComo1991}).  We didn't insist on this requirement for Part I (see also \cite{JK1994}) since the class of effective descent morphisms in a Barr-exact $\C$ is well known to be the class of regular epimorphisms \cite{JanTho1994}.

Given an extension $e\colon{E \to B}$ in $\C$, if we write $\ARR(Y)$ for the category of morphisms with codomain $Y$, then there is an induced pair of adjoint functors: $e_*\colon \ARR(E) \to \ARR(B)$, left adjoint of $e^*\colon\ARR(B) \to \ARR(E)$, where $f_*(k \colon X \to A) \defeq fk$, and $f^*(h\colon Y \to B)$ is given by the pullback of $h$ along $f$. This adjunction also restricts to the categories of extensions above $E$ and $B$: $e_*|\colon \EXT(E) \to \EXT(B)$, left adjoint of $e^*|\colon\EXT(B) \to \EXT(E)$ (defined similarly). We say that $e$ is of effective (global) descent if $e^*$ is \emph{monadic}, and $e$ is of effective $\EE$-descent if $e^*|$ is \emph{monadic} (see \cite{JanTho1997}). Let us add for the interested reader that, in order to prove the fundamental theorem of categorical Galois theory, G.~Janelidze showed that (in an admissible Galois structure) if $e^*|$ is monadic and we write $T$ for the \emph{monad induced by $e_{\ast}| \dashv e^{*}|$} \cite{McLane1997}, then the category of those extensions which are split by $e$ is equivalent to the category of those \emph{Eilenberg-Moore $T$-algebras} \cite{McLane1997} $(f\colon X \to A, \mu \colon Tf \to f)$ such that $f\colon X \to A$ is a trivial covering (see for instance \cite[Proposition 4.2,Theorem 5.3]{JanComo1991}).

In this section we show that double extensions of racks and quandles are of effective global and $\EE^1$-descent in the category of extensions. From Lemma 3.2 \cite{EvGoeVdl2012} and the above, we have what can be understood as local $\EE^1$-Barr exactness:

\begin{lemma}\label{LemmaBarrKock}  Assuming that $\C$ is Barr-exact, and given a commutative square of extensions together with the horizontal kernel pairs and the factorization $f$ between them;
\begin{equation}\label{DiagramLemmaBarrKock}
\vcenter{\xymatrix @R=15pt @C=30pt {
\Eq(\sigma_{\ttop}) \ar[d]_-{f} \ar@{{}{}{}}[rd]|-{(*)} \ar@<2pt>[r] \ar@<-2pt>[r] & E_{\ttop} \ar[r]^-{\sigma_{\ttop}}   \ar[d]^-{f_E} & B_{\ttop} \ar[d]^-{f_B}\\
\Eq(\sigma_{\pperp}) \ar@<2pt>[r] \ar@<-2pt>[r] & E_{\pperp} \ar[r]_-{\sigma_{\pperp}}  &  \ B_{\pperp}
}}
\end{equation} 
then, the right hand square is a double extension if and only if any of the two left hand (commutative) squares is an extension. 

If so, then $\sigma=(\sigma_{\ttop},\sigma_{\pperp})$ is the coequalizer in $\EXT\C$ of the parallel pair $(*)$ on the left, which is in turn the kernel pair of $\sigma$. Such an equivalence relation $f = \Eq(\sigma)$ in $\EXT\C$ is \emph{stably effective} in the sense that it is the kernel pair of its coequalizer, and any pullback of its coequalizers is still a regular epimorphism (see for instance \cite[Section 2.B]{JanSoTho2004}). In particular, double extensions are the coequalizers of their kernel pairs (computed component-wise in $\C$).
\end{lemma}
\begin{proof}
The first part is a direct consequence of Lemma 3.2 \cite{EvGoeVdl2012}. Since the component-wise coequalizer $\sigma$ of $(*)$ is a pushout square, it coincides with the coequalizer in $\EXT\C$. Then $(*)$ is the kernel pair of $\sigma$ since pullbacks along double extensions are computed component-wise. It is stably effective because everything is computed component-wise, and $\C$ is Barr-exact.
\end{proof}

Note that we also have the classical result, see for instance \cite[Example 6.10]{Bar1971}, which is called the \textit{Barr-Kock} Theorem in \cite[Theorem 2.17]{BouGra2004bis}. From there we easily obtain (as in Remark 4.7 \cite{EvGrVd2008}, or Lemma 3.2 (2) \cite{GraRo2004}):
\begin{lemma}\label{LemmaDoubleExtensionsEffectiveDecent}
Double extensions are of effective (global) descent in $\EXT\C$.
\end{lemma}
\begin{proof}
Let $\sigma\colon f_E \to f_B$ be a double extension. The monadicity of $\sigma^*$ in each component $\sigma_{\ttop}^*$ and $\sigma_{\pperp}^*$ (see \cite{JanTho1994}) easily yields the monadicity of $\sigma^*$ itself. For instance, we use the characterization in terms of \emph{discrete fibrations} \cite[Theorem 3.7]{JanSoTho2004}.

Consider a discrete fibration of equivalence relations $f_R \colon R_{\ttop} \to R_{\pperp}$ above the kernel pair $\Eq(\sigma)\colon\Eq(\sigma_{\ttop}) \to \Eq(\sigma_{\pperp})$ of $\sigma = (\sigma_{\ttop},\sigma_{\pperp})$ as in the commutative diagram of plain arrows below. Then observe that $f_R$ is computed component-wise and consists in a comparison map between a pair of discrete fibrations of equivalence relations $R_{\ttop}$ and $R_{\pperp}$, above the pair of kernel pairs $\Eq(\sigma_{\ttop})$ and $\Eq(\sigma_{\pperp})$ with comparison map $\Eq(\sigma) \colon\Eq(\sigma_{\ttop}) \to \Eq(\sigma_{\pperp})$. The projections of the equivalence relation $f_R$ are also double extensions as the pullback of the projections of $\Eq(\sigma)$ which are themselves double extensions by Lemma \ref{LemmaBarrKock}. We build the square $(*)$ on the right, first by taking the coequalizer $\gamma = (\gamma_{\ttop},\gamma_{\pperp})$ of $f_R$, which is computed component-wise (see \ref{LemmaBarrKock} again). The factorization $(\bar{\beta}_{\ttop}, \bar{\beta}_{\pperp})$ is then obtained by the universal property of $\gamma$. 
\[ \xymatrix @R=15pt @C=30pt {
f_R \pullback \ar[d]_{\hat{\beta}} \ar@<2pt>[r] \ar@<-2pt>[r] & f_C \ar@{{}{}{}}[dr]|-{(*)} \ar@{{}{--}{>}}[r]^{(\gamma_{\ttop},\gamma_{\pperp})}   \ar[d]_{\beta} & f_D \ar@{{}{..}{>}}[d]^{(\bar{\beta}_{\ttop},\bar{\beta}_{\pperp})}\\
\Eq(\sigma) \ar@<2pt>[r] \ar@<-2pt>[r] & f_E \ar[r]_{\sigma}  &  \ f_B
}
\]
By Lemma \ref{LemmaBarrKock}, $f_R$ is the kernel pair of $\gamma$ which is a double extension. Also $(*)$ is a pullback square as it is component-wise by \cite[Theorem 2.17]{BouGra2004bis}. Finally $\gamma$ is pullback stable as a coequalizer, since everything is computed component-wise (see Lemma \ref{LemmaBarrKock}).
\end{proof}
What is exactly needed in our context is not effective global descent but effective $\EE^1$-descent. This derives from Lemma \ref{LemmaDoubleExtensionsEffectiveDecent} because of Corollary \ref{CorollaryPullbacksReflectDExtensions}, as it is explained in \cite[Section 2.7]{JanTho1994}.
\begin{corollary}
Double extensions are of effective $\EE^1$-descent in $\EXT\C$.
\end{corollary}

\begin{remark}\label{RemarkEffectiveDescentRegC}
It was shown in \cite{Ever2012} that given a regular category $\C$, $\EXT(\C)$ is regular if and only if its effective global descent morphisms are the regular epimorphisms (i.e.~the pushout squares of regular epimorphisms). As far as we know, in the categories of racks and quandles, the classes of effective global and $\EE^1$-descent morphisms contain the class of double extensions and are strictly contained in the class of regular epimorphisms. We do not need to characterize these more precisely for what follows.
\end{remark}

\subsection{Strongly Birkhoff Galois structure}\label{SectionAdmissibility}
In order for categorical Galois theory (and in particular its fundamental theorem) to hold in the context of a Galois structure such as $\Gamma$ from Convention \ref{ConventionGaloisStructure}, $\Gamma$ is further required to be \emph{admissible}, in the sense of \cite{JK1994,BorJan1994}, which implies for instance that pullbacks of unit morphisms along primitive extensions are unit morphisms, or subsequently that coverings, normal coverings and trivial coverings are preserved by pullbacks along extensions. We actually work with a stronger property for our Galois structures, which we require to be \emph{strongly Birkhoff} in the sense of \cite[Proposition 2.6]{EvGrVd2008}, where this condition is shown to imply the admissibility condition. The Galois structure $\Gamma$ is said to be \emph{strongly Birkhoff} if \emph{reflection squares} at extensions are double extensions. Given $f\colon{A\to B}$ in $\C$, \emph{the reflection square at} $f$ (with respect to $\Gamma$) is the morphism $(\eta_A, \eta_B)$ with domain $f$ and codomain $\I F(f)$ in $\ARR(\C)$.
\begin{equation}\label{DiagramReflectionSquare} \vcenter{\xymatrix @R=0.8pt @ C=11pt{
A \ar[rrr]^{\eta_A} \ar[rd] |-{p}  \ar[dddd]_{f} & & &  \I F(A)  \ar[dddd]|{\I F(f)} \\
 & B \times_{\I F(B)} \I F(A)  \ar[rru]|-{\pi_2} \ar[lddd]|-{\pi_1}  \\
\\
\\
B  \ar[rrr]_{\eta_B} &  & & \I F(B)
}}
\end{equation}
Our Galois structure $\Gamma^1$ is strongly Birkhoff if the reflection squares at double extensions (as defined in $\C$, for $\C = \RCK$ or $\C = \QND$) should be \emph{double extensions in $\EXT\C$}, which defines the concept of \emph{$3$-fold extension} (see \cite{EvGrVd2008,Ever2010}).

\begin{definition}\label{DefinitionThreeFoldExtension}
Given any regular category $\C$, define \emph{the category $\EXT^2\C$} whose objects are double extensions (as in Diagram \eqref{EquationDoubleExtension}) and whose morphisms $(\sigma,\beta)\colon{\gamma \to \alpha}$ between two double extensions $\gamma$ and $\alpha$ are given by the data of the (oriented) commutative diagram in $\EXT\C$ (on the left) or equivalently in $\C$ (on the right): 
\begin{equation}\label{EquationThreeFoldExtension} \vcenter{\xymatrix @R=8pt @ C=18pt{\db{f_C} \ar[rrr]^{\red{\sigma = (\sigma_{\ttop},\sigma_{\pperp})}} \ar[rd] ^(0.6){\pi = (\pi_{\ttop},\pi_{\pperp})}  \ar[dddd]|-{\gamma = (\gamma_{\ttop},\gamma_{\pperp})} & & &  \db{f_A} \ar[dddd]|{\alpha=(\alpha_{\ttop},\alpha_{\pperp})} \\
 & f_P  \ar[rru]|-{} \ar[lddd]|-{}  \\
\\
\\
\db{f_D}  \ar[rrr]_{\red{\beta = (\beta_{\ttop},\beta_{\pperp})}} &  & & \db{f_B}
}} \qquad \qquad  \vcenter{\xymatrix@1@!0@=32pt{
& C_{\ttop} \ar[rr]^-{\red{\sigma_{\ttop}}} \ar[dd]|(.25){\gamma_{\ttop}}|-{\hole} \ar[ld]_-{\db{f_C}} && A_{\ttop} \ar[dd]^-{\alpha_{\ttop}} \ar[ld]|-{\db{f_A}} \\
C_{\pperp} \ar[rr]^(.75){\red{\sigma_{\pperp}}} \ar[dd]_-{\gamma_{\pperp}} && A_{\pperp} \ar[dd]^(.25){\alpha_{\pperp}} \\
& D_{\ttop} \ar[ld]_-{\db{f_D}} \ar[rr]^(.25){\red{\beta_{\ttop}}}|-{\hole} && B_{\ttop} \ar[ld]^-{\db{f_B}} \\
D_{\pperp} \ar[rr]_-{\red{\beta_{\pperp}}} && B_{\pperp},}}
\end{equation}
where $f_P$ is the pullback of $\alpha$ and $\beta$. A \emph{$3$-fold extension} $(\sigma,\beta)$ in $\C$ is given by such a morphism in $\EXT^2\C$ such that $\sigma$, $\beta$ and the comparison map $\pi$ are also double extensions i.e.~a $3$-fold extension in $\C$ is the data of a double extension in $\EXT\C$. Note that most results from Sections \ref{SectionBasicPropertiesOfDoubleExtensions} and \ref{SectionBeyondBarrExactness} generalize to $3$-fold extensions and \emph{higher} extensions in our context (see Part III).
\end{definition}

Now since all but the ``very top'' component of the \emph{centralization units} in higher dimension (such as $\eta^1$ above) are identities (see Corollary 5.2 \cite{ImKel1986}), we can break down the strong Birkhoff condition in two steps: first the closure by quotients of ($1$-fold) coverings along double extensions, or equivalently, the fact that reflection squares at double extensions are pushout squares in $\EXT\C$ (Birkhoff condition); and second, a \emph{permutability condition}, in the base category $\C$, on the kernel pair of the non-trivial component of the centralization unit $\eta^1$. From Section 
3.4.5 in Part I, we get:
\begin{theorem}\label{TheoremGaloisStructureD2}
The Galois structure $\Gamma^1 \defeq (\EXT\C,\CEXT\C, \Fi, \eta^1, \epsilon ^1, \EE^1)$, where $\C$ is either $\RCK$ or $\QND$, is strongly $\EE^1$-Birkhoff, i.e.~given a double extension of racks or quandles $\alpha = (\alpha_{\ttop},\alpha_{\pperp})\colon{f_A \to f_B}$ (as in Diagram \eqref{EquationDoubleExtension}), the reflection square at $\alpha$ (with respect to the reflection $\Fi \dashv \Ii$) is a $3$-fold extension, i.e.~the reflection square's comparison map is a double extension, and it defines a cube of double extensions in $\EXT\C$.
\end{theorem}
\begin{proof}
Since the bottom component of $\eta^1$ is an isomorphism, it suffices to show that the top component is a double extension for the whole cube to be a $3$-fold extension. This was shown in Corollary 
3.4.8 of Part I.
\end{proof}

In particular this justifies the study of $\Gamma^1$-coverings and the relative \emph{second order centrality} in the categories of racks and quandles.

\begin{remark}\label{RemarkDoubleCentralExtensionsAreClosedUnderQuotients}  A consequence of the strong Birkhoff condition is that if $\gamma$ is a morphism of $\EXT\C$, and $\gamma$ factorizes as $\gamma=\alpha \beta$, where $\alpha$ and $\beta$ are double extensions, then if $\gamma$ is a trivial $\Gamma^1$-covering, by Observation \ref{ObservationCompositeOfDoubleExtensions}, both $\beta$ and $\alpha$ are trivial $\Gamma^1$-coverings. Hence if $\gamma$ is a $\Gamma^1$-covering (see Convention \ref{ConventionGaloisStructure} or Section \ref{SectionTowardsHigherCoveringTheory} below), then both $\alpha$ and $\beta$ are $\Gamma^1$-coverings. From there, and by the fact that $\Gamma^1$-covering are reflected by pullbacks along double extensions (see Convention \ref{ConventionGaloisStructure}), it is easy to conclude that $\Gamma^1$-coverings are closed under quotients along $3$-fold extensions in $\EXT^2\C$.
\end{remark}

\subsection{Towards higher covering theory}\label{SectionTowardsHigherCoveringTheory}
The main aim of this article is to describe \emph{what are the double central extensions of racks and quandles} as in the case of groups \cite{Jan1991}. Following the more general terminology for \emph{coverings}, this consists in characterizing the $\Gamma^1$-coverings of racks and quandles. These are defined abstractly in $\EXT\RCK$ (or $\EXT\QND$) as the double extensions $\alpha\colon{f_A \to f_B}$ for which there exists a double extension $\sigma\colon{f_E \to f_B}$, such that $\sigma$ \emph{splits} $\alpha$, i.e.~the pullback of $\alpha$ along $\sigma$ yields a trivial $\Gamma^1$-covering (see Convention \ref{ConventionGaloisStructure}).

\subsubsection{Projective presentations in dimension 2}
In Part I (Section 
1.3.3) we reminded ourselves that a double extension $\alpha\colon{f_A \to f_B}$ is split by some double extension $\sigma\colon{f_E \to f_B}$ if and only if $\alpha$ can be split by a projective presentation of its codomain $f_B$ -- provided such a projective presentation exists. Hence we want to recall that given any Barr-exact category $\C$, if we choose extensions to be the regular epimorphisms in $\C$, extensions in $\C$ with projective domain and projective codomain are projective objects in $\EXT\C$ (with respect to double extensions -- see for instance \cite[Section 5]{Ever2010}). Note that when $\C$ is a variety of algebras, and $F\colon{\SET\to\C}$ is the left adjoint (with counit $\epsilon$) of the forgetful functor $\U\colon \C \to \SET$, the canonical projective presentation of an object $B$ in $\C$ is given by the counit morphism $\epsilon_B\colon{F(B) \to B}$ (where we omit to write $\U$). Given an extension $f_B\colon{B_{\ttop} \to B_{\pperp}}$ in such a $\C$, we define the \emph{canonical projective presentation of $f_B$} to be the double extension $p_{f_B}\colon{p_B \to f_B}$, below, where $P \defeq F(B_{\pperp})\times_{B_{\pperp}} B_{\ttop}$, is the pullback of $f_B$ and $\epsilon_{B_{\pperp}}$.
\begin{equation}\label{EquationProjectivePresentationD2}\vcenter{\xymatrix @R=1pt @ C=12pt{
F(P) \ar[rrr]^-{p^{\ttop}_{f_B}} \ar[rd]_-{\epsilon_P}  \ar[dddd]_-{p_B} & & &  B_{\ttop}  \ar[dddd]^-{f_B} \\
 & P  \ar[rru]|-{\pi_2} \ar[lddd]|-{\pi_1}  \\
\\
\\
F(B_{\pperp})  \ar[rrr]^-{p^{\pperp}_{f_B} \defeq\, \epsilon_{B_{\pperp}}} &  & & B_{\pperp}
}}
\end{equation}

\subsubsection{Trivial $\Gamma^1$-coverings}\label{ParagraphTrivialDoubleExtensions}
Now we want to be able to identify when the pullback of a double extension $\alpha$ is a trivial $\Gamma^1$-covering in $\EXT\C$ (where $\C$ stands for $\RCK$ or $\QND$). As usual, because the Galois structure $\Gamma^1$ is strongly Birkhoff, trivial $\Gamma^1$-coverings are easy to characterize. Remember that trivial $\Gamma^1$-coverings are those double extensions in $\EXT\C$ which ``behave exactly like'' the primitive double extensions, i.e.~those double extensions in $\CEXT\C$ -- see for instance \cite[Section 1.3]{JK1994} and Example \ref{ExamplePrimitiveDExtensionsAreTrivialDExtensions} below.

From Part I we know that trivial ($1$-fold) coverings of racks (or quandles) are characterized as those extensions that reflect \emph{loops}, which are trails $(x,g)$ whose endpoint $y = x \cdot g$ coincide with the head $x$. Further remember from Paragraph 
3.1.9 of Part I, that given a morphism of racks (or quandles) $f\colon A \to B$, an \emph{$f$-membrane} $M = ((a_0,b_0), ((a_i,b_i),\delta_i)_{1\leq i \leq n})$ is the data of a primitive trail in $\Eq(f)$, whose \emph{length} is the natural number $n$, and whose \emph{endpoints} $a_M$ and $b_M$ are the endpoints of the trails in $\C$ obtained via the projections of $\Eq(f)$. An \emph{$f$-horn} is an $f$-membrane $M = ((a_0,b_0), ((a_i,b_i),\delta_i)_{1\leq i \leq n})$ such that $x \defeq a_0 = b_0$. It is said to close (into a disk) if moreover the endpoints coincide $a_M = b_M$. It is said to \emph{retract} if for each $ 1 \leq k  \leq n$, the truncated horn $M_{\leq k} \defeq (x, (a_i,b_i,\delta_i)_{1\leq i \leq k})$ closes. Finally, the \emph{associated $f$-symmetric pair} of the membrane or horn $M$ is given by the paths $g^M_a \defeq \gr{a_1}^{\delta_1} \cdots  \gr{a_n}^{\delta_n}$ and $g^M_b \defeq \gr{b_1}^{\delta_1} \cdots  \gr{b_n}^{\delta_n}$ in $\Pth(A)$; in general, an \emph{$f$-symmetric path} is a path $g \in \Pth\degree(A)$, such that $g = g^M_a (g^M_b)^{-1}$ for some membrane $M$ as above. These definitions were used in Part I to characterize a general element in the aforementioned centralization congruence $\Ci A$ of some extension $f\colon A \to B$. We showed that $(x,y) \in \Ci A$ if and only if $x \cdot g = y$ for some \emph{$f$-symmetric path} $g$. We repeat this approach for the two-dimensional context in Section \ref{SectionDoubleCoverings}. For now we observe that:
\begin{lemma}\label{LemmaDoubleTrivialExtensions}
If $\alpha\colon{f_A \to f_B}$ is a double extension in $\RCK$ (or in $\QND$), then the following conditions are equivalent:
\begin{enumerate}
\item $\alpha$ is a trivial $\Gamma^1$-covering;
\item any $f_A$-horn which is sent by $\alpha_{\ttop}$ to a $f_B$-disk in $B_{\ttop}$, actually closes into a disk in $A_{\ttop}$;
\[\Bigg( \ \vcenter{\xymatrix@C=2pt @R=5pt {& & & \bullet  \ar@{{}{-}{}}[dddlll]|-{\dir{>}} _-{g_a}    \ar@{{}{-}{}}[dddrrr]|{\dir{>}}^-{ g_b}    \\
& & \ar@<0.1ex>@{{}{-}{}}[rr]|-{f_A} & & & & \\
&  \ar@<0.6ex>@{{}{-}{}}[rrrr]|-{f_A} & &&& & \\
\textcolor{Goldenrod}{\bullet} \ar@<1.6ex>@{{}{-}{}}[rrrrrr]|-{f_A}  & & & & &  & \textcolor{RoyalBlue}{\bullet}
}} \longmapsto \quad  \vcenter{\xymatrix@C=3pt @R=4pt {& & \alpha_{\ttop}(\bullet)  \ar@/_6ex/@{{}{-}{}}[ddd]|-{\dir{>}} _-{\alpha_{\ttop}(g_a)}   \ar@/^6ex/@{{}{-}{}}[ddd]|{\dir{>}}^-{\alpha_{\ttop}(g_b)}   \\
& \ar@<0.7ex>@{{}{-}{}}[rr]|-{f_B} & & &  \\
& \ar@<1.5ex>@{{}{-}{}}[rr]|-{f_B} \ar@<-0.4ex>@{{}{-}{}}[rr]|-{f_B} & & & \\
& & \quad \textcolor{ForestGreen}{\bullet} \quad
}}\ \Bigg)\quad \Longrightarrow \quad \vcenter{\xymatrix@C=3pt @R=4pt {& & \bullet  \ar@/_6ex/@{{}{-}{}}[ddd]|-{\dir{>}} _-{g_a}   \ar@/^6ex/@{{}{-}{}}[ddd]|{\dir{>}}^-{g_b}   \\
& \ar@<0.7ex>@{{}{-}{}}[rr]|-{f_A} & & &  \\
& \ar@<1.5ex>@{{}{-}{}}[rr]|-{f_A} \ar@<-0.4ex>@{{}{-}{}}[rr]|-{f_A} & & & \\
& & \textcolor{Goldenrod}{\bullet} = \textcolor{RoyalBlue}{\bullet} 
}}\]
\item $\alpha_{\ttop}$ reflects $f_A$-symmetric loops, in the sense that if the image by $\alpha_{\ttop}$ of an $f_A$-symmetric trail $(x, g)$ loops in $B_{\ttop}$, then the trail was already a loop in $A_{\ttop}$: $x \cdot g = x$.
\end{enumerate}
In what follows, we prefer to call a double extension $\alpha$ which satisfies these conditions a \emph{trivial double covering}. This terminology will be justified by Theorem \ref{TheoremCharacterizationDoubleCentralExtensions} where we characterize $\Gamma^1$-coverings to be the double coverings from Definition \ref{DefinitionDoubleCoverings} below.
\end{lemma}
\begin{proof}
Using the material from Section \ref{SectionAdmissibility}, we observe that our definition of trivial $\Gamma$-covering from Convention \ref{ConventionGaloisStructure} coincides, for an admissible or strongly Birkhoff Galois structure $Gamma$, with the more common definition: the extension $t$ is a trivial $\Gamma$-covering if and only if the reflection square at $t$ is a pullback. Hence $\alpha$ is a trivial $\Gamma^1$-covering (ore trivial double covering) if and only if the reflection square at $\alpha$ is a pullback. Since pullbacks along double extensions are computed component-wise, and the bottom component is trivial, it suffices to check that the diagram below, where $P \defeq \Fii(A_{\ttop}) \times_{\Fii(B_{\ttop})} B_{\ttop}$,
\[
\vcenter{\xymatrix @R=1pt @ C=11pt{
A_{\ttop} \ar[rrr]^{\alpha_{\ttop}} \ar[rd] |-{p}  \ar[dddd]_-{\eta^1_{A_{\ttop}}}  & & &  B_{\ttop}  \ar[dddd]^-{\eta^1_{B_{\ttop}}} \\
 & P  \ar[rru]|-{\pi_2} \ar[lddd]|-{\pi_1}  \\
\\
\\
\Fii(A_{\ttop})  \ar[rrr]_{\Fii(\alpha_{\ttop})} &  & & \Fii(B_{\ttop}),
}}\]
is a pullback square, i.e.~the comparison map $p$ should be an isomorphism. Since, $(\alpha_{\ttop},\Fii(\alpha_{\ttop}))$ is already a double extension by Corollary 3.4.8 of Part I, it suffices to check that $\Eq(\alpha_{\ttop}) \cap \Ci(f_A) = \Delta_{A_{\ttop}}$ (the \emph{diagonal relation} on $A_{\ttop}$). Now any element $(a,b) \in \Ci(f_A)$ is either such that $a$ and $b$ are the endpoints of a $f_A$-horn, or equivalently, $a$ and $b$ are respectively the head and endpoint of an \emph{$f_A$-symmetric trail} (i.e.~a trail whose path component is $f_A$-symmetric).
\end{proof}


\begin{example}\label{ExamplePullbackSquaresAreTrivialDExtensions}
As a consequence of the fact that $\FIi \dashv \I$ is $\EE^1$-reflective (using Observation \ref{ObservationCompositeOfDoubleExtensions}) (or simply by Lemma \ref{LemmaDoubleTrivialExtensions} above): if the comparison map $p$ of a double extension $\alpha\colon{f_A \to f_B}$ is an isomorphism (i.e.~if $\alpha$ is a pullback square), then both $\alpha$ and $(f_A,f_B)$ are trivial double coverings (i.e. trivial $\Gamma^1$-coverings).
\end{example}

\begin{example}\label{ExamplePrimitiveDExtensionsAreTrivialDExtensions}
Since any primitive double extension (i.e. a double extensions whose domain and codomain are ($1$-fold) coverings) is a trivial double covering and coverings are closed under quotients along double extensions \cite{Ren2020}, if $\alpha\colon{f_A \to f_B}$ is a double extension and $f_A$ is a covering, then $\alpha\colon{f_A \to f_B}$ is a trivial double covering (i.e. a trivial $\Gamma^1$-covering).
\end{example}

Note that the concept of trivial double covering is not symmetric in the role of $(\alpha_{\ttop},\alpha_{\pperp})$ and $(f_A,f_B)$. It is not true that in general $(\alpha_{\ttop},\alpha_{\pperp})$ is a trivial double covering if and only if the double extension $(f_A,f_B)$ is one.

\begin{example}\label{ExampleDTrivialExtension}
Consider the sets $Q_2 = \lbrace\bullet,\ \star \rbrace$, $Q_3 = \lbrace \bullet, \ \star_1,\ \star_0 \rbrace$ and $Q_4 \defeq \lbrace \star_1,\ \star_0,\ \bullet_1,\ \bullet_0 \rbrace$ as well as the morphisms $t_{\star} \colon{Q_3 \to Q_2}$ and $t\colon Q_4 \to Q_2$, which identify the bullets with $\bullet$ and the stars with $\star$. We write $Q_6 \defeq \lbrace  \star_{11},\ \star_{10},\ \star_{01},\ \star_{00}, \ \bullet_1,\ \bullet_0 \rbrace$ for the pullback of $t_{\star}$ and $t$ such that the first projection $\pi_1\colon Q_6 \to Q_4$ identifies $\star_{11}$ with $\star_{10}$ and $\star_{01}$ with $\star_{00}$, and symmetrically for the second projection, which moreover identifies $\bullet_1$ with $\bullet_0$.
\[ \vcenter{\xymatrix @R=0.5pt @ C=4pt{
Q_6 \pullback \ar@{{}{-}{>>}}[rrr]^-{\pi_2} \ar@{{}{-}{>>}}[dddd]_-{\pi_1} & & & Q_{3} \ar@{{}{-}{>>}}[dddd]^-{t_{\star}} \\
 & \\
\\
\\
Q_4 \ar@{{}{-}{>>}}[rrr]_-{t} &  & & Q_2
}} \qquad \Bigg{|} \qquad \vcenter{\xymatrix @C= 20pt@R=15pt {
\star_{11} \ar@{-}[r]|-{\pi_1}  \ar@{..}[d]|-{\pi_2}  & \star_{10} \ar@{..}[d]|-{\pi_2} \\
\star_{01} \ar@{-}[r]|-{\pi_1}  & \star_{00}
}} \qquad \vcenter{\xymatrix @C= 20pt@R=15pt {
\bullet_{1} \ar@{..}[d]|-{\pi_2}  \\ \bullet_0
}} \]
Define the \emph{involutive} ($\qndiop = \qndop$) quandle $Q \defeq \lbrace  \star_{11},\ \star_{10},\ \star_{01},\ \star_{00}, \ \bullet_1, \bullet_1' \ \bullet_0 \rbrace$ such that $\bullet_1 \qndop \star_{11} = \bullet_1 \qndop \star_{01} = \bullet_1'$, $\bullet_1' \qndop \star_{11} = \bullet_1' \qndop \star_{01} = \bullet_1$ and  $x \qndop y =  x$ for any other choice of $x$ and $y$ in $Q$. The function $p\colon Q \to Q_6$ which identifies $\bullet_1'$ with $\bullet_1$ is a surjective morphism of quandles. Then note that the double extension $(\pi_1 p,\ t_{\star})$ is a trivial double covering since $\pi_2 p$ is a covering. However, the double extension $(\pi_2 p,\ t)$ is not a double trivial covering since $\bullet_1 \qndop \star_{11} \neq \bullet_1 \qndop \star_{10}$ even though their images by $\pi_2 p$ coincide. 
\end{example}

Finally, we give an example of trivial double covering which doesn't arise as an instance of Examples \ref{ExamplePullbackSquaresAreTrivialDExtensions} and \ref{ExamplePrimitiveDExtensionsAreTrivialDExtensions}.

\begin{example}\label{ExampleDihedralQuandle1}
As it is explained in \cite[Example 1.14]{Eis2014}, a \emph{dihedral quandle} $D_n$ is the involutive quandle obtained from the (additive) cyclic group $\Z_n \defeq \Z / n\Z = \lbrace 0,\ \ldots,\ n \rbrace$ by $x \qndop y \defeq 2x - y$, for $x$ and $y$ in $D_n$. Note that $D_1$ and $D_2$ are the trivial quandles (i.e.~sets) with one and two elements respectively. For $n>2$, $D_n$ is the subquandle $\Z_n \rtimes \lbrace 1 \rbrace$ of the conjugation quandle $\Z_n \rtimes \Z_2$, corresponding to the $n$ reflections of the regular $n$-gon. In general it injects into the \emph{circular quandle} $\mathbb{S}_1$ defined on the unit circle in $\mathbb{R}_2$ by the ``central symmetries along $\mathbb{S}_1$'': $x \qndop y \defeq 2\langle x,y \rangle y - x$, for each $x$ and $y$ in $\mathbb{S}_1$, such that $- \qndop y$ defines the unique involution which fixes $y$ and sends $x$ to $-x$ whenever $x$ and $y$ are orthogonal (see \cite[Section 3.6]{Eis2014}). 

Now given two natural numbers $n$ and $m$ we have that $D_{nm}$ is the product of $D_n$ and $D_m$. We consider the following double extension of dihedral quandles in $\QND$ where for $j \in \N$, $\bar{0}\colon D_j \to D_{\pperp}$ is the terminal map to $D_{\pperp}$ and for $i \neq 0$ in $\N$, the morphism $\bar{i}\colon D_{ij} \to D_i$ sends $x \in D_{ij}$ to $x \mod i$ in $D_i$:
\begin{equation}\label{DiagramDihedralQuandles}
\vcenter{\xymatrix @R=1pt @ C=11pt{
D_{2nm} \ar[rrr]^{\bar{m}} \ar[rd]_-{\bar{nm}}  \ar[dddd]_-{\bar{n}}  & & &  D_m \ar[dddd]^-{\bar{0}} \\
 &  D_{nm}   \ar[rru]|-{\bar{m}} \ar[lddd]|-{\bar{n}}  \\
\\
\\
D_n  \ar[rrr]_{\bar{0}} &  & & \lbrace 0 \rbrace,
}}
\end{equation}
Note that this double extension is symmetric in the roles of $m$ and $n$. By Lemma \ref{LemmaTrivialDihedral} below, both $(\bar{n},\bar{0})$ and $(\bar{m},\bar{0})$ are trivial double covering whenever $2$, $m$ and $n$ are coprime. If $m=2$ and $n$ are coprime, then $(\bar{n},\bar{0})$ is a trivial double covering but $(\bar{m},\bar{0})$ is not (indeed $0 \qndop 0 = 0 \neq 2n = 0 \qndop n$). See also Example \ref{ExampleTrivialDoubleExtensionsAreDoubleCoverings} below.
\end{example}
\begin{lemma}\label{LemmaTrivialDihedral}
Using the notations from Example \ref{ExampleDihedralQuandle1}, the double extension $(\bar{n},\bar{0})$ is a trivial double covering if and only if $x = 0 \mod n$ whenever $2mx = 0 \mod n$.
\end{lemma}
\begin{proof}
Consider an $\bar{m}$-horn $M$ of length $i>0 \in \N$ which is sent to a loop by $\bar{n}$. Such a horn $M$ is given by the data of $x \in D_{2nm}$ as well as natural numbers $y_j <m$,  $a_j <n$ and $b_j <n$ for each $0 \leq j \leq i$ such that
\begin{equation}\label{EquationModuloN}
x + \sum_{0\leq j \leq i} (-1)^j y_j  + 2m \sum_{0\leq j \leq i} (-1)^j a_j  = x + \sum_{0\leq j \leq i} (-1)^j y_j  + 2m \sum_{0\leq j \leq i} (-1)^j b_j \mod n;
\end{equation} 
and thus also $2m \big(\sum_{0\leq j \leq i} (-1)^j (a_j - b_j) \big) = 0 \mod n$. Now if $\sum_{0\leq j \leq i} (-1)^j (a_j - b_j) = 0 \mod n$, then Equation \ref{EquationModuloN} also holds modulo $2nm$, and the horn $M$ closes in $D_{2nm}$. Conversely if Equation \ref{EquationModuloN} holds modulo $2nm$, we deduce that $\sum_{0\leq j \leq i} (-1)^j (a_j - b_j) = 0 \mod n$.
\end{proof}

\section{Double coverings}\label{SectionDoubleCoverings}

The concepts of covering or relatively the concepts of centrality induced by the Galois theory of racks and quandles are expressed, in each dimension, via the trivial action of certain data. In dimension zero, a rack $A_{\ttop}$ is actually a set if any element $a\in A_{\ttop}$ acts trivially on $A_{\ttop}$. In dimension 1, an extension $f_A \colon{A_{\ttop} \to A_{\pperp}}$ is a covering if given elements $a$ and $b \in A_{\ttop}$, such that $(a,b) \in \Eq(f_A)$ (i.e.~$f_A(a) = f_A(b)$), the action of $\gr{a}\, \gr{b}^{-1}$ is trivial: $x \qndop a \qndiop b = x$ for all $x \in A_{\ttop}$. In dimension 2, we work with double extensions $\alpha = (\alpha_{\ttop},\alpha_{\pperp})\colon{f_A \to f_B}$. The data we are interested in is then given by those $2 \times 2$ matrices with entries in $A_{\ttop}$, whose rows are elements in $\Eq(f_A)$ and whose columns are elements in $\Eq(\alpha_{\ttop})$. \[\text{$0$-dimensional} \colon \ \cdot\  a \quad \qquad \text{$1$-dimensional} \colon \ \xymatrix@C=15pt{a \ar@{{}{-}{}}[r]|-{f_A} & b} \qquad \quad \text{$2$-dimensional}\colon \ \vcenter{\xymatrix @C= 15pt@R=10pt {
a \ar@{-}[r]|-{f_A}  \ar@{..}[d]|-{\alpha_{\ttop}}  & b  \ar@{..}[d]|-{\alpha_{\ttop}} \\
d \ar@{-}[r]|-{f_A}  & c
}} \]
Such $2 \times 2$ matrices characterize the elements of $\Eq(f_A) \square \Eq(\alpha_{\ttop})$, namely the \emph{largest double equivalence relation  above $\Eq(f_A)$ and $\Eq(\alpha_{\ttop})$} \cite{CPP1992, Smth1976, BorBou2004,JanPed2001}, also called \emph{double parallelistic relation} in \cite[Definition 2.1, Proposition 2.1]{Bou2003}. We sometimes write these elements as quadruples $(a,b,c,d) \in \Eq(f_A) \square \Eq(\alpha_{\ttop})$ which encode the entries of the corresponding $2 \times 2$ matrix as above. Their ``trivial action on the elements of $A_{\ttop}$'' is the condition we are interested in. We define \emph{double coverings of racks and quandles} and later show that these coincide with the $\Gamma^1$-coverings.

\begin{definition}\label{DefinitionDoubleCoverings}
A double extension of racks (or quandles) $\alpha\colon{f_A \to f_B}$ (as in Diagram \eqref{EquationDoubleExtension}) is said to be a \emph{double covering} or an \emph{algebraically central double extension} if any of the equivalent conditions $(i)$ - $(iv)$ below are satisfied:
\[ \vcenter{\xymatrix@R=0.5pt{(i)\quad x \qndop a \qndiop b \qndop c \qndiop d = x, \\ 
(ii) \quad x \qndiop a \qndop d \qndiop c \qndop b = x, \\
(iii) \quad x \qndiop a \qndop b \qndiop c \qndop d = x, \\
(iv) \quad x \qndop a \qndiop d \qndop c \qndiop b = x,}} \qquad\text{ for all }\  x\in A_{\ttop} \  \text{ and } \  \vcenter{\xymatrix @C= 15pt@R=10pt {
a \ar@{-}[r]|-{f_A}  \ar@{..}[d]|-{\alpha_{\ttop}}  & b  \ar@{..}[d]|-{\alpha_{\ttop}} \\
d \ar@{-}[r]|-{f_A}  & c
}} \ \in\  \Eq(f_A) \square \Eq(\alpha_{\ttop}).\]
\end{definition}

Note that, by the symmetries of quadruples $(a,b,c,d)$ in $\Eq(f_A) \square \Eq(\alpha_{\ttop})$, one could equivalently use any cyclic permutation of the letters $a$, $b$, $c$, and $d$ in the equalities $(i)$ -- $(iv)$. The equivalence between each of these $(i)$ -- $(iv)$, is shown in Section \ref{SectionThinkingAboutCommutator}. 

\begin{remark}\label{RemarkSymmetricDoubleCoverings} In Definition \ref{DefinitionDoubleCoverings}, the roles of $f_A$ and $\alpha_{\ttop}$ are symmetric. Hence $(\alpha_{\ttop},\alpha_{\pperp})$ is a double covering (or algebraically central) if and only if $(f_A,f_B)$ is a double covering, which can be viewed as a property of the underlying commutative square in $\RCK$ (or $\QND$). Unlike trivial $\Gamma^1$-coverings (also called trivial double coverings), the $\Gamma^1$-coverings are indeed expected to be symmetric in the same sense (see \cite[Section 3]{Ever2010}).
\end{remark}

\begin{example}\label{ExampleTrivialDoubleExtensionsAreDoubleCoverings}
It is easy to show that given a double extension $\alpha\colon{f_A \to f_B}$, if either $\alpha$ or $(f_A,f_B)$ is a trivial double covering, then $\alpha$ is a double covering. Note for instance that given a quadruple $(a,b,c,d) \in \Eq(f_A) \square \Eq(\alpha_{\ttop})$, the $\alpha_{\ttop}$-horn $M$, displayed below, is sent to a disk by $f_A$.
\[M \quad \colon \qquad \vcenter{\xymatrix@C=2pt @R=6pt {& & & x  \ar@{{}{-}{}}[dddlll]|-{\dir{>}} _(0.23){\gr{a}\quad} _(0.47){\gr{b}^{-1}}_(0.63){\gr{c}\quad}_(0.87){\gr{d}^{-1}}  \ar@{{}{-}{}}[dddrrr]|{\dir{>}}^(0.25){\; \gr{d}} ^(0.45){\; \gr{c}^{-1}}^(0.65){\; \gr{c}}^(0.85){\; \gr{d}^{-1}}  \\
& & \ar@<0.6ex>@{{}{-}{}}[rr]|-{\alpha_{\ttop}} \ar@<-1.2ex>@{{}{-}{}}[rr]|-{\alpha_{\ttop}}& & & & \\
&   \ar@<-0.2ex>@{{}{-}{}}[rrrr]|-{\alpha_{\ttop}} & &&& & \\
y\ar@<1.5ex>@{{}{-}{}}[rrrrrr]|-{\alpha_{\ttop}}  & & & & &  & x
}} \qquad y \defeq x\cdot (\gr{a}\, \gr{b}^{-1}\, \gr{c}\, \gr{d}^{-1}) \]
From Example \ref{ExampleDihedralQuandle1}, when $m=2$ and $n$ are coprime, we have that $(\bar{m},\bar{0})$ is not a trivial double covering. However, it still satisfies the conditions of a double covering, which can be deduced from the fact that $(\bar{n},\bar{0})$ is a trivial double covering.
\end{example}

\begin{example}\label{ExampleToyDoubleCovering}
Not all double coverings arise from double trivial coverings. Consider the function $t\colon Q_4 \to Q_2$ from Example \ref{ExampleDTrivialExtension} and its kernel pair $\pi_1,\, \pi_2 \colon Q_8 \rightrightarrows Q_2$ where the elements of $Q_8 = \lbrace  \star_{11},\ \star_{10},\ \star_{01},\ \star_{00}, \ \bullet_{11},\ \bullet_{10},\ \bullet_{01},\ \bullet_{00} \rbrace$ organise as in the Diagram \ref{DiagramToyDoubleExtension} below.  We define the involutive quandle $Q$ with underlying set $Q_8 \cup \lbrace \bullet_{00}' \rbrace$ such that, for $i \in \lbrace 0,\ 1 \rbrace$, $\bullet_{00} \qndop \star_{ii} = \bullet_{00}'$, $\bullet_{00}' \qndop \star_{ii} = \bullet_{00}$ and $x \qndop y = x$ for any other choice of $x$ and $y$ in $Q$. The function $p\colon Q \to Q_8$ defined by $f(\bullet_{00}') = \bullet_{00}$ and $f(x) = x$ for all $x \in Q_8 \subset Q$, is a morphism of quandles such that the double extension below is a double covering.
\begin{equation}\label{DiagramToyDoubleExtension}
\vcenter{\xymatrix @R=0.5pt @ C=10pt{
Q \ar[rrr]^{\pi_2'} \ar[rd] |-{p}  \ar[dddd]_-{\pi_1'}  & & &  Q_4 \ar[dddd]^-{t} \\
 & Q_8  \ar[rru]|-{\pi_2} \ar[lddd]|-{\pi_1}  \\
\\
\\
Q_4  \ar[rrr]_{t} &  & & Q_2,
}} \qquad \Bigg{|} \qquad \vcenter{\xymatrix @C= 20pt@R=15pt {
\star_{11} \ar@{-}[r]|-{\pi_1'}  \ar@{..}[d]|-{\pi_2'}  & \star_{10} \ar@{..}[d]|-{\pi_2'} \\
\star_{01} \ar@{-}[r]|-{\pi_1'}  & \star_{00}
}} \qquad \vcenter{\xymatrix @C= 20pt@R=15pt {
\bullet_{11} \ar@{-}[r]|-{\pi_1'}  \ar@{..}[d]|-{\pi_2'}  & \bullet_{10} \ar@{..}[d]|-{\pi_2'} \\
\bullet_{01} \ar@{-}[r]|-{\pi_1'}  & \bullet_{00}  \ar@{..}@<0.5ex>[r]^-{} \ar@<-0.5ex>@{-}[r]_-{}   & \bullet_{00}'
}} 
\end{equation}
In anticipation of the results of Section \ref{SectionDoubleNormalExtensions}, observe that neither $(\pi_1',t)$ nor $(\pi_2',t)$ are \emph{normal $\Gamma^1$-coverings} since $\bullet_{00} \qndop \star_{00} \neq \bullet_{00} \qndop \star_{01}$ even though $\bullet_{10} \qndop \star_{11} = \bullet_{10} \qndop \star_{10}$; and also $\bullet_{00} \qndop \star_{00} \neq \bullet_{00} \qndop \star_{10}$ even though $\bullet_{01} \qndop \star_{11} = \bullet_{01} \qndop \star_{01}$.
\end{example}


\begin{observation}
Finally we relate our condition (algebraic centrality of double extensions) with the existing concept of \emph{abelian quandle} (or rack) defined in \cite{Joy1979}. If $\alpha\colon f_A \to f_B$ is a double covering of racks (or quandles), then we have that
\begin{equation}\label{EquationAbelianCondition}
(a \qndop d) \qndop (b \qndop c)
= a \qndop a \qndop c \\
= (a \qndop b) \qndop (d \qndop c),
\end{equation}
for each square $\vcenter{\xymatrix @C= 15pt@R=10pt {
a \ar@{-}[r]|-{ }  \ar@{..}[d]|-{ }  & b  \ar@{..}[d]|-{ } \\
d \ar@{-}[r]|-{ }  & c
}}$ in $\Eq(f_A) \square \Eq(\alpha_{\ttop})$ or $\Eq(\alpha_{\ttop}) \square \Eq(f_A)$, and symmetrically in ``each corner'' of this square (i.e.~replace $(a,b,c,d)$ in \eqref{EquationAbelianCondition} by any cyclic permutation of the quadruple). The converse is not true in general.
\end{observation}



\subsection{Thinking about a commutator}\label{SectionThinkingAboutCommutator}
Let $A$ be a rack (or quandle) and $\ER(A)$ be the lattice of (internal) \emph{equivalence relations} (also called \emph{congruences} -- see for instance \cite{Smth1976}), over $A$. We define the following binary operation on $\ER(A)$.

\begin{definition}\label{DefinitionCommutator}
Given a rack $A$ and a pair of congruences $R$ and $S$ in $\ER(A)$, we define $[R,S]$, element of $\ER(A)$, as the congruence generated by the set of pairs of elements of $A$:
\[ \lbrace (x \qndop a \qndiop b \qndop c \qndiop d , x)\qquad\vert \qquad  x\in A \  \text{ and } \  \vcenter{\xymatrix @C= 15pt@R=10pt {
a \ar@{-}[r]|-{R}  \ar@{..}[d]|-{S}  & b  \ar@{..}[d]|-{S} \\
d \ar@{-}[r]|-{R}  & c
}} \ \in\  R\square S \rbrace.\]
\end{definition}
Note that $[R,S]$ is in particular included in the intersection $R \cap S$. Working towards the Corollaries \ref{CorollaryCommutatorGeneralDescription} and \ref{CorollaryEquivalentConditionsInDefinitionOfDoubleCovering} we have that:
\begin{lemma}
Given a rack $A$ and a pair of congruences $R$ and $S$ in $\ER(A)$, then $[R,S]$ is generated by the set of pairs
\[ \lbrace (x \qndiop a \qndop d \qndiop c \qndop b , x)\qquad\vert \qquad  x\in A \  \text{ and } \  \vcenter{\xymatrix @C= 15pt@R=10pt {
a \ar@{-}[r]|-{R}  \ar@{..}[d]|-{S}  & b  \ar@{..}[d]|-{S} \\
d \ar@{-}[r]|-{R}  & c
}} \ \in\  R\square S \rbrace.\]
\end{lemma}
\begin{proof}
By definition, $[R,S]$ includes the pairs $(x \qndiop a \qndop a \qndiop b \qndop c \qndiop d ,x \qndiop a)$ for all $x$, $a$, $b$, $c$ and $d$ as in the statement. Then by compatibility with the rack operations and reflexivity, $[R,S]$ also includes the pairs 
\[(x , x \qndiop a \qndop d \qndiop c \qndop b),\]
for all such $x$, $a$, $b$, $c$ and $d$. By symmetry this then induces that $[R,S]$ includes the congruence relation generated by the set of pairs from the statement. Now a similar argument shows that such a congruence includes the set of pairs defining $[R,S]$ as in Definition \ref{DefinitionCommutator}.
\end{proof}

\begin{corollary}
Given a rack $A$ and a pair of congruences $R$ and $S$ in $\ER(A)$, the congruence $[S,R]$ is generated by the set of pairs
\[ \lbrace (x \qndiop a \qndop b \qndiop c \qndop d, x) \qquad\vert \qquad  x\in A \  \text{ and } \  \vcenter{\xymatrix @C= 15pt@R=10pt {
a \ar@{-}[r]|-{R}  \ar@{..}[d]|-{S}  & b  \ar@{..}[d]|-{S} \\
d \ar@{-}[r]|-{R}  & c
}} \ \in\  R\square S \rbrace.\]
\end{corollary}

\begin{corollary}\label{CorollaryCommutatorGeneralDescription}
Given a rack $A$ then for any congruences $R$ and $S$ in $\ER(A)$, the congruence $[R,S] = [S,R]$ is equivalently generated by any of the set of pairs:
\[ \vcenter{\xymatrix@R=0.5pt{(i)\quad (x \qndop a \qndiop b \qndop c \qndiop d, x), \\ 
(ii) \quad (x \qndiop a \qndop d \qndiop c \qndop b, x), \\
(iii) \quad (x \qndiop a \qndop b \qndiop c \qndop d, x), \\
(iv) \quad (x \qndop a \qndiop d \qndop c \qndiop b, x),}} \qquad\text{ for all }\  x\in A \  \text{ and } \  \vcenter{\xymatrix @C= 15pt@R=10pt {
a \ar@{-}[r]|-{}  \ar@{..}[d]|-{}  & b  \ar@{..}[d]|-{} \\
d \ar@{-}[r]|-{}  & c
}} \ \in\  R \square S. \]
\end{corollary}
\begin{proof}
It suffices to show that $[R,S]$ contains the pairs $(x \qndiop a \qndop b \qndiop c \qndop d, x)$ for any $x \in A$ and $(a,b,c,d) \in R \square S$. Given such data, we compute that
\[ \vcenter{
\xymatrix @C= 15pt@R=8pt {
b  \ar@{-}[r]|-{}  \ar@{..}[d]|-{}  & a  \ar@{..}[d]|-{} \\
c \ar@{-}[r]|-{}  & d
}}\quad  \qndop\quad \vcenter{
\xymatrix @C= 15pt@R=8pt {
b  \ar@{-}[r]|-{}  \ar@{..}[d]|-{}  & b  \ar@{..}[d]|-{} \\
c \ar@{-}[r]|-{}  & c
}} \quad = \quad  \vcenter{
\xymatrix @C= 15pt@R=8pt {
b \qndop b  \ar@{-}[r]|-{}  \ar@{..}[d]|-{}  & a \qndop b  \ar@{..}[d]|-{} \\
c \qndop c \ar@{-}[r]|-{}  & d \qndop c
}} \qquad \in \ R \square S.  \]
Then by definition $[R,S]$ contains the pair $(x \qndop (b \qndop b) \qndiop (a \qndop b) \qndop (d \qndop c) \qndiop (c \qndop c) , x )$, which reduces to $(x \qndop b \qndiop b \qndiop a \qndop b \qndiop c \qndop d \qndop c \qndiop c, x )$. This concludes the proof.
\end{proof}

\begin{corollary}\label{CorollaryEquivalentConditionsInDefinitionOfDoubleCovering}
The conditions $(i)$-$(iv)$ from Definition \ref{DefinitionCommutator} are indeed all equivalent. Moreover, a double extension $\alpha\colon{f_A \to f_B}$ of racks (or quandles) is a double covering (an algebraically central double extension), if and only if $[\Eq(f_A),\Eq(\alpha_{\ttop})]=\Delta_{A_{\ttop}}$ (the \emph{diagonal relation} on $A_{\ttop}$).
\end{corollary}
Based on this result, and in anticipation of Theorems \ref{TheoremCharacterizationDoubleCentralExtensions} and \ref{TheoremEpiReflectivityOfCExt}, we call $[\Eq(f_A),\Eq(\alpha_{\ttop})]$ the \emph{centralization congruence} of the double extension $\alpha\colon{f_A \to f_B}$. Now observe the following:
\begin{lemma}
Given a rack $A$ and a congruence $R$ on $A$, the congruence $[R, A\times A] $ is the congruence generated by the set of pairs $ \lbrace (x \qndop a \qndiop b, x)\ \vert \ x \in A \text{ and } (a,b) \in R \rbrace$.
\end{lemma}
\begin{proof}
Write $S$ for the congruence generated by the set of pairs from the statement. Observe that given $x$, $a$ and $b$ such that $(a,b) \in R$ we have the square 
\[ 
\vcenter{\xymatrix @C= 20pt @R=12pt {
a  \ar@{-}[r]|-{R}  \ar@{..}[d]|-{A\times A}  & b  \ar@{..}[d]|-{A\times A} \\
a \ar@{-}[r]|-{R}  & a.
}}\qquad \in \quad R \square (A\times A).\]
By definition we then have $S \leq [R, A\times A] $. Now observe that for any $(a,b) \in R$ and $(c,d) \in R$:
\[\vcenter{\xymatrix @C= 20pt @R=1pt {
(x \qndop a \qndiop b)\  \ar@{-}[r]|-{S}   &\ x\  \ar@{-}[r]|-{S} &\ (x \qndop d \qndiop c) 
}}\]
are in relation by $S$, and thus $S$ also contains the generators of $[R, A\times A]$.
\end{proof}

\begin{corollary}
Given a morphism $f\colon{A \to B}$ in $\RCK$ (or $\QND$), the congruence $\Ci(f)$ can be computed as $[\Eq(f), A \times A]$, and in particular $f$ is a covering (in the sense of \cite{Eis2014}) if and only if $[\Eq(f), A \times A] = \Delta_A$.
\end{corollary}

Recall that in the category of groups, we have the classical commutator $[-,-]_{\GRP}$, such that a group $G$ is abelian if and only if its \emph{commutator subgroup} is trivial $[G,G]_{\GRP} = \lbrace e \rbrace$ and a surjective homomorphism $f\colon G \to H$ is a central extension if and only if $[\Ker(f),G]_{\GRP} = \lbrace e \rbrace$. Moreover, a double extension of groups $\gamma\colon{f_G \to f_H}$ is a double central extension of groups \cite{Jan1991} if and only if $[\Ker(f_G),\Ker(\gamma_{\ttop})] = \lbrace e \rbrace$ and $[\Ker(f_G) \cap \Ker(\gamma_{\ttop}), G_{\ttop}]= \lbrace e \rbrace$ are both trivial.

For the zero-dimensional case in our context, the corresponding description of \emph{centrality} in terms of the operation $[-,-]$ only works for quandles. Indeed, if $x$ and $a$ are in the quandle $A$, then $x \qndop a = x \qndiop x \qndop a$, which means that $(x\qndop a, x) \in [A \times A, A \times A]$. If $A$ is a rack though, this trick does not work. In particular we compute that $[\Fr 1 \times \Fr 1, \Fr 1 \times \Fr 1] = \Delta_A \neq \Fr 1 \times \Fr 1 = \Co  (\Fr 1)$.

\begin{corollary}
Given a quandle $A$, the congruence $\Co(A)$ can be computed as $[A \times A, A \times A]$, in particular $A$ is a trivial quandle if and only if $[A \times A, A \times A] = \Delta_A$.
\end{corollary}

Note that in the category of groups, two-dimensional centrality is expressed using two requirements. In our context, one of the corresponding requirements entails the other (Corollary \ref{CorollaryTwoConditionsInOne}). First observe that our commutator is \emph{monotone}.

\begin{lemma}\label{LemmaCommutatorCompatibleWithOrder}
Given a rack $A$, as well as congruences $R$, $S$ and $T$ in $\ER(A)$ such that $S \leq T$, then $[R,S] \leq [R,T]$.
\end{lemma}
\begin{proof}
This is a direct consequence of the fact that $R\square S \leq R\square T$.
\end{proof}

\begin{corollary}
If $R$ and $S$ are congruences on $A$ such that $R \leq S$ then $[R,S]=[R,A\times A]$.
\end{corollary}
\begin{proof}
It suffices to show that $[R,S]$ contains $T \defeq \langle (x \qndiop a, x \qndiop b) |aRb \rangle$. As before, observe that for any $aRb$, we have the quadruple $(a,b,a,a) \in R\square S$.
\end{proof}

\begin{corollary}\label{CorollaryTwoConditionsInOne}
If $R$ and $S$ are congruences on $A$ then $[R\cap S, A\times A] = [R\cap S, S] \leq [R, S]$. In particular, the comparison map $p$ of a double covering $\alpha\colon{f_A \to f_B}$ is a covering.
\end{corollary}

Note that the converse of Corollary \ref{CorollaryTwoConditionsInOne} is not true in general. For instance, observe that the double extension from Diagram \eqref{DiagramDihedralQuandles} of Example \ref{ExampleDihedralQuandle1} is such that the comparison map $\bar{mn}\colon D_{2nm} \to D_{nm}$ is always a quandle covering. However when $m=3$ and $n=6$, Diagram \eqref{DiagramDihedralQuandles} is not a double covering since $0 \qndop 0 \qndiop 0 \qndop 0 \qndiop 6 = 12 \neq 0$.

In Section \ref{SectionConjugationQuandles}, where we further investigate the relationship with groups, we shall see that the converse of Corollary \ref{CorollaryTwoConditionsInOne} holds for ``double coverings of conjugation quandles''. More comments and results about our commutator can be found in Section \ref{SectionFurtherDevelopments}.

\subsection{The case of conjugation quandles}\label{SectionConjugationQuandles}
Recall that a \emph{conjugation quandle} is any quandle which is obtained as the image of a group by the functor $\Conj\colon\GRP \to \RCK$. As we reminded ourselves in the Introduction, we use the functors $\Conj$ and its left adjoint $\Pth$ to compare the covering theory of racks and quandles with the theory of central extensions of groups (see \cite{JK1994,Jan1991,Jan2008}). For instance, we mentioned that a surjective group homomorphism is central if and only if its image is a covering in $\RCK$ (or $\QND$ -- \cite[Examples 2.34;1.2]{Eis2014}). However, as the following example shows, the \emph{centralization} (in the sense of $\Fi$) of a morphism between conjugation quandles doesn't coincide with the (image by $\Conj$ of the) centralization (in the sense of $\ab^1$) of a group homomorphism in $\GRP$.
\begin{example}\label{ExampleCentralizationsDontCoincide}
Indeed, consider the quotient map $q \colon S_3 \to S_3 / A_3 = \lbrace -1,1\rbrace$ in $\GRP$, sending the group of permutations of the set of $3$ elements to the (multiplicative) group $\lbrace -1,1\rbrace $ by quotienting $S_3$ by $A_3= \lbrace (),(123),(321)\rbrace$, the alternating subgroup of $S_3$. The morphism $q$ sends $2$-cycles to $-1$. Observe that the (classical group) commutator $[S_3,A_3]_{\GRP} = A_3$. Hence the centralization of $q$ in $\GRP$ is the identity morphism on $\lbrace -1,1\rbrace $. Now observe that $x \qndop x \qndiop y = z$ for any $2$-cycles $x \neq y \neq z$. Hence $2$-cycles are also identified by the centralization of $q$ in $\QND$. However, the action of a $2$-cycle on a $3$-cycle always gives the other $3$-cycle. Hence the successive action of a pair of $2$-cycles on a $3$-cycle does nothing. Similarly since both $3$-cycles are inverse of each-other, $3$-cycles act trivially on each-other. One easily deduces that if $Q_{ab\star}$ is the involutive quandle with $3$ elements whose operation is defined in the table below, then the centralization of the morphism of quandles $q$ is obtained via the quotient $\eta^1_{S_3}\colon S_3 \to (S_3 / \Ci q) = Q_{ab\star}$, such that $\eta^1_{S_3}(123)=a$, $\eta^1_{S_3}(321)=b$ and all other elements of $S_3$ are sent to $\star$. Finally we obtain $\Fi(q)\colon Q_{ab\star} \to S_3 / A_3 = \lbrace -1,1\rbrace $ which takes the values $\Fi(q)[a]=1=\Fi(q)[b]$ and $\Fi(q)[\star]=-1$.
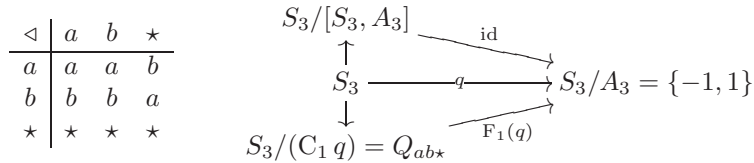
\begin{center}
\begin{tabular}{c | c c c}
$\qndop$    & $a$ & $b$ & $\star$ \\
\hline
$a$ & $a$ & $a$ & $b$ \\
$b$ & $b$ & $b$ & $a$ \\
$\star$ & $\star$ & $\star$ & $\star$ \\
\end{tabular}  $\qquad \vcenter{\xymatrix@R=8pt@C=15pt{
S_3/[S_3,A_3] \ar[drr]^-{\id}  & & \\
S_3 \ar[u]^-{} \ar[d]_-{} \ar[rr]|-{q} && S_3/A_3 = \lbrace -1,1\rbrace\\
S_3/(\Ci q) = Q_{ab\star} \ar[urr]_-{\Fi(q)}  & &
}}$
\end{center}
\end{example}

In this section we further study how our concept of double covering behaves when applied to the image of $\Conj\colon\GRP \to \RCK$. First recall that given a group $G$, and given a path $g = \gr{g_1}^{\delta_1} \cdots \gr{g_n}^{\delta_n} \in \Pth(\Conj(G))$, there is always another path ``of length one'' $\gr{g_0}$, where $g_0 = g_1^{\delta_1} \cdots g_n^{\delta_n} \in G$, such that $x \cdot g = x \cdot {\gr{g_0}}$ for all $x \in \Conj(G)$. A primitive path in $\Conj(G)$ always ``reduces'' (as an inner-automorphism, not as a homotopy class -- see Paragraph 
2.1.8 of Part I) to a one-step primitive path. As a consequence, our notion of double covering simplifies significantly when the quandle operations of interest are derived from the conjugation operation in groups. Note that \emph{connectedness} in \emph{symmetric spaces} also reduces to \emph{strong connectedness} (i.e. connectedness in ``one step'')  -- see \cite[Section 3.7]{Eis2014} and references therein.

\begin{example}
Consider a group $G$ and a pair of surjective group homomorphisms $f$ and $h$ with domain $G$ in $\GRP$. Let us write $R \defeq \Conj(\Eq(f))$ and $S \defeq  \Conj(\Eq(f))$ (note that $\Conj$ preserves limits). In $\QND$ one derives easily that $[R,S] = [R\cap S, G \times G]$ since given a square $(a,b,c,d)\in R\square S$, we have $(d, (ab^{-1}c)) \in  (R\cap S)$ such that moreover $x \qndop (ab^{-1}c) \qndiop d = x \qndop a \qndiop b \qndop c \qndiop d$.
\end{example}

Now observe that the functor $\Pth\colon{\RCK \to \GRP}$ preserves pushouts, and thus the image by $\Pth$ of a double extension $\alpha$ of racks (or quandles), yields a pushout square of extensions in $\GRP$. Since $\GRP$ is a Mal'tsev category, $\Pth(\alpha)$ is a double extension as well (see \cite{CKP1993} and Proposition 5.4 therein). Note however that the comparison map of $\Pth(\alpha)$ in $\GRP$ is not the image of the comparison map of $\alpha$ in $\RCK$ or $\QND$. 

In the other direction, the conjugation functor $\Conj\colon{\GRP \to \RCK}$ preserves pullbacks. Hence it sends a double extension of groups $\gamma\colon{f_G \to f_H}$ to a double extension of quandles $\Conj(\gamma)$, and it sends the comparison map $p$ of $\gamma$ to the comparison map $\Conj(p)$ of $\Conj(\gamma)$.
For a general double extension of racks and quandles $\alpha$, the comparison map of $\alpha$ being a covering is necessary but not sufficient for alpha to be a double covering. However, by the example above and the preceding discussion we have:
\begin{proposition}\label{PropositionImageOfADoubleCentralExtensionOfGroups}
Given a double extension of groups $\gamma\colon{f_G \to f_H}$, its image by $\Conj$ is a double covering of quandles if and only if its comparison map $\Conj(p)$ is a covering in $\QND$ or equivalently if and only if the comparison map $p$ of $\gamma$ is a central extension of groups.
\end{proposition}

In particular, the image by $\Conj$ of a double central extension of groups yields a double covering in $\QND$. However, one cannot deduce that $\gamma$ is a double central extension of groups from the fact that $\Conj(\gamma)$ is a double covering. Finally we show that the image by $\Pth$ of a double covering of quandles is not necessarily a double extension of groups.
\begin{example} Consider a double extension of groups $\gamma\colon{f_G \to f_H}$ such that $k_1k_2^{-1} \neq k_2^{-1} k_1$ for some $k_1 \in \Ker(f_G)$ and $k_2 \in \Ker(\gamma_{\ttop})$, but $ka = ak$ for all $k \in \Ker(f_G)\cap \Ker(\gamma_{\ttop})$ and $a \in G_{\ttop}$. 

For instance, define $\gamma_{\pperp}\colon G_{\pperp} \to H_{\pperp}$ as the surjective group homomorphism $\Fg(\lbrace a,c\rbrace) \to \Fg(\lbrace c \rbrace)$ such that $\gamma_{\pperp}(c) = c$ and $\gamma_{\pperp}(a) = e$. Similarly define $f_H\colon H_{\ttop} \to H_{\pperp}$ as $\Fg(\lbrace b,c \rbrace) \to \Fg(\lbrace c \rbrace)$ such that $f_H(c) = c$ and $f_H(b) = e$. Write $P \defeq G_{\pperp} \times_{H_{\pperp}} H_{\ttop}  = \Fg(\lbrace a, b, c \rbrace)$ for their pullback, with projections $\pi_1\colon P \to G_{\pperp}$ and $\pi_2\colon P \to H_{\ttop}$, and take the \emph{canonical projective presentation} $\epsilon^g_{P} \colon \Fg(P) \to P$, obtained from the counit $\epsilon^g$ of free-forgetful adjunction $\Fg \dashv \U$. Compute its centralization $\ab^1(\epsilon^g_{P})\colon \Fg(P)/[\Ker(\epsilon^g_{P}),\Fg(P)]_{\GRP} \to P$, and define $f_G \defeq \pi_1 \ab^1(\epsilon^g_{P})$. Similarly define $\gamma_{\ttop} \defeq \pi_2 \ab^1(\epsilon^g_{P})$. The resulting double extension of groups is as required.

By Proposition \ref{PropositionImageOfADoubleCentralExtensionOfGroups}, the double extension of quandles $\Conj(\gamma)$ is a double covering. However we show that the double extension $\Pth(\Conj(\gamma))$ cannot be a double central extension of groups. First observe that the unit $\pth_{\Conj(G_{\ttop})}\colon \Conj(G_{\ttop}) \to \Conj(\Pth(\Conj(G_{\ttop})))$ is a monomorphism, since the identity morphism on $\Conj(G_{\ttop})$ factors through it. Now, if $e_{\ttop}$ is the neutral element in $G_{\ttop}$, then $\gr{k_1}\, \gr{e_{\ttop}}^{-1} \in \Ker(\vec{f_G})$ and $\gr{e_{\ttop}} \, \gr{k_2}^{-1} \in \Ker(\vec{\alpha_{\ttop}})$. Suppose by contradiction that $\gr{k_1}\, \gr{e_{\ttop}}^{-1} \gr{e_{\ttop}}\, \gr{k_2}^{-1} = \gr{e_{\ttop}}\, \gr{k_2}^{-1}\, \gr{k_1}\, \gr{e_{\ttop}}^{-1}$. We have that $\gr{k_1}\, \gr{k_2}^{-1} = \gr{k_1}\, \gr{e_{\ttop}}^{-1} \gr{e_{\ttop}}\, \gr{k_2}^{-1}$ and, by the compatibility of $\pth_{\Conj(G_{\ttop})}$ with $\qndop$, we have, moreover:
\[ \gr{k_2}^{-1}\, \gr{k_1} = (\gr{k_2 \qndop e_{\ttop}})^{-1}\, \gr{k_1 \qndop e_{\ttop}} = (\gr{k_2}\, \qndop\, \gr{e_{\ttop}})^{-1}\, (\gr{k_1}\, \qndop \, \gr{e_{\ttop}}) = \gr{e_{\ttop}}\, \gr{k_2}^{-1}\, \gr{e_{\ttop}}^{-1} \, \gr{e_{\ttop}} \, \gr{k_1}\, \gr{e_{\ttop}}^{-1} = \gr{e_{\ttop}}\, \gr{k_2}^{-1}\, \gr{k_1}\, \gr{e_{\ttop}}^{-1}.\] 
Hence we must also have that  $\gr{k_1}\, \gr{k_2}^{-1} = \gr{k_2}^{-1}\, \gr{k_1}$ and thus $\gr{k_1} \qndiop \gr{k_2} = \gr{k_1}$, which implies $\gr{k_1 \qndiop k_2} = \gr{k_1}$. Since $\pth_{\Conj(G_{\ttop})}$ is injective, we must have $ k_2 k_1 k_2^{-1} = k_1 \qndiop k_2 = k_1 \in \Conj(G_{\ttop})$ which is in contradiction with the hypothesis $k_1 k_2^{-1} \neq k_2^{-1} k_1$. Hence it must also be that $\gr{k_1}\, \gr{e_{\ttop}}^{-1} \gr{e_{\ttop}}\, \gr{k_2}^{-1} \neq \gr{e_{\ttop}}\, \gr{k_2}^{-1}\, \gr{k_1}\, \gr{e_{\ttop}}^{-1}$ and $\Pth(\gamma)$ cannot be a double central extension of groups.
\end{example}

\begin{remark} By anticipation of Theorems \ref{TheoremCharacterizationDoubleCentralExtensions} and \ref{TheoremEpiReflectivityOfCExt}, we cannot hope for a direct three-dimensional version of the Diagrams \ref{DiagramSquareAdjD1} and \ref{DiagramSquareAdmAdjD2} in which the bottom adjunction's left adjoint would be the centralization of double extensions of groups.
\end{remark}

Now in order to further study double coverings (algebraically central double extensions) for general racks and quandles, and their relation to $\Gamma^1$-coverings, we need a characterization for general elements in the \emph{centralization congruence} $[\Eq(f_A),\Eq(\alpha_{\ttop})]$ of a double extension $\alpha\colon f_A \to f_B$. Think about the transitive closure of the set of pairs from Definition \ref{DefinitionCommutator}. In order to identify these general pairs, we make a detour via the generalized notion of primitive trail and the characterization of normal $\Gamma^1$-coverings.

\subsection{A concept of \emph{primitive trail} in each dimension: from \emph{membranes} to \emph{volumes}.}\label{SectionVolumes} Similarly to what was studied in dimension zero and one, we shall further be interested in the ``action of sequences of two-dimensional data''. Given a rack $A$ in dimension zero, we have the fundamental concept of a primitive path, which is merely a sequence of elements in $A\times\lbrace -1,\, 1\rbrace$, viewed as a formal sequence of \emph{symmetries} (Part I, Paragraph 
2.3.3). Given a rack $A$, its \emph{centralization} (or set of connected components) is obtained by identifying elements which are ``connected by the action of a primitive path in $A$''. In dimension one, the centralization of an extension $f\colon{A \to B}$ is in some sense obtained by the study of elements which are ``linked by the action on $A$ of a primitive path from $\Eq(f)$'', leading to the concept of a membrane (see Paragraph \ref{ParagraphTrivialDoubleExtensions} or Part I). Now given a double extension of racks $\alpha$, we shall be interested in the action on $A_{\ttop}$ of primitive paths from $\Eq(f_A) \square \Eq(\alpha_{\ttop})$. We exhibit the 2-dimensional generalizations of the lower-dimensional concepts of \emph{primitive trail}, \emph{membrane}, and \emph{horn}.

\begin{definition}\label{DefinitionFHMembrane}
Given a pair of morphisms $f\colon {A \to B}$ and $h\colon {A \to C}$ in $\RCK$ (or $\QND$), we define an $\langle f,h\rangle$-volume as the data $V = ((a_0,b_0,c_0,d_0), ((a_i,b_i,c_i,d_i),\delta_i)_{1\leq i \leq n})$ of a primitive trail in $\Eq(f) \square \Eq(h)$. The first quadruple $(a_0,b_0,c_0,d_0)$ is the \emph{head of $V$}. We call such an $\langle f,h\rangle$-volume $V$ an $\langle f,h\rangle$-horn if the head reduces to a point: $a_0=b_0=c_0=d_0 \eqdef x$ which we specify as $V = (x, ((a_i,b_i,c_i,d_i), \delta_i)_{1 \leq i \leq n})$. Let us define $a \defeq (a_i)_{1\leq i \leq n}$ and similarly define $b$, $c$ and $d$.  The \emph{associated $\langle f,h\rangle$-symmetric quadruple} of the volume or horn $V$ is given by the paths 
$ g^V_a \defeq \gr{a_1}^{\delta_1} \cdots  \gr{a_n}^{\delta_n}, \quad g^V_b \defeq \gr{b_1}^{\delta_1} \cdots  \gr{b_n}^{\delta_n}, \quad g^V_c \defeq \gr{c_1}^{\delta_1} \cdots  \gr{c_n}^{\delta_n}\  \text{ and }\  g^V_d \defeq \gr{d_1}^{\delta_1} \cdots  \gr{d_n}^{\delta_n}$ in $\Pth(A)$.
The \emph{endpoints} of the volume or horn are given by $a_V = a_0 \cdot g^V_a$, $b_V = b_0 \cdot g^V_b$, $c_V = c_0 \cdot g^V_c$ and $d_V = d_0 \cdot g^V_d$. Finally we call $(a,b)$-membrane the $f$-membrane defined by $M^V_{(a,b)} \defeq ((a_0,b_0),((a_i,b_i),\delta_i)_{1\leq i \leq n})$. The other $f$-membrane, labelled $(c,d)$, and the two $h$-membranes, labelled by $(a,d)$ and $(b,c)$, are defined similarly.
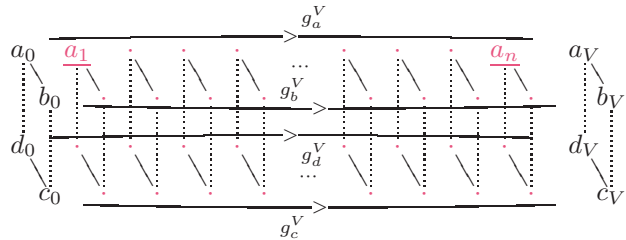
\[\xymatrix@1@!0@C=10pt@R=18pt{
a_0 \ar@{..}[dd]^-{} \ar@{-}[rd]_-{} &\ \ar@<+1ex>@/^1pt/@{{ }{-}{ }}[rrrrrrrrrrrrrrrrrr]^(.55){g^V_a}|-{>}  & \gr{a_1} \ar@{..}[dd] \ar@{-}[rd]_-{} & & \grl{\cdot}\ar@{..}[dd] \ar@{-}[rd]_-{} & &\grl{\cdot}\ar@{..}[dd] \ar@{-}[rd]_-{} & &\grl{\cdot}\ar@{..}[dd] \ar@{-}[rd]_-{} & & \  \ar@{{}{}{}}[rd]|(.35){\cdots} & & \grl{\cdot}\ar@{..}[dd] \ar@{-}[rd]_-{} & &\grl{\cdot}\ar@{..}[dd] \ar@{-}[rd]_-{} & &\grl{\cdot}\ar@{..}[dd] \ar@{-}[rd]_-{} & & \gr{a_n}\ar@{..}[dd] \ar@{-}[rd]_-{} & & & a_V \ar@{..}[dd]^(.25){} \ar@{-}[rd]_-{} \\
& b_0 \ar@{..}[dd]^-{} &\ \ar@{{ }{-}{ }}@<-0.7ex>@/_1pt/[rrrrrrrrrrrrrrrrrr]^(.45){g^V_b}|-{>} & \grl{\cdot}\ar@{..}[dd]^(.25){} & & \grl{\cdot}\ar@{..}[dd]^(.25){} & & \grl{\cdot}\ar@{..}[dd]^(.25){} & & \grl{\cdot}\ar@{..}[dd]^(.25){} & & \  & & \grl{\cdot}\ar@{..}[dd]^(.25){} & &\grl{\cdot}\ar@{..}[dd]^(.25){} & &\grl{\cdot}\ar@{..}[dd]^(.25){} & & \grl{\cdot}\ar@{..}[dd]^(.25){}  & & & b_V \ar@{..}[dd]^(.25){} \\
d_0 \ar@{-}[rd]_-{} & \ \ar@{{ }{-}{ }}@<+0.7ex>@/^1pt/[rrrrrrrrrrrrrrrrrr]_(.55){g^V_d}|-{>} & \grl{\cdot}\ar@{-}[rd]_-{} & & \grl{\cdot}\ar@{-}[rd]_-{} & &\grl{\cdot}\ar@{-}[rd]_-{} & &\grl{\cdot}\ar@{-}[rd]_-{} & & \ \ar@{{}{}{}}[rd]|(.65){\cdots} & & \grl{\cdot}\ar@{-}[rd]_-{} & &\grl{\cdot}\ar@{-}[rd]_-{} & &\grl{\cdot}\ar@{-}[rd]_-{} & & \grl{\cdot}\ar@{-}[rd]_-{} & & & d_V \ar@{-}[rd]_-{} \\
& c_0  & \ \ar@{{ }{-}{ }}@<-1ex>@/_1pt/[rrrrrrrrrrrrrrrrrr]_(.45){g^V_c}|-{>} & \grl{\cdot} & & \grl{\cdot} & &\grl{\cdot} & &\grl{\cdot} & &\  & &\grl{\cdot} & &\grl{\cdot} & & \grl{\cdot} & & \grl{\cdot} & & & c_V
}\]
\end{definition}

Note that because double parallelistic relations are symmetric, both in the ``vertical'' and in the ``horizontal'' direction, Definition \ref{DefinitionFHMembrane} is ``symmetric'' in the role of opposite membranes. 

\begin{remark}
A morphism of racks $f\colon{A \to B}$ obviously sends a primitive trail $(x, (a_i,\delta_i)_{1\leq i \leq n})$ in $A$ to a primitive trail in $(f(x), (f(a_i),\delta_i)_{1\leq i \leq n})$ in $B$. Similarly, a morphism $\alpha\colon f_A \to f_B$ in $\EXT\RCK$ sends a $f_A$-membrane to a $f_B$-membrane, and a morphism of $\EXT^2\C$, such as $(\sigma,\beta)\colon{\gamma \to \alpha}$ in Definition \ref{DefinitionThreeFoldExtension} sends a $\langle f_C,\gamma_{\ttop} \rangle$-volume to a $\langle f_A,\alpha_{\ttop}\rangle$-volume (via the induced morphism $\square_{(\sigma,\beta)}$ such as in Lemma \ref{LemmaSurjectionBetweenParallelisticDERel}).
\end{remark}

\begin{notation}\label{NotationDoubleExtAndKernelPair} Given a double extension of racks (or quandles) $\alpha \colon{f_A \to f_B}$, we can build its kernel pair in $\EXT\RCK$ component-wise, which we denote:
\[\xymatrix@1@!0@=30pt{
& \Eq(\alpha_{\ttop}) \pullback \ar[rr]^-{\pi_{2}} \ar[dd]|(.25){\pi_{1}}|-{\hole} \ar[ld]_-{\bar{f}} && A_{\ttop} \ar[dd]^-{\alpha_{\ttop}} \ar[ld]|-{f_A} \\
\Eq(\alpha_{\pperp}) \pullback \ar[rr]^(.75){p_{2}} \ar[dd]_-{p_{1}} && A_{\pperp} \ar[dd]^(.25){\alpha_{\pperp}} \\
& A_{\ttop} \ar[ld]_-{f_A} \ar[rr]^(.25){\alpha_{\ttop}}|-{\hole} && B_{\ttop} \ar[ld]^-{f_B} \\
A_{\pperp} \ar[rr]_-{\alpha_{\pperp}} && B_{\pperp}}\]
\end{notation}

\begin{remark}\label{RemarkVolumesAsMembranes}
Using Notation \ref{NotationDoubleExtAndKernelPair}, the $\langle f_A,\alpha_{\ttop}\rangle$-volumes $V$ (see Definition \ref{DefinitionFHMembrane}) correspond bijectively to the $\bar{f}$-membranes $M$ in $\Eq(\alpha_{\ttop})$, since such an $M$ is defined as the data $(((a_0,d_0),(b_0,c_0)), (((a_i,d_i),(b_i,c_i)),\delta_i)_{1\leq i \leq n})$ for a certain sequence of elements $(a_i,b_i,c_i,d_i)$ in $\Eq(f_A) \square \Eq(\alpha_{\ttop})$, where $0 \leq i \leq n$. Under the appropriate bijective correspondence, the $(a,b)$-membrane (and $(c,d)$-membrane) of a $\langle f_A,\alpha_{\ttop}\rangle$-volume are obtained from the corresponding $\bar{f}$-membrane via the projections $\pi_1$ and $\pi_2$ respectively. A $\bar{f}$-horn then corresponds to a $\langle f_A,\alpha_{\ttop}\rangle$-volume whose head $(a_0,b_0,c_0,d_0)$ is such that $a_0=b_0$ and $c_0=d_0$.

Similarly, $\langle f_A,\alpha_{\ttop}\rangle$-volumes correspond bijectively to $\bar{\alpha}$-membranes in $\Eq(f_A)$, where $\bar{\alpha}$ is the kernel pair of $(f_A,f_B)$ in $\EXT\RCK$. A $\bar{\alpha}$-horn then corresponds bijectively to a $\langle f_A,\alpha_{\ttop}\rangle$-volume whose head $(a_0,b_0,c_0,d_0)$ is such that $a_0=d_0$ and $b_0=c_0$.
\end{remark}

\subsection{Normal $\Gamma^1$-coverings and rigid horns}\label{SectionDoubleNormalExtensions} We illustrate these definitions in the characterization of normal $\Gamma^1$-coverings, which we subsequently refer to as \emph{normal double coverings}.

\begin{proposition}\label{PropositionDNormalExtensions}
Given a double extension of racks (or quandles) $\alpha \colon{f_A \to f_B}$, it is a double $\Gamma^1$-covering if and only if, given a $\langle f_A,\alpha_{\ttop}\rangle$-volume $V = ((a_0,b_0,c_0,d_0), ((a_i,b_i,c_i,d_i),\delta_i)_{1\leq i \leq n})$ (as in Definition \ref{DefinitionFHMembrane}), if its $f_A$-membranes are horns (i.e. $a_0=b_0$ and $c_0=d_0$) then \emph{the $\alpha_{\ttop}$-membranes of $V$ are rigid} in the sense that its $(d,c)$-horn closes if and only if its $(a,b)$-horn closes. We call a double extension satisfying this condition a \emph{normal double covering}.
\end{proposition}
Observe that by the ``symmetries'' of $\langle f_A,\alpha_{\ttop}\rangle$-volumes (in the role of $f_A$-membranes), it suffices to show that in any such volume $V$, a closing $(a,b)$-horn implies a closing $(c,d)$-horn in order to deduce that in any such volume $V$, a closing $(c,d)$-horn implies a closing $(a,b)$-horn (and conversely the latter implies the former). We relate this to the fact that $(\pi_1,p_1)$ (from Notation \ref{NotationDoubleExtAndKernelPair}) is a trivial double covering if and only if $(\pi_2,p_2)$ is one.
\begin{proof}[Proof of Proposition \ref{PropositionDNormalExtensions}]
By definition, $\alpha$ is a normal $\Gamma^1$-covering if and only if, in Notation \ref{NotationDoubleExtAndKernelPair}, the left face $(\pi_1,p_1)$ (or equivalently the top face $(\pi_2,p_2)$) is a trivial double covering. Then by Lemma \ref{LemmaDoubleTrivialExtensions}, the double extension $(\pi_1,p_1)$ is a trivial double covering if and only if given any $\bar{f}$-horn $M$, such that $\pi_1(M)$ closes in $A_{\ttop}$, then $M$ closes in $\Eq(\alpha_{\ttop})$, i.e. $\pi_2(M)$ also has to close. By Remark \ref{RemarkVolumesAsMembranes} the preceding translates into the statement: $(\pi_1,p_1)$ is a trivial double covering if and only if given any $\langle f_A,\alpha_{\ttop}\rangle$-volume $V$ such that $a_0=b_0$ and $c_0=d_0$, if the $(a,b)$-horn of $V$ closes then the $(c,d)$-horn of $V$ has to close. Similarly $(\pi_2,p_2)$ is a trivial double covering if and only if given any volume $V$ such that $a_0=b_0$ and $c_0=d_0$, a closing $(c,d)$-horn implies a closing $(a,b)$-horn.
\end{proof}

Of course trivial $\Gamma^1$-coverings (i.e.~trivial double coverings) are examples of normal $\Gamma^1$-coverings (i.e.~normal double coverings). However, these two concepts do not coincide.

\begin{example}\label{ExampleOfSymmetricDNormalExtension}
Consider the set $\lbrace \star,\ \bullet \rbrace$, seen as a trivial quandle, as well as two copies $f \colon{Q^{\diamond} \to \lbrace \star,\ \bullet \rbrace}$ and $f \colon{Q_{\diamond} \to \lbrace \star,\ \bullet \rbrace}$ of the same morphism where $Q^{\diamond} \defeq \lbrace \star^{\diamond}, \star, \ \bullet^1, \bullet^0 \rbrace$, and $Q_{\diamond} \defeq \lbrace \star, \star_{\diamond}, \ \bullet_1, \bullet_0 \rbrace$ are such that $\star_{\diamond}$ (respectively $\star^{\diamond}$)  acts on $\bullet_1$ and $\bullet_0$ (respectively $\bullet^1$ and $\bullet^0$) by interchanging $1$ and $0$, and all the other actions are trivial (see also Example 2.3.14 in Part I). We then denote the kernel pair of $f$ by $Q^{\diamond}_{\diamond}$ with underlying set $\lbrace \star^{\diamond},\ \star,\ \star^{\diamond}_{\diamond},\ \star_{\diamond}, \bullet^1_1,\ \bullet^1_0,\ \bullet^0_1,\ \bullet^0_0 \rbrace$, such that the element $\star^{\diamond}$ acts on bullets by interchanging the exponents $1$ and $0$ and similarly with $\star_{\diamond}$ for the indices. Then $\star^{\diamond}_{\diamond}$ interchanges both indices and exponents of the bullets, whereas $x \qndop y = x$ for any other choice of $x$ and $y$ in $Q^{\diamond}_{\diamond}$.
\begin{equation}\label{EquationExampleOfSymmetricDNormalExtension} \vcenter{\xymatrix @R=0.5pt @ C=4pt{
Q^{\diamond}_{\diamond} \pullback \ar@{{}{-}{>>}}[rrr]^-{\pi_{\diamond}} \ar@{{}{-}{>>}}[dddd]_-{\pi^{\diamond} } & & & Q_{\diamond}  \ar@{{}{-}{>>}}[dddd]^-{f_{\diamond} } \\
 & \\
\\
\\
Q^{\diamond} \ar@{{}{-}{>>}}[rrr]_-{f^{\diamond}} &  & & \lbrace \star,\ \bullet \rbrace
}} \qquad \Bigg{|} \qquad \vcenter{\xymatrix @C= 20pt@R=15pt {
\bullet^1_1 \ar@{-}[r]|-{\pi^{\diamond}}  \ar@{..}[d]|-{\pi_{\diamond}}  & \bullet^1_0  \ar@{..}[d]|-{\pi_{\diamond}} \\
\bullet^0_1 \ar@{-}[r]|-{\pi^{\diamond}}  & \bullet^0_0
}} \qquad \vcenter{\xymatrix @C= 20pt@R=15pt {
\star^{\diamond}_{\diamond} \ar@{-}[r]|-{\pi^{\diamond}}  \ar@{..}[d]|-{\pi_{\diamond}}  & \star^{\diamond}  \ar@{..}[d]|-{\pi_{\diamond}} \\
\star_{\diamond} \ar@{-}[r]|-{\pi^{\diamond}}  & \star
}}
\end{equation}
The projection $\pi_{\diamond}$ identifies all the elements that have the same indices (including blanks), and similarly $\pi^{\diamond}$ identifies elements with the same exponents.

Observe that none of the morphisms above are quandle coverings. Moreover, both double extensions $(\pi_{\diamond}, f^{\diamond})$ and $(\pi^{\diamond}, f_{\diamond})$ are such that the conditions of Lemma \ref{LemmaDoubleTrivialExtensions} are not satisfied. However, the conditions of Proposition \ref{PropositionDNormalExtensions} are easily seen to be satisfied by both $(\pi_{\diamond}, f^{\diamond})$ and $(\pi^{\diamond}, f_{\diamond})$. In order to check this, observe that the only ``non-trivial'' element in $\Eq(\pi_{\diamond}) \square \Eq(\pi^{\diamond})$ is the square on the right of \eqref{EquationExampleOfSymmetricDNormalExtension} (or any symmetric equivalent) and for any $g$, $h\in \Pth(Q_{\diamond}^{\diamond})$ and for any $i$, $j$, $k$, $l \in \lbrace 0,\ 1\rbrace$, we have that $\bullet^i_j \cdot g = \bullet^i_j \cdot h$ if and only if $\bullet^k_l \cdot g  = \bullet^k_l \cdot h$.
\end{example}

Even if \ref{ExampleOfSymmetricDNormalExtension} is symmetric in the sense that both $(\pi_{\diamond}, f^{\diamond})$ and $(\pi^{\diamond}, f_{\diamond})$ are double normal coverings, Proposition \ref{PropositionDNormalExtensions}, does not seem to be symmetric in the role of $(\alpha_{\ttop},\alpha_{\pperp})$ and $(f_A,f_B)$. Observe that in Example \ref{ExampleDTrivialExtension}, the double extension $(\pi_1 p, t_{\star})$ is a trivial double covering and thus also a normal double covering. However, the double extension $(\pi_2 p, t)$ is neither a trivial double covering nor a normal double covering since $\bullet_1 \qndop \star_{11} \neq \bullet_1 \qndop \star_{10}$ even though $\bullet_0 \qndop \star_{01} = \bullet_0 \qndop \star_{00}$.


Recall that any normal $\Gamma^1$-covering (normal double covering) $\alpha$ is in particular a $\Gamma^1$-covering, since $\alpha$ is split by $\alpha$. Now unlike trivial double coverings and normal double coverings, $\Gamma^1$-coverings are expected to be symmetric in the same way that double coverings are (see Remark \ref{RemarkSymmetricDoubleCoverings}). If we were to weaken the condition characterizing normal $\Gamma^1$-coverings to obtain a candidate condition for the characterization of $\Gamma^1$-coverings, we would look for a way to make it symmetric in the roles of $f_A$ and $\alpha_{\ttop}$.

Now observe that an obvious asymmetrical feature of the characterization in Proposition \ref{PropositionDNormalExtensions} is the fact that we look at properties of $\bar{f}$-horns in $\Eq(\alpha_{\ttop})$, some of which cannot be expressed as $\bar{\alpha}$-horns in $\Eq(f_A)$. In the spirit of the discussions at page \pageref{SectionDoubleCoverings} and \pageref{SectionVolumes}, we are looking at the ``successive action'' of ``two-dimensional data'' on some ``one-dimensional data'' (in a fixed privileged direction). What we are aiming for is the ``successive action'' of ``two-dimensional data'' on some ``zero-dimensional data''.

We get rid of the asymmetry in Proposition \ref{PropositionDNormalExtensions} by collapsing the one-dimensional head of the volumes we study. Looking at $\langle f_A,\alpha_{\ttop}\rangle$-horns in $A_{\ttop}$, these can be described both as $\bar{f}$-horns in $\Eq(\alpha_{\ttop})$ and as $\bar{\alpha}$-horns in $\Eq(f_A)$. From Proposition \ref{PropositionDNormalExtensions} we produce the concept of a double extension with \emph{rigid horns}.

\begin{definition}\label{DefinitionRigidHorns}
A double extension of racks (or quandle) $\alpha$ is said to \emph{have rigid horns} if any $\langle f_A,\alpha_{\ttop}\rangle$-horn $V$ in $A_{\ttop}$ has rigid $\alpha_{\ttop}$-membranes in the sense of Proposition \ref{PropositionDNormalExtensions}: if $V = ((a_0,b_0,c_0,d_0), ((a_i,b_i,c_i,d_i),\delta_i)_{1\leq i \leq n})$, as in Definition \ref{DefinitionFHMembrane}, its $(d,c)$-horn closes if and only if its $(a,b)$-horn closes.
\end{definition}

Even though Definition \ref{DefinitionRigidHorns} still seems asymmetric at first, it is actually not so anymore. Indeed we use the terminology \emph{rigid horns} because we may show that given a double extension $\alpha\colon{f_A \to f_B}$, any $\langle f_A,\alpha_{\ttop}\rangle$-horn $V$ in $A_{\ttop}$ has rigid $\alpha_{\ttop}$-membranes if and only if any $\langle f_A,\alpha_{\ttop}\rangle$-horn $V$ in $A_{\ttop}$ \emph{has rigid $f_A$-membranes} (its $(a,d)$-horn closes if and only if its $(b,c)$-horn closes). Observe that by Definition \ref{DefinitionRigidHorns}, the double extension $(f_A,f_B)$ has rigid horns if and only if any $\langle f_A,\alpha_{\ttop}\rangle$-horn has rigid $f_A$-membranes. Again we may show that $(\alpha_{\ttop},\alpha_{\pperp})$ has rigid horns (in the sense of Definition \ref{DefinitionRigidHorns}) if and only if the double extension $(f_A,f_B)$ has rigid horns, as it is the case for double coverings. We skip this (rather elementary) step as it can be deduced from the fact that the concepts of double covering and double extension with rigid horns coincide.

\begin{proposition}\label{PropositionDoubleCoveringsRigidHorns}
A double extension of racks $\alpha\colon{f_A \to f_B}$ is a double covering if and only if $\alpha$ has rigid horns (Definition \ref{DefinitionRigidHorns}).
\end{proposition}
\begin{proof}
Suppose that $\alpha$ has rigid horns in the sense of Definition \ref{DefinitionRigidHorns}. Then given an element $(a,b,c,d) \in \Eq(f_A) \square \Eq(\alpha_{\ttop})$ and an element $x\in X$ we build an $\langle f_A,\alpha_{\ttop}\rangle$-horn $V$ described by superposition of the two $f_A$-membranes $M_1$ and $M_0$ below (the so-obtained ``left-hand side'' $\alpha_{\ttop}$-membrane of $V$ is as in Example \ref{ExampleTrivialDoubleExtensionsAreDoubleCoverings}). Since the $\alpha_{\ttop}$-membranes of $V$ are rigid and $M_0$ closes into a disk, we conclude that $y \defeq x\cdot (\gr{a}\, \gr{b}^{-1}\, \gr{c}\, \gr{d}^{-1})$ is equal to $x$.
\begin{equation}\label{EquationDoubleCoveringsRigidHorns}
M_1 \ \colon \quad \vcenter{\xymatrix@C=2pt @R=6pt {& & & x  \ar@{{}{-}{}}[dddlll]|-{\dir{>}} _(0.23){\gr{a}\quad} _(0.47){\gr{b}^{-1}}_(0.63){\gr{c}\quad}_(0.87){\gr{d}^{-1}}  \ar@{{}{-}{}}[dddrrr]|{\dir{>}}^(0.25){\; \gr{b}} ^(0.45){\; \gr{b}^{-1}}^(0.65){\; \gr{c}}^(0.85){\; \gr{c}^{-1}}  \\
& & \ar@<0.6ex>@{{}{-}{}}[rr]|-{f_A} \ar@<-1.2ex>@{{}{-}{}}[rr]|-{f_A}& & & & \\
&   \ar@<-0.2ex>@{{}{-}{}}[rrrr]|-{f_A} & &&& & \\
y\ar@<1.5ex>@{{}{-}{}}[rrrrrr]|-{f_A}  & & & & &  & x
}} \qquad \qquad M_0 \ \colon \quad \vcenter{\xymatrix@C=2pt @R=6pt {& & & x  \ar@{{}{-}{}}[dddlll]|-{\dir{>}} _(0.23){\gr{d}\quad} _(0.47){\gr{c}^{-1}}_(0.63){\gr{c}\quad}_(0.87){\gr{d}^{-1}}  \ar@{{}{-}{}}[dddrrr]|{\dir{>}}^(0.25){\; \gr{c}} ^(0.45){\; \gr{c}^{-1}}^(0.65){\; \gr{c}}^(0.85){\; \gr{c}^{-1}}  \\
& & \ar@<0.6ex>@{{}{-}{}}[rr]|-{f_A} \ar@<-1.2ex>@{{}{-}{}}[rr]|-{f_A}& & & & \\
&   \ar@<-0.2ex>@{{}{-}{}}[rrrr]|-{f_A} & &&& & \\
x \ar@<1.5ex>@{{}{-}{}}[rrrrrr]|-{f_A}  & & & & &  & x
}} 
\end{equation}

Conversely suppose that $\alpha$ is a double covering and consider an $\langle f_A,\alpha_{\ttop}\rangle$-horn $V$ given by $V = (x, ((a_i,b_i,c_i,d_i), \delta_i)_{1 \leq i \leq n})$ as in Definition \ref{DefinitionFHMembrane}. Suppose that the $(c,d)$-membrane of $V$ closes into a disk, we have to show that the $(a,b)$-membrane closes into a disk (the converse is then given by symmetry of $V$ in the role of the $f_A$-membrane). 

More generally, and without assumption on the double extension $\alpha$, we show that the endpoints $a_V$ and $b_V$ of such a horn $V$ are in relation by $[\Eq(f_A),\Eq(\alpha_{\ttop})]$, which we temporarily denote by $\approx$. Observe that for all $z \in A_{\ttop}$ we have that $z \qndop^{-\delta_n} d_n \qndop^{\delta_n} a_n \approx z \qndop^{-\delta_n} c_n \qndop^{\delta_n} b_n$ (replace $\approx$ by $=$ when $\alpha$ is a double covering).
By taking $z= d_V \defeq x  \qndop^{\delta_1} d_1 \cdots \qndop^{\delta_n} d_n$ (and by reflexivity of $\approx$ and compatibility with the operation $\qndop$) we derive
\begin{equation}\label{Equation1DoubleCharactProof}
x  \qndop^{\delta_1} d_1 \cdots \qndop^{\delta_{n-1}} d_{n-1} \qndop^{\delta_n} a_n \approx x  \qndop^{\delta_1} c_1 \cdots \qndop^{\delta_{n-1}} c_{n-1} \qndop^{\delta_n} b_n.
\end{equation}
Then consider the square
\[\vcenter{\xymatrix @C= 20pt @R= 10pt {
a_{n-1} \qndop^{\delta_n} a_n  \ar@{-}[r]|-{}  \ar@{..}[d]|-{}  & d_{n-1} \qndop^{\delta_n} a_n   \ar@{..}[d]|-{} \\
b_{n-1} \qndop^{\delta_n} b_n  \ar@{-}[r]|-{}  & c_{n-1} \qndop^{\delta_n} b_n
}} \quad \in\ \Eq(f_A) \square \Eq(\alpha_{\ttop}), \]
and derive that for each $z \in A_{\ttop}$:
\begin{align*}
z  \qndop^{-\delta_{n-1}} (d_{n-1} \qndop^{\delta_n} a_n) \qndop^{\delta_{n-1}} (a_{n-1} \qndop^{\delta_n} a_n) & \approx z  \qndop^{-\delta_{n-1}} (c_{n-1} \qndop^{\delta_n} b_n) \qndop^{\delta_{n-1}} (b_{n-1} \qndop^{\delta_n} b_n); \\
z \qndop^{-\delta_n} a_n \qndop^{-\delta_{n-1}} d_{n-1} \qndop^{\delta_{n-1}} a_{n-1} \qndop^{\delta_n} a_n  & \approx z \qndop^{-\delta_n} b_n \qndop^{-\delta_{n-1}} c_{n-1} \qndop^{\delta_{n-1}} b_{n-1} \qndop^{\delta_n} b_n.
\end{align*} 
Applying this to Equation~\eqref{Equation1DoubleCharactProof} we obtain
\[
x  \qndop^{\delta_1}  d_1 \cdots \qndop^{\delta_{n-2}} d_{n-2} \qndop^{\delta_{n-1}}  a_{n-1} \qndop^{\delta_n} a_n \approx x  \qndop^{\delta_1}  c_1 \cdots \qndop^{\delta_{n-2}} c_{n-2} \qndop^{\delta_{n-1}}  b_{n-1} \qndop^{\delta_n} b_n.\]
We repeat the argument with 
\[\vcenter{\xymatrix @C= 20pt @R= 10pt {
a_{n-2} \qndop^{\delta_{n-1}} a_{n-1} \qndop^{\delta_n} a_n   \ar@{-}[r]|-{}  \ar@{..}[d]|-{}  & d_{n-2} \qndop^{\delta_{n-1}} a_{n-1}   \qndop^{\delta_n} a_n  \ar@{..}[d]|-{} \\
b_{n-2} \qndop^{\delta_{n-1}} b_{n-1} \qndop^{\delta_n} b_n  \ar@{-}[r]|-{}  & c_{n-2} \qndop^{\delta_{n-1}} b_{n-1}   \qndop^{\delta_n} b_n
}} \quad \in\ \Eq(f_A) \square \Eq(\alpha_{\ttop}),\]
and conclude by induction that also $x \qndop^{\delta_1} a_1  \cdots  \qndop^{\delta_n} a_n  \approx x \qndop^{\delta_1} b_1  \cdots  \qndop^{\delta_n} b_n$.
\end{proof}

Given a double extension $\alpha\colon{f_A \to f_B}$, the \emph{rigid horns} condition from Definition \ref{DefinitionRigidHorns}, or more precisely Definition \ref{DefinitionXAlpha} below, make sense of what it means for two elements of $A_{\ttop}$ to be ``linked under the action of a primitive path from $\Eq(f_A) \square \Eq(\alpha_{\ttop})$'' (see page \pageref{SectionVolumes}).
\begin{definition}\label{DefinitionXAlpha}
Given a double extension $\alpha\colon{f_A \to f_B}$, we define the set $X_{\alpha}$ to be the set of those pairs $(x,y)$ in $A_{\ttop} \times A_{\ttop}$ such that there exists a $\langle f_A,\alpha_{\ttop}\rangle$-horn $V$ as in Definition \ref{DefinitionFHMembrane} such that $x$ and $y$ are the endpoints of one of the membranes $M_1^V$ of $V$, such that moreover the membrane $M^V_0$, which is opposite to $M_1^V$, closes into a disk.
\end{definition}
These pairs in $X_{\alpha}$ are the pairs of elements which \emph{would be identified if $\alpha$ had rigid horns}. We just saw that $X_{\alpha}$ contains the generators of $[\Eq(f_A),\Eq(\alpha_{\ttop})]$ and moreover $X_{\alpha} \subseteq [\Eq(f_A),\Eq(\alpha_{\ttop})]$. Hence if we can show that $X_{\alpha}$ defines a congruence on $A_{\ttop}$, we can deduce that $X_{\alpha}$ is the centralizing congruence $[\Eq(f_A),\Eq(\alpha_{\ttop})]$ (see Corollary \ref{CorollaryCharacterizingC2} below).

Now recall from Part I that coverings are equivalently described via membranes or via symmetric paths. Proposition \ref{PropositionDoubleCoveringsRigidHorns} corresponds to the description via membranes. In the following section, we adapt the idea of a symmetric path to the two-dimensional context. Equipped with this concept and that of a rigid horn, we provide a full description of a general element in $[\Eq(f_A),\Eq(\alpha_{\ttop})]$.

\subsection{Symmetric paths for double extensions}
We describe \emph{symmetric paths} in a slightly more general context than expected, because of Lemma \ref{LemmaKernelOfFgFFgH} below.
\begin{definition}\label{DefinitionK2}
Given a pair of morphisms $f \colon{G \to H}$, $h \colon{G \to K}$ in $\GRP$, and a generating set $A \subseteq G$ (i.e.~such that $G = \langle a \, \mid\, a \in A \rangle_G$), we define (implicitly \emph{with respect to $A$}):
\begin{enumerate}[label=(\roman*)]
\item four elements $g_a$, $g_b$, $g_c$ and $g_d$ in $G$ to be \emph{$\langle f,h\rangle$-symmetric (to each other)} if there exists $n\in \N$ and a sequence of quadruples $(a_1,b_1,c_1,d_1)$, $\ldots$, $(a_n,b_n,c_n,d_n)$ in the set $A^4\cap (\Eq(f)\square \Eq(h))$, and finally, if for each $1\leq i \leq n$, there is $\delta_i \in \lbrace -1,\, 1\rbrace$ such that: 
\begin{equation}\label{EquationFHSymmetricQuadruple}
g_a = a_1^{\delta_1} \cdots  a_n^{\delta_n}, \quad g_b = b_1^{\delta_1} \cdots  b_n^{\delta_n}, \quad g_c = c_1^{\delta_1} \cdots  c_n^{\delta_n}, \quad g_d = d_1^{\delta_1} \cdots  d_n^{\delta_n}.
\end{equation}
We often call such $g_a$, $g_b$, $g_c$ and $g_d$ an \emph{$\langle f,h\rangle$-symmetric quadruple}.
\item $\K{\langle f,h\rangle}$ to be the set of \emph{$\langle f,h\rangle$-symmetric paths}, i.e.~the elements $g \in G$ such that $g = g_a g_b^{-1} g_c g_d^{-1}$ for some $\langle f,h\rangle$-symmetric quadruple $g_a$, $g_b$, $g_c$ and $g_d \in G$.
\end{enumerate}
\end{definition}

\begin{lemma}\label{LemmaGroupCharacterization2}
Given the hypotheses of Definition \ref{DefinitionK2}, the set of $\langle f,h\rangle$-symmetric paths $\K{\langle f,h\rangle}$ defines a normal subgroup of $G$.
\end{lemma}
\begin{proof}
Let $g_a$, $g_b$, $g_c$ and $g_d$ be $\langle f,h\rangle$-symmetric (to each other). Observe that $g_d^{-1}$, $g_c^{-1}$, $g_b^{-1}$ and $g_a^{-1}$ are also $\langle f,h\rangle$-symmetric, and thus $\K{\langle f,h\rangle}$ is closed under inverses. Moreover, if $h_a$, $h_b$, $h_c$ and $h_d$ are $\langle f,h\rangle$-symmetric, and $g=g_ag_b^{-1}g_cg_d^{-1}$, $h=h_ah_b^{-1}h_ch_d^{-1}$, then
\[gh=k_a k_b^{-1}k_ck_d^{-1},\]
with $k_a = h_ah_b^{-1}h_bh_a^{-1}g_a$, $k_b =h_ah_a^{-1}h_bh_a^{-1} g_b$, $k_c =h_dh_d^{-1}h_bh_a^{-1} g_c$ and $k_d =h_dh_c^{-1}h_bh_a^{-1} g_d$ which are $\langle f,h\rangle$-symmetric. Finally since $A$ generates $G$, for any $k \in G$, $kg_a$, $kg_b$, $kg_c$ and $kg_d$ are $\langle f,h\rangle$-symmetric to each other, and thus $kgk^{-1}= kg_ag_b^{-1}k^{-1}kg_cg_d^{-1}k^{-1}  \in \K{\langle f,h\rangle}$ is an $\langle f,h\rangle$-symmetric path.
\end{proof}

\begin{notation}
For a double extension of racks (or quandles) $\alpha\colon{f_A \to f_B}$, we often write \emph{$\langle f_A,\alpha_{\ttop}\rangle$-symmetric} (quadruple or path) instead of $\langle \vec{f},\vec{\alpha_{\ttop}}\rangle$-symmetric (quadruple or path -- see for instance Definition \ref{DefinitionFHMembrane}). An \emph{$\langle f_A,\alpha_{\ttop}\rangle$-symmetric trail} $(x,g)$ in $A_{\ttop}$ is a trail where $g$ is an $\langle f_A,\alpha_{\ttop}\rangle$-symmetric path. 
\end{notation}

\begin{lemma}
Given a double extension in $\RCK$ (or $\QND$) $\alpha\colon{f_A \to f_B}$, the set $X_{\alpha}$ (Defintion \ref{DefinitionXAlpha}) is the underlying set of the congruence $\sim_{\K{\langle f_A,\alpha_{\ttop}\rangle}}$ induced by the action of $\langle f_A,\alpha_{\ttop}\rangle$-symmetric paths on $A_{\ttop}$.
\end{lemma}
\begin{proof}
Given $x$ and $y \in A_{\ttop}$ such that $x \sim_{\K{\langle f_A,\alpha_{\ttop}\rangle}} y$, i.e. such that $y = x \cdot (g_a g_b^{-1} g_c g_d^{-1})$ for some $\langle f_A,\alpha_{\ttop}\rangle$-symmetric quadruple as in Definition \ref{DefinitionK2}. The pair $(x,y)$ is in $X_{\alpha}$ as one can deduce from the construction of $V$ as in Equation \eqref{EquationDoubleCoveringsRigidHorns} from the proof of Proposition \ref{PropositionDoubleCoveringsRigidHorns}, where one replaces every occurrence of $\gr{a}$ by $\gr{a_1}^{\delta_1} \cdots  \gr{a_n}^{\delta_n}$ and also for $\gr{b}$ by $\gr{b_1}^{\delta_1} \cdots  \gr{b_n}^{\delta_n}$, and similarly $\gr{c}$ by $\gr{c_1}^{\delta_1} \cdots  \gr{c_n}^{\delta_n}$ and $\gr{d}$ by $\gr{d_1}^{\delta_1} \cdots  \gr{d_n}^{\delta_n}$.

Conversely, and without loss of generality, consider an $\langle f_A,\alpha_{\ttop}\rangle$-horn $V$ given by the data $V = (x, ((a_i,b_i,c_i,d_i), \delta_i)_{1 \leq i \leq n})$ as in Definition \ref{DefinitionFHMembrane}, such that moreover the endpoints $c_V = d_V$. Observe that the endpoint $b_V = a_V \cdot \big( (g^V_a)^{-1} \, g^V_d\, (g^V_c)^{-1}\, g^V_b \big)$ is obtained from the endpoint $a_V$ by the action of an $\langle f_A,\alpha_{\ttop}\rangle$-symmetric path.
\end{proof}

As a conclusion to the discussion below Definition \ref{DefinitionXAlpha}, we give a characterization of a general element in $[\Eq(f_A),\Eq(\alpha_{\ttop})]$ which we show to be an \emph{orbit congruence} (see Part I and reference therein).

\begin{corollary}\label{CorollaryCharacterizingC2}
Given a double extension of racks (or quandles) $\alpha\colon{f_A \to f_B}$, the centralization congruence $[\Eq(f_A),\Eq(\alpha_{\ttop})]$ coincides with the congruence $\sim_{\K{\langle f_A,\alpha_{\ttop}\rangle}}$ generated by the action of $\langle f_A,\alpha_{\ttop}\rangle$-symmetric paths, also described by the set of pairs in $X_{\alpha}$ (Definition \ref{DefinitionXAlpha}, i.e. those pairs of elements of $A_{\ttop}$ which would be identified if $\alpha$ had rigid horns).
\end{corollary}
 
\subsubsection{Describing symmetric paths differently ?}
Given a morphism $f$ in $\RCK$ (or $\QND$), $f$-symmetric paths are described as the elements in the kernel $\Ker(\vec{f})$ of $\vec{f}$ (which is our notation for $\Pth(f)$). It is unclear to us whether this result generalizes in higher dimensions. Our understanding is that the question should be: given a double extension $\alpha$, do the normal subgroups $\Ker(\vec{f_A}) \cap \Ker(\vec{\alpha_{\ttop}})$ and $\K{\langle f_A,\alpha_{\ttop}\rangle}$ coincide ? Whether the answer is negative or positive, this would help to specify more precisely how to understand these $\langle f_A,\, \alpha_{\ttop}\rangle$-symmetric paths algebraically. Following the strategy from Section 2.4.11 of Part I, we were able to show that:
\begin{lemma}\label{LemmaKernelOfFgFFgH}
Given two surjective functions $f\colon {A \to B}$ and $h \colon{A \to C}$ such that $\Eq(f) \circ \Eq(h) = \Eq(h) \circ \Eq(f)$, the intersection $\Ker(\Fg(f)) \cap \Ker(\Fg(h))$ of the kernels of the induced group homomorphisms $\Fg(f) \colon {\Fg(A) \to \Fg(B)}$ and $\Fg(h) \colon {\Fg(A) \to \Fg(C)}$ is $\K{\langle \Fg(f),\Fg(h)\rangle}$ (with respect to $A$) as in Definition \ref{DefinitionK2}.
\end{lemma}
Our proof is rather combinatorial and can be found in a separate publication \cite{Ren2021}. Now given a double extension of racks (or quandles) $\alpha$, it is easy to obtain $\K{\langle f_A,\alpha_{\ttop}\rangle}\leq \Ker(\vec{f_A}) \cap \Ker(\vec{\alpha_{\ttop}})$ as the image of $\K{\langle \Fg(f_A),\Fg(\alpha_{\ttop})\rangle} = \Ker(\Fg(f_A)) \cap \Ker(\Fg(\alpha_{\ttop}))$ by $q_{A_{\ttop}}\colon{\Fg(\U(A_{\ttop})) \to \Pth(A_{\ttop})}$ (defined as in 
Paragraph \ref{ParagraphTightRelationshipWithGroups} above). Hence $\Ker(\vec{f_A}) \cap \Ker(\vec{\alpha_{\ttop}}) = \K{\langle \vec{f_A},\vec{\alpha_{\ttop}}\rangle}$ if and only if the induced morphism $\bar{q} \colon {\Ker(\Fg(f_A)) \cap \Ker(\Fg(\alpha_{\ttop})) \to \Ker(\vec{f_A}) \cap \Ker(\vec{\alpha_{\ttop}})}$ is a surjection. We were unfortunately not able to identify a reason why this should be true in general (see Observation \ref{ObservationDescribingSymmetricPathsDifferently} for alternative descriptions).

Besides, we note that even if the two groups do not coincide, it might still be that the action of $\K{\langle f_A,\alpha_{\ttop}\rangle}$ on $A_{\ttop}$ and the action of $\Ker(\vec{f_A}) \cap \Ker(\vec{\alpha_{\ttop}})$ on $A_{\ttop}$ define the same congruence in $\RCK$ (or $\QND$). Finally, we ask the ``even weaker'' question: is $\Ker(\vec{f_A}) \cap \Ker(\vec{\alpha_{\ttop}})$ in the center of $\Pth(A_{\ttop})$ ? This would imply that the image by $\Conj\Pth$ of a double covering is still a double covering (see Section \ref{SectionConjugationQuandles}).

\begin{observation}\label{ObservationDescribingSymmetricPathsDifferently}

More precisely, observe that a double extension of racks or quandles $\alpha$ is sent to a double extension $\alpha' = \Pth(\alpha)$ in groups since $\Pth$ preserves pushouts of surjections and $\GRP$ is a Mal'tsev category. Call $c'\colon{\Pth(A_{\ttop}) \twoheadrightarrow P'}$ the surjective comparison map of $\alpha'$. The double extension $\alpha$ is also sent to a double extension in $\SET$ by the forgetful functor, as pullbacks and surjections are preserved. This double extension in $\SET$ is then sent by $\Fg$ to a double extension $\alpha''$ in $\GRP$, since pushouts of surjections are preserved by left-adjoints. Write $c'' \colon{\Fg(A_{\ttop}) \twoheadrightarrow P''}$ for the comparison map of $\alpha''$. Finally $\alpha''$ is sent by $\Conj$ to a double extension in $\RCK$ again, which is sent to a double extension $\alpha'''$ by $\Pth$. Write $c''' \colon {\Fg(A_{\ttop} \times \Fg(A_{\ttop})) \twoheadrightarrow P'''}$ for the comparison map of $\alpha'''$. We have thus three layers $\alpha'''$, $\alpha''$ and $\alpha'$ of double extensions in $\GRP$ fitting into a fork $\alpha''' \rightrightarrows \alpha'' \to \alpha'$ of $3$-dimensional arrows, such that each arrow is a square of double extensions, and the top pair is a reflexive graph whose legs are $3$-fold extensions.

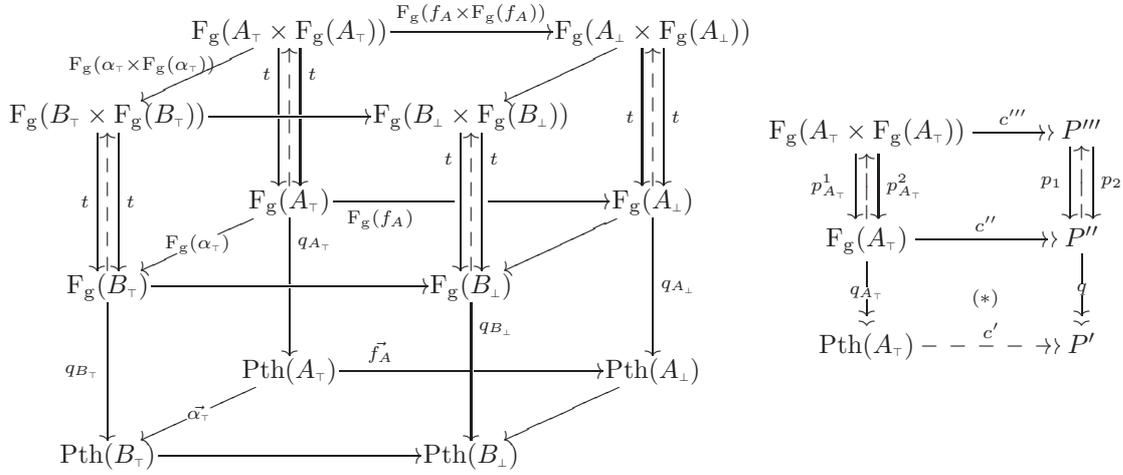
\begin{figure}[!h]
{\begin{center}
$\vcenter{\xymatrix@1@!0@R=32pt@C=68pt{
&\Fg(A_{\ttop} \times \Fg(A_{\ttop})) \ar[rr]^-{\Fg(f_A \times \Fg(f_A))} \ar@<+4pt>[dd]^(.25){t}|-{\hole \ \hole} \ar@<-4pt>[dd]_(.25){t}|-{\hole \ \hole} \ar@{{<}{--}{}}[dd]|-{\hole \ \hole} \ar[ld]_(.6){\Fg(\alpha_{\ttop} \times \Fg(\alpha_{\ttop}))} && \Fg(A_{\pperp} \times \Fg(A_{\pperp})) \ar@<+4pt>[dd]^-{t} \ar@<-4pt>[dd]_-{t} \ar@{{<}{--}{}}[dd] \ar[ld]|-{ } \\
\Fg(B_{\ttop} \times \Fg(B_{\ttop})) \ar@<+4pt>[dd]^-{t} \ar@<-4pt>[dd]_-{t} \ar@{{<}{--}{}}[dd]  \ar[rr]^(.75){} &&\Fg(B_{\pperp} \times \Fg(B_{\pperp})) \ar@<+4pt>[dd]^(.25){t} \ar@<-4pt>[dd]_(.25){t} \ar@{{<}{--}{}}[dd] \\
& \Fg(A_{\ttop})  \ar[rr]_(.25){\Fg(f_A)}|-{\hole\ \hole \ \hole} \ar[dd]^(.25){q_{A_{\ttop}}}|-{\hole} \ar[ld]|-{\Fg(\alpha_{\ttop})} && \Fg(A_{\pperp}) \ar[dd]^-{q_{A_{\pperp}}} \ar[ld]|-{} \\
\Fg(B_{\ttop}) \ar[rr]^(.75){} \ar[dd]_-{q_{B_{\ttop}}} && \Fg(B_{\pperp}) \ar[dd]^(.25){q_{B_{\pperp}}} \\
& \Pth(A_{\ttop}) \ar[ld]|-{\vec{\alpha_{\ttop}}} \ar[rr]^(.25){\vec{f_A}}|(.5){\hole} && \Pth(A_{\pperp}) \ar[ld]^-{} \\
\Pth(B_{\ttop}) \ar[rr]_-{} && \Pth(B_{\pperp})}} \  \vcenter{\xymatrix@C=30pt{\Fg(A_{\ttop} \times \Fg(A_{\ttop})) \ar@<+1ex>[d]^-{p_{A_{\ttop}}^2} \ar@{{<}{--}{}}[d]|-{}  \ar@<-1ex>[d]_-{p_{A_{\ttop}}^1} \ar@{{}{-}{>>}}[r]^-{c'''}  & P''' \ar@<+1ex>[d]^-{p_2} \ar@<-1ex>[d]_-{p_1} \ar@{{<}{--}{}}[d]|-{}  \\
\Fg(A_{\ttop}) \ar@{{}{}{}}[rd]|-{(*)} \ar@{{}{-}{>>}}[d]|-{q_{A_{\ttop}}} \ar@{{}{-}{>>}}[r]^-{c''} & P''  \ar@{{}{-}{>>}}[d]|-{q} \\
\Pth(A_{\ttop}) \ar@{{}{--}{>>}}[r]^-{c'} & P'}} $
\end{center}}\caption{The fork $\alpha''' \rightrightarrows \alpha'' \to \alpha'$}\label{FigureObservationSymmetricPaths}
\end{figure}

By the universal property of the pullbacks, $P'''$, $P''$ and $P'$, there is an induced reflexive graph $p_1,p_2 \colon{P''' \rightrightarrows P''}$, as well as a surjection $q\colon{P'' \to P'}$ which coequalises $p_1$ and $p_2$, such that the whole fork fits into the commutative diagrams of Figure \ref{FigureObservationSymmetricPaths}. By Lemma 1.2 in \cite{Bou2003}, $(*)$ is a double extension if and only if $q$ is the coequalizer of $p_1$ and $p_2$, which is also equivalent to the fork being a double extension. 
These three equivalent conditions are satisfied if and only if the aforementioned morphism $\bar{q}\colon{\Ker(\Fg(f_A)) \cap \Ker(\Fg(\alpha_{\ttop})) \to \Ker(\vec{f_A}) \cap \Ker(\vec{\alpha_{\ttop}})}$ is a surjection.
\end{observation}

\section{The $\Gamma^1$-coverings (or double central extensions of racks and quandles)}\label{SectionCharacterization}
In this section, we show that the concept of double covering of racks and quandles (or algebraically central double extension) and the concept of $\Gamma^1$-covering (or double central extension of racks and quandles) coincide. In order to do so, we first show that double coverings are reflected and preserved by pullbacks along double extensions. Since trivial $\Gamma^1$-coverings are double coverings, this implies that $\Gamma^1$-coverings are also double coverings.

\subsection{Double coverings are reflected and preserved by pullbacks}

We first show a general result about morphisms induced by $3$-fold extensions (see Definition \ref{DefinitionThreeFoldExtension}). Observe that given the hypothesis of Lemma \ref{LemmaBarrKock}, we deduce from \cite[Lemma 2.1]{Bou2003} that if the right hand square of Diagram \eqref{DiagramLemmaBarrKock} is a double extension, then $f$ is an extension even if $\C$ is merely regular (and not Barr-exact).

\begin{lemma}\label{LemmaSurjectionBetweenParallelisticDERel}
Consider a $3$-fold extension $(\sigma,\beta)\colon{\gamma \to \alpha}$ in a regular category $\C$.
The morphism $\square_{(\sigma,\beta)} \colon{\Eq(f_C) \square \Eq(\gamma_{\ttop}) \to \Eq(f_A) \square \Eq(\alpha_{\ttop})}$ induced by $(\sigma,\beta)$ between the parallelistic double equivalence relations is a regular epimorphism.
\end{lemma}
\begin{proof} First we recall how to build the double parallelistic relations of interest. By taking kernel pairs horizontally and then vertically, we build the Diagrams \eqref{Diagram3by3}, where the induced pairs $(p_1,p_2)\colon R_{\gamma} \rightrightarrows \Eq(f_C)$ and $(\pi_1,\pi_2)\colon R_{\alpha} \rightrightarrows \Eq(f_A)$ on the top rows, are the kernel pairs of $\bar{\gamma}$ and $\bar{\alpha}$ by a local version of the denormalised $3 \times 3$ Lemma (see \cite{Bou2003} and Lemma \ref{LemmaDenormalised3X3} below). As a consequence, all the rows and columns of Diagrams \eqref{Diagram3by3} are exact forks.
\begin{equation}\label{Diagram3by3}
\vcenter{\xymatrix @R=15pt @C=30pt {
\Eq(f_{\gamma}) \ar@<2pt>[d]^-{} \ar@<-2pt>[d]_-{} \ar@<2pt>[r]^-{p_1} \ar@<-2pt>[r]_-{p_2} & \Eq(f_C) \ar@<2pt>[d]^-{ } \ar@<-2pt>[d]_-{ } \ar[r]^-{\bar{\gamma}} & \Eq(f_D) \ar@<2pt>[d]^-{ } \ar@<-2pt>[d]_-{ } \\
\Eq(\gamma_{\ttop}) \ar[d]_-{f_{\gamma}} \ar@<2pt>[r]^-{} \ar@<-2pt>[r]_-{} & C_{\ttop} \ar[r]^-{\gamma_{\ttop}}   \ar[d]^-{f_C} & D_{\ttop} \ar[d]^-{f_D}\\
\Eq(\gamma_{\pperp}) \ar@<2pt>[r]^-{} \ar@<-2pt>[r]_-{} & C_{\pperp} \ar[r]_-{\gamma_{\pperp}}  &  \ D_{\pperp}
}} \qquad 
\vcenter{\xymatrix @R=15pt @C=30pt  {
\Eq(f_{\alpha}) \ar@<2pt>[d]^-{} \ar@<-2pt>[d]_-{} \ar@<2pt>[r]^-{\pi_1} \ar@<-2pt>[r]_-{\pi_2} & \Eq(f_A) \ar@<2pt>[d]^-{} \ar@<-2pt>[d]_-{} \ar[r]^-{\bar{\alpha}} & \Eq(f_B) \ar@<2pt>[d]^-{} \ar@<-2pt>[d]_-{} \\
\Eq(\alpha_{\ttop}) \ar[d]_-{f_{\alpha}} \ar@<2pt>[r]^-{} \ar@<-2pt>[r]_-{} & A_{\ttop} \ar[r]^-{\alpha_{\ttop}}   \ar[d]^-{f_A} & B_{\ttop} \ar[d]^-{f_B}\\
\Eq(\alpha_{\pperp}) \ar@<2pt>[r]^-{} \ar@<-2pt>[r]_-{} & A_{\pperp} \ar[r]_-{\alpha_{\pperp}}  &  \ B_{\pperp}
}} 
\end{equation}
Then by Proposition 2.1 from \cite{Bou2003}, $\Eq(f_{\gamma}) = \Eq(f_C) \square \Eq(\gamma_{\ttop})$ and $\Eq(f_{\alpha}) = \Eq(f_A) \square \Eq(\alpha_{\ttop})$ are the double parallelistic relations of interest. 

Now the $3$-fold extension $(\sigma, \beta)$ induces morphisms between the left-hand and right-hand Diagrams \eqref{Diagram3by3}, such that on the top row we have
\[ \xymatrix @R=15pt @C=30pt {
\Eq(f_C) \square \Eq(\gamma_{\ttop}) \ar[d]_-{\square_{(\sigma,\beta)}} \ar@<2pt>[r]^-{p_1} \ar@<-2pt>[r]_-{p_2} & \Eq(f_C) \ar[r]^-{\bar{\gamma}}   \ar[d]^-{\bar{\sigma}} & \Eq(f_D) \ar[d]^-{\bar{\beta}}\\
\Eq(f_A) \square \Eq(\alpha_{\ttop}) \ar@<2pt>[r]^-{\pi_1} \ar@<-2pt>[r]_-{\pi_2} & \Eq(f_A) \ar[r]_-{\bar{\alpha}}  &  \Eq(f_B).
}
\]
Hence, by \cite[Lemma 2.1]{Bou2003}, it suffices to prove that the right hand commutative square $(\bar{\sigma},\bar{\beta})$ is a double extension. This can be deduced from the fact that $(\sigma,\beta)$ is a $3$-fold extension. When $\C$ is Barr-exact category, we may use \cite[Lemma 3.2]{EvGoeVdl2012}. However, for a general regular category $\C$, we consider the ``fork of comparison maps'':
\begin{equation}\label{DiagramForkOfComparisonMaps}
\vcenter{\xymatrix@R=15pt @C=30pt {
\Eq(f_C)  \ar[d]_-{p} \ar@<2pt>[r]^-{} \ar@<-2pt>[r]_-{} & C_{\ttop} \ar[r]^-{f_C}   \ar[d]^-{} & C_{\pperp}  \ar[d]^-{}  \\
\Eq(f_A) \times_{\Eq(f_B)} \Eq(f_D) \ar@<2pt>[r]^-{} \ar@<-2pt>[r]_-{} & A_{\ttop} \times_{B_{\ttop}} D_{\ttop}  \ar[r]_-{f_P} & A_{\pperp} \times_{B_{\pperp}} D_{\pperp}.
}}
\end{equation} 
where the bottom row is exact by Lemma \ref{LemmaDenormalised3X3} (as for the top rows in Diagrams \eqref{Diagram3by3} above).
Moreover, the right hand square $(f_C,f_P)$ is a double extension since $(\sigma,\beta)$ is a $3$-fold extension, and thus the morphism $p$ is a regular epimorphism. Since $p$ is also the comparison map of $(\bar{\sigma},\bar{\beta})$, this concludes the proof. \qedhere
\end{proof}

Using the study of the denormalised $3 \times 3$ Lemma from \cite{Bou2003}, we obtain the following result, where, as usual, we locally use double extensions instead of working globally in a Mal'tsev category.

\begin{lemma}\label{LemmaDenormalised3X3}
Given a regular category $\C$ as well as a $3 \times 3$ diagram such as any of the two Diagrams \eqref{Diagram3by3}, where all columns are exact, the middle row and the bottom row are exact, and the bottom right-hand square is a double extension, then the top row is also exact.
\end{lemma}
\begin{proof}
The top row is left-exact by \cite[Theorem 2.2]{Bou2003}. Then the top right morphism is a regular epimorphism by \cite[Lemma 2.1]{Bou2003}. We conclude by the fact that in any category with pullbacks, regular epimorphisms are the coequalizers of their kernel pairs.
\end{proof}

Working in the categories $\RCK$ and $\QND$ again we obtain the following.
\begin{corollary}
Double coverings are stable by pullbacks along double extensions and reflected along $3$-fold extensions. In particular, $\Gamma^1$-coverings are double coverings.
\end{corollary}
\begin{proof}
Consider a $3$-fold extension $(\sigma,\beta)\colon{\gamma \to \alpha}$ in $\RCK$ (or $\QND$) such as in Definition \ref{DefinitionThreeFoldExtension}. 

Assume that $\gamma$ is a double covering. Given $x \in A_{\ttop}$ and $(a,b,c,d) \in \Eq(f_A) \square \Eq(\alpha_{\ttop})$, the surjectivity of $\sigma_{\ttop}$ and $\square_{(\sigma,\beta)}$ (from Lemma \ref{LemmaSurjectionBetweenParallelisticDERel}) yields $x' \in C_{\ttop}$ and $(a',b',c',d') \in \Eq(f_C) \square \Eq(\gamma_{\ttop})$ such that $\sigma_{\ttop}(x')=x$, $\sigma_{\ttop}(a')=a$, $\sigma_{\ttop}(b')=b$, $\sigma_{\ttop}(c')=c$, and $\sigma_{\ttop}(d')=d$. Since $x' \qndop a' \qndiop b' \qndop c' \qndiop d' = x'$ in $C_{\ttop}$, the image of this equation by $\sigma_{\ttop}$ yields $x \qndop a \qndiop b \qndop c \qndiop d = x$ in $A_{\ttop}$. Hence $\alpha$ is a double covering.

Conversely assume that $\alpha$ is a double covering and suppose that $(\sigma,\beta)$ describes the pullback of $\alpha$ and $\beta$, i.e.~suppose that the comparison map $\pi$ of $(\sigma,\beta)$ is an isomorphism (see Definition \ref{DefinitionThreeFoldExtension}). Then we consider $x \in C_{\ttop}$ and $(a,b,c,d) \in \Eq(f_C) \square \Eq(\gamma_{\ttop})$, and we have to show that $y \defeq x \qndop a \qndiop b \qndop c \qndiop d$ is equal to $x$. It suffices to check the equality in both components of the pullback $C_{\ttop}$, via the projections $\gamma_{\ttop}$ and $\sigma_{\ttop}$. We have indeed $\gamma_{\ttop}(y) = \gamma_{\ttop}(x)$ and since $\alpha$ is a double covering, we have also $\sigma_{\ttop}(y) = \sigma_{\ttop}(x)$. Hence $\gamma$ is a double covering.
\end{proof}

\subsection{Double coverings are $\Gamma^1$-coverings}

As described in Section \ref{SectionTowardsHigherCoveringTheory}, given a double covering $\alpha \colon {f_A \to f_B}$, we build the canonical double projective presentation of its codomain $f_B$: \[ p_{f_B} \defeq (p_{f_B}^1,p_{f_B}^0) \colon {p_B \to f_B} \quad \text{(see Diagram \eqref{EquationProjectivePresentationD2})}.\]
We then consider the pullback of our double covering $\alpha$ along our projective presentation $p_{f_B}$. This yields a double covering $\gamma\colon {f_C \to p_B}$ with projective codomain $p_B \colon {\Fr(P) \to \Fr(B_{\pperp})}$ (or $p_B \colon {\Fq(P) \to \Fq(B_{\pperp})}$ if we work in $\QND$). We show that such double coverings are always trivial double coverings, which implies that the double covering $\alpha$ is a $\Gamma^1$-covering.

\begin{proposition}\label{PropositionCharacterizationDoubleCoverings}
If a double covering of racks $\gamma = (\gamma_{\ttop},\gamma_{\pperp})$ has a projective codomain of the form $p \colon {\Fr(P) \to \Fr(B_{\pperp})}$ for some sets $P$ and $B_{\pperp}$:
\[
\vcenter{\xymatrix@R=2pt @ C=17pt{
C_{\ttop} \ar[rrr]^{f_C} \ar[rd]_(0.6){\langle \gamma_{\ttop},f_C \rangle}  \ar[dddd]_-{\gamma_{\ttop}} & & & C_{\pperp}  \ar[dddd]^-{\gamma_{\pperp}} \\
 & Q \ar[rru]|-{\pi_2} \ar[lddd]|-{\pi_1}&   \\
\\
\\
\Fr(P)  \ar[rrr]_{p} & & & \Fr(B_{\pperp}),
}}
\]
where $Q \defeq \Fr(P)  \times_{\Fr(B_{\pperp})} C_{\pperp}$, then $\gamma$ is a trivial double covering. The result holds similarly in $\QND$, for a double covering $\gamma$ with codomain of the form $p \colon {\Fq(P) \to \Fq(B_{\pperp})}$.
\end{proposition}
\begin{proof}
Consider a $f_C$-membrane $M = (x, (a_i,b_i,\delta_i)_{1\leq i \leq n})$ in $C_{\ttop}$ such that the image of $M$ by $\gamma_{\ttop}$ closes into a disk in $\Fr(P)$, i.e.~$\gamma_{\ttop}(x) \cdot \big( \vec{\gamma_{\ttop}}(\gr{a_1}^{\delta_1} \cdots\gr{a_n}^{\delta_n} \gr{b_n}^{-\delta_n} \cdots \gr{b_1}^{-\delta_1}) \big) = \gamma_{\ttop}(x)$.
\[
\vcenter{\xymatrix@C=2pt @R=6pt {& & & x  \ar@{{}{-}{}}[dddlll]|-{\dir{>}} _(0.2){\gr{a_1}^{\delta_1}}_(0.45){\iddots \ }_(0.7){\gr{a_n}^{\delta_n} }   \ar@{{}{-}{}}[dddrrr]|{\dir{>}}^(0.2){\ \gr{b_1}^{\delta_1}}^(0.45){\, \ddots}^(0.7){\ \gr{b_n}^{\delta_n}}     \\
& & \ar@{{}{-}{}}[rr]|-{f_C} & & & & \\
&  \ar@<0.7ex>@{{}{-}{}}[rrrr]|-{f_C} & &&& & \\
a^M \defeq x \cdot (g^M_a) \ar@<2.3ex>@{{}{-}{}}[rrrrrr]|-{f_C}  & & & & &  & b^M \defeq x \cdot (g^M_b)
}} \longmapsto \quad  \vcenter{\xymatrix@C=4pt @R=6pt {& & \gamma_1(x)  \ar@/_8ex/@{{}{-}{}}[ddd]|{\dir{>}} _(0.3){\vec{\gamma_1}(g^M_a)}    \ar@/^8ex/@{{}{-}{}}[ddd]|{\dir{>}}^(0.3){\vec{\gamma_{\ttop}}(g^M_b)}    \\
& \ar@<0.52ex>@{{}{-}{}}[rr]|-{f_C} & & &  \\
& \ar@<1.3ex>@{{}{-}{}}[rr]|-{f_C} \ar@<-0.95ex>@{{}{-}{}}[rr]|-{f_C} & & & \\
& & \gamma_{\ttop}(a^M)
}}
\]
Let us write $y_i \defeq \gamma_{\ttop}(a_i)$ and $x_i \defeq \gamma_{\ttop}(b_i)$ for each $1 \leq i \leq n$. Then 
\[h \defeq \vec{\gamma_{\ttop}}(\gr{a_1}^{\delta_1} \cdots\gr{a_3}^{\delta_n} \gr{b_n}^{-\delta_n} \cdots \gr{b_1}^{-\delta_1}) = \gr{y_1}^{\delta_1} \cdots\gr{y_n}^{\delta_n} \gr{x_n}^{-\delta_n} \cdots \gr{x_1}^{-\delta_1} \in \Pth(\Fr(C_{\ttop})), \]  
yields the neutral element $e$ since the action of $\Pth(\Fr(C_{\ttop}))$ on $\Fr(C_{\ttop})$ is free -- note that in the context of $\QND$, this path $h$ is in the group $\Pth\degree(\Fq(C_{\ttop}))$ which acts freely on $\Fq(C_{\ttop})$.

Since $p$ is projective (with respect to double extensions), there is a splitting $s \defeq (s_{\ttop},s_{\pperp})$ of $\gamma$ such that $\gamma s =(1_{\Fr(P)},1_{\Fr(B_{\pperp})})$ is the identity morphism. If we define $d_i \defeq s_{\ttop}(\gamma_{\ttop}(a_i))$ and $c_i \defeq s_{\ttop}(\gamma_{\ttop}(b_i))$ for each $1 \leq i \leq n$, then we have that
\begin{itemize}
\item $\gr{c_1}^{\delta_1} \cdots\gr{c_n}^{\delta_n} \gr{d_n}^{-\delta_n} \cdots \gr{d_1}^{-\delta_1} = \vec{s_{\ttop}}(h^{-1}) = e \in \Pth(C_{\ttop})$ is trivial;
\item moreover $\gamma_{\ttop}(d_i) = \gamma_{\ttop}(a_i)$ and $\gamma_{\ttop}(c_i) = \gamma_{\ttop}(b_i)$ for each $1\leq i \leq n$;
\item and thus the product $g^M_a(g^M_b)^{-1}$ defines an $\langle f_C,\gamma_{\ttop}\rangle$-symmetric path in $\Pth(C_{\ttop})$ since $g^M_a(g^M_b)^{-1}  = g^M_a(g^M_b)^{-1} \vec{s_{\ttop}}(h^{-1}) = \gr{a_1}^{\delta_1} \cdots\gr{a_n}^{\delta_n} \gr{b_n}^{-\delta_n} \cdots \gr{b_1}^{-\delta_1} \gr{c_1}^{\delta_1} \cdots\gr{c_n}^{\delta_n} \gr{d_n}^{-\delta_n} \cdots \gr{d_1}^{-\delta_1}$.
\end{itemize}
Since $\gamma$ is a double covering, we conclude that 
\[ x = x \cdot (g^M_a(g^M_b)^{-1})  = x \cdot (\gr{a_1}^{\delta_1} \cdots\gr{a_n}^{\delta_n} \gr{b_n}^{-\delta_n} \cdots \gr{b_1}^{-\delta_1}),\]
which shows that $\gamma$ is a trivial double covering.
\end{proof}

\begin{theorem}\label{TheoremCharacterizationDoubleCentralExtensions}
A double extension of racks (or quandles) is a $\Gamma^1$-covering (also called double central extension of racks and quandles), if and only if it is a double covering (also called algebraically central double extension of racks and quandles). The category of double coverings and the category of $\Gamma^1$-coverings above an extension of racks (or quandles) are isomorphic.
\end{theorem}

\subsection{Centralizing double extensions}\label{SectionCentralization}

Consider a double extension of racks (or quandles) $\alpha\colon{f_A \to f_B}$. We may universally \emph{centralize it} (i.e.~make it into a double covering) by a quotient of its \emph{initial object} $A_{\ttop}$. We studied the reflection of $\CEXT\RCK$ in $\EXT\RCK$.  Now in $\EXT^2\RCK$ (from Definition \ref{DefinitionThreeFoldExtension}), we may identify the full subcategory $\CEXT^2\RCK$ whose objects are the double coverings (or equivalently the $\Gamma^1$-coverings, also called double central extensions). The following result is the $2$-dimensional equivalent of Theorem 
3.4.1 from Part I. We define $\EE_2$ as the class of \emph{$3$-fold extensions} from Definition \ref{DefinitionThreeFoldExtension}.

\begin{theorem}\label{TheoremEpiReflectivityOfCExt}
The category $\CEXT^2\RCK$ is a ($\EE_2$)-reflective subcategory of the category $\EXT^2\RCK$ with left adjoint $\FIi$ and unit $\eta^2$ defined for an object $\alpha\colon{f_A \to f_B}$ in $\EXT^2\RCK$ by \[\eta^2_{\alpha} \defeq (\eta^2_{\alpha_{\ttop}}, \eta^2_{\alpha_{\pperp}}) \defeq ((\eta^2_{A_{\ttop}},\id_{A_{\pperp}}),( \id_{B_{\ttop}},\id_{B_{\pperp}}))\colon{ \alpha \longrightarrow \FIi(\alpha)},\] where $\eta^2_{A_{\ttop}}\colon{A_{\ttop} \to \FII(A_{\ttop})}$ is defined as the quotient of $A_{\ttop}$ by the \emph{centralizing relation} $\Cii(\alpha) \defeq [\Eq(f_A), \Eq(\alpha_{\ttop})]$, and its equivalent descriptions from Corollary \ref{CorollaryCharacterizingC2}.
Observing that $\Cii(\alpha) \leq \Eq(f_A) \cap \Eq(\alpha_{\ttop})$, the image $\FIi(\alpha) \defeq (\FIil(\alpha_{\ttop}),\alpha_{\pperp}) \colon{ \FIil(f_A) \to f_B}$ is defined via the unique factorization of the comparison map $p\colon{A_{\ttop} \to A_{\pperp} \times_{B_{\pperp}} B_{\ttop}}$ through the quotient $\eta^2_{A_{\ttop}}$:
\[\vcenter{\xymatrix@R=10pt@C=15pt{
A_{\ttop} \ar[rd]_-{\eta^2_{A_{\ttop}}} \ar[rr]^-{p} && A_{\pperp} \times_{B_{\pperp}} B_{\ttop}\\
& \FII(A_{\ttop}) \ar[ur]_-{p'}  &
}} \qquad \qquad \vcenter{\xymatrix@1@!0@=32pt{
& A_{\ttop} \ar[rr]^-{\red{\eta^2_{A_{\ttop}}}} \ar[dd]|(.25){\alpha_{\ttop}}|-{\hole} \ar[ld]_-{\db{f_A}} && \FII(A_{\ttop}) \ar[dd]^-{\FIil(\alpha_{\ttop}) \defeq \pi_2 p'} \ar[ld]|-{ \db{\FIil(f_A) \defeq \pi_1 p'} \quad \quad } \\
A_{\pperp} \ar@{{}{=}{}}[rr]^(.75){} \ar[dd]_-{\alpha_{\pperp}} && A_{\pperp} \ar[dd]^(.25){\alpha_{\pperp}} \\
& B_{\ttop} \ar[ld]_-{\db{f_B}} \ar@{{}{=}{}}[rr]^(.25){\red{}}|-{\hole} && B_{\ttop} \ar[ld]^-{\db{f_B}} \\
B_{\pperp} \ar@{{}{=}{}}[rr]_-{\red{}} && B_{\pperp}}}\]
where $\pi_1$ and $\pi_2$ are the projections of $A_{\pperp} \times_{B_{\pperp}} B_{\ttop}$, as in Equation \ref{EquationDoubleExtension}.

The image by $\FIi$ of a morphism $(\sigma,\beta)\colon{\gamma \to \alpha}$ is then given by the identity in all but the initial component: $\FIi(\sigma,\beta) = ((\FII(\sigma_{\ttop}),\sigma_{\pperp}), (\beta_{\ttop},\beta_{\pperp}))$, where $\FII(\sigma_{\ttop})$ is defined by the unique factorization of $\eta^2_{A_{\ttop}} \sigma_{\ttop}$ through $\eta^2_{C_{\ttop}}$, as displayed below, where $P \defeq \FII(C_{\ttop}) \times_{\FII(A_{\ttop})} A_{\ttop}$:
\begin{equation}\label{EquationReflectionSquareFIi}
\vcenter{\xymatrix @R=1pt @ C=11pt{
C_{\ttop} \ar[rrr]^{\sigma_{\ttop}} \ar[rd] |-{p}  \ar[dddd]_-{\eta^2_{C_{\ttop}}}  & & &  A_{\ttop}  \ar[dddd]^-{\eta^2_{A_{\ttop}}} \\
 & P  \ar[rru]|-{\pi_2} \ar[lddd]|-{\pi_1} \\
\\
\\
\FII(C_{\ttop})  \ar@{{}{--}{>}}[rrr]_{\FII(\sigma_{\ttop})} &  & & \FII(A_{\ttop})
}}
\end{equation}
\end{theorem}
\begin{proof}
First observe that the double extension $\FIi(\alpha)$ is indeed a double covering by construction. As was already mentioned in the proof of Theorem \ref{TheoremGaloisStructureD2}, $\eta^2$ is easily seen to be a $3$-fold extension since its bottom component is an isomorphism. Hence given $(a,b,c,d) \in \Eq(\FIil(f_A)) \square \Eq(\FIil(\alpha_{\ttop}))$, it is the quotient of some $(a',b',c',d') \in \Eq(f_A) \square \Eq(\alpha_{\ttop})$ by Lemma \ref{LemmaSurjectionBetweenParallelisticDERel}, and thus for any $x \in \FII(A_{\ttop})$, the elements $x$ and $x \cdot (\gr{a}\, \gr{b}^{-1}\, \gr{c}\, \gr{d}^{-1})$ have $\eta^2_{A_{\ttop}}$-pre-images $x'$ and $x' \cdot (\gr{a'}\, (\gr{b'})^{-1}\, \gr{c'}\, (\gr{d'})^{-1})$ which are in relation by $\Cii(\alpha)$. 

Then we show that $\eta^2_{\alpha}$ has the right universal property. We first show the universality of $\eta^2_{\alpha}$ in the subcategory $\EXT^2(f_B)$ of double extensions over $f_B$. Consider a double covering $\gamma\colon{f_C \to f_B}$ and a morphism $\tau \in \EXT^2(f_B)$ between $\alpha$ and $\gamma$, yielding the commutative diagram of plain arrows in $\EXT\RCK$ on the left, whose top and bottom components in $\C$ are given on the right.
\[\vcenter{\xymatrix@R=15pt@C=15pt{
\FIil(f_A) \ar@{{}{..}{>}}@/_2ex/@<-4ex>[dd]|-{\vartheta \defeq (\vartheta_{\ttop},\theta_{\pperp})\quad} \ar[drr]^-{\FIi(\alpha)}  & & \\
f_A \ar[u]^-{\eta^2_{\alpha_{\ttop}}} \ar[d]_-{\theta} \ar[rr]^-{\alpha} && f_B\\
f_C \ar[urr]_-{\gamma}  & &
}} \qquad \Bigg{|} \qquad \vcenter{\xymatrix@R=15pt@C=15pt{
\FII(A_{\ttop}) \ar@{{}{..}{>}}@/_2ex/@<-4ex>[dd]|-{\vartheta_{\ttop}} \ar[drr]^-{\FIil(\alpha_{\ttop})}  & & \\
A_{\ttop} \ar[u]^-{\eta^2_{A_{\ttop}}} \ar[d]_-{\theta_{\ttop}} \ar[rr]^-{\alpha_{\ttop}} && B_{\ttop}\\
C_{\ttop} \ar[urr]_-{\gamma_{\ttop}}  & &
}} \qquad \vcenter{\xymatrix@R=15pt@C=15pt{
A_{\pperp} \ar@{{}{..}{>}}@/_2ex/@<-3ex>[dd]|-{\theta_{\pperp}}
\ar[drr]^-{\alpha_{\pperp}}  & & \\
A_{\pperp} \ar@{{}{=}{}}[u]^-{} \ar[d]_-{\theta_{\pperp}} \ar[rr]^-{\alpha_{\pperp}} && B_{\pperp}\\
C_{\pperp} \ar[urr]_-{\gamma_{\pperp}}  & &
}}  \]
Given any pair $(x \cdot g,x) \in \Cii(\alpha)$, where $g$ is some $\langle f_A,\alpha_{\ttop}\rangle$-symmetric path, $\vec{\theta}(g)$ is a $\langle f_C,\gamma_{\ttop}\rangle$-symmetric path and thus $\theta(x) = \theta(x) \cdot \vec{\theta}(g)$ since $\gamma$ is a double covering. As a consequence, $\Eq(\theta_{\ttop}) \leq \Cii(\alpha)$ and thus there exists a unique factorization $\vartheta_{\ttop}$ of $\theta_{\ttop}$ through $\eta^2_{A_{\ttop}}$. Since $f_C \vartheta_{\ttop} \eta^2_{A_{\ttop}} = \theta_{\pperp} \FIil(f_A) \eta^2_{A_{\ttop}}$, and $\eta^2_{A_{\ttop}}$ is an epimorphism, we may define the morphism $\vartheta \defeq (\vartheta_{\ttop}, \theta_{\pperp}) \colon {\FIil(f_A) \to f_C}$, which is moreover a double extension by Lemma \ref{LemmaWeakRightCancellation}. This shows the existence of a factorization of $\theta$ through $\eta^2_{\alpha_{\ttop}}$. The uniqueness of $\vartheta$ is easily deduced from the uniqueness in each component. 

Now working in the category $\EXT^2\RCK$, we consider a double covering $\gamma\colon{f_C \to f_D}$ and a morphism $(\tau,\iota)\colon{\alpha \to \gamma}$ in $\EXT^2\RCK$. We compute the pullback $\rho$ of $\gamma$ along $\iota$ and the induced comparison map $\pi$ of the underlying square of $(\tau,\iota)$ in $\EXT\RCK$:
\[
\vcenter{\xymatrix@!0@C=35pt@R=18pt{
\FIil(f_A) \ar[rrrrd]^{\FIi(\alpha)}  \\
& f_A \ar[ul]^-{\eta^2_{\alpha_{\ttop}}} \ar[rrr]_-{\alpha} \ar[rd]|-{\pi}  \ar[ddd]_-{\tau} & & & f_B  \ar[ddd]^-{\iota} \\
& & f_C \times_{f_D} f_B \ar[rru]_-{\rho} \ar[ldd]^-{\varrho} &   \\
\\
& f_C  \ar[rrr]_{\gamma} & & & f_D,
}}
\]
Since $\rho$ is a double covering (by pullback-preservation), we obtain $\vartheta\colon{\FIil(f_A) \to f_C \times_{f_D} f_B}$ by the preceding discussion. Then the morphism $(\varrho \vartheta, \iota)$ is a factorization of $(\tau,\iota)$ through $\eta^2(\alpha)$ which is easily shown to be unique.
\end{proof}

Note that, as usual, the monadicity of $\I$ implies that $\CEXT^2\C$, is closed under limits computed in $\EXT^2\C$. Also since $\eta^2$ has regular epimorphic components, double coverings are closed under subobjects in $\EXT^2\C$ (see for instance \cite[Section 3.1]{JK1994}; note that the same comments hold for the adjunction $\Fi \dashv \I$). Closure by quotients along $3$-fold extensions was discussed in Remark \ref{RemarkDoubleCentralExtensionsAreClosedUnderQuotients}. We conclude the proof that $\FIi \dashv \I$ fits into a strongly Birkhoff Galois structure $\Gamma^2$ in the next article, \emph{Part III}, where we study higher coverings of arbitrary dimensions.

\section{Further developments}\label{SectionFurtherDevelopments}

Besides the following theoretical developments, more explicit examples of double coverings should be studied, for instance in the known contexts of application cited in Part I. Now from the perspective of categorical Galois theory, future developments should also include the description of a \emph{weakly universal double covering} above an extension, and subsequently the characterization of the \emph{fundamental double groupoid} of an extension (see for instance \cite{BroJan2004}). From there, the \emph{fundamental theorem of categorical Galois theory} should be applied in order to ``classify'' the double coverings above an extension.

Another obvious line of work concerns the commutator defined in Section \ref{SectionThinkingAboutCommutator}. A review of the links between commutators and Galois theory can be found in \cite{Jan2016}. For instance, it should be checked whether or not our commutator coincides with (or compares to) what was defined for general varieties in terms of internal pregroupoids \cite{JanPed2001}, or other theories such as the classical approach of \cite{FreMck1987} which was already applied in this context \cite{BonSta2019}. In the last paragraphs below, we suggest to apply the developments of \cite{Jan2008} to our context, with the objective to investigate the links between categorical Galois theory and \emph{homology} \cite{FeRoSa1995,CJKLS2003, DuEverGra2012} within racks and quandles (see also \cite[Section 9]{Eis2014}).

\subsection{Galois structure with abstract commutator}\label{SectionGaloisStructureWithAbstractCommutator}
 
As it was suggested in Part I, one of the important outcomes for the application of higher categorical Galois theory in groups was the elegant generalization of the \emph{higher Hopf formulae} from \cite{BroEll1988} to \emph{semi-abelian categories} \cite{JanMarTho2002,EvGrVd2008,Ever2010,DuEverGra2012}. In \cite{Jan2008}, G.~Janelidze shares his perspective on how to understand the mechanics behind the Hopf formulae from the perspective of categorical Galois theory, and in particular via the description and understanding of what an \emph{abstract Galois group} is. He introduces the definition of a \emph{Galois structure with (abstract) commutators}, which is suggested as another starting point (more general than that from \cite{EvGrVd2008}) for applying the methodology that he illustrates in the context of groups. In this section, we adapt this definition in order to include the covering theory of quandles as an example, in such a way which is compatible with the aims and developments from \cite{Jan2008}. Further details about the application of the ideas from \cite{Jan2008} to the covering theory of quandles is left for future work. 

Our definition of Galois structure with (abstract) commutators is not aimed at being the most general possible. Our main point is the use of higher extensions to clarify the conditions which are displayed in \cite{Jan2008}.

\begin{definition}\label{DefinitionGaloisWithCommutator}
A Galois structure with commutators is a system $\Gamma = (\C, \X, F, \I, \eta, \epsilon, \EE, [-,-])$, in which:
\begin{enumerate}
\item $\Gamma = (\C, \X, F, \I, \eta, \epsilon, \EE)$ is an \emph{admissible} Galois structure (see Convention \ref{ConventionGaloisStructure} as well as \cite{JK1994,BorJan1994} for \emph{admissibility});
\item\label{ItemRightCancelletion} If $f = pq$ in $\C$ and $f$ and $q$ are in $\EE$, then $p$ is in $\EE$;
\item $[-,-]$ is a family of binary operations $\ER_{\EE}(A) \times \ER_{\EE}(A) \to \ER_{\EE}(A) $defined for each $A$ in $\C$ and all written as $(S, T) \mapsto [S, T]$; here $\ER_{\EE}(A)$ denotes the class of $\EE$-congruences on $\C$, i.e. the class of subobjects of $A \times A$ that are kernel pairs of morphisms from $\EE$;
\item For $S$ and $T$ in $\ER_{\EE}(A)$, we always have $[S,T] \leq S\cap T$;
\item If $(\sigma,\beta)\colon{\gamma \to \alpha}$ is a morphism  between the double extensions $\gamma$ and $\beta$, then $(\sigma, \beta)$ induces a morphism $[\sigma_{\ttop}]\colon [\Eq(f_C), \Eq(\gamma_{\ttop})] \to [\Eq(f_A), \Eq(\alpha_{\ttop})]$; 
\item if $(\sigma,\beta)\colon{\gamma \to \alpha}$ above is a $3$-fold extension, then $[\sigma_{\ttop}]\colon [\Eq(f_C), \Eq(\gamma_{\ttop})] \to [\Eq(f_A), \Eq(\alpha_{\ttop})]$ is in $\EE$;
\item For each $A$ in $\C$, $F(A) = A / [A\times A,A\times A]$ and $\eta_A$ is the canonical morphism $A \to A / [A\times A,A\times A]$;
\item For a morphism $p\colon{E \to B}$ from $\EE$, $p$ is a $\Gamma$-covering if and only if $[E\times E, \Eq(p)] = \Delta_E$, i.e. $[E\times E, \Eq(p)]$ is the smallest congruence on $E$.
\end{enumerate}
\end{definition}

Observe that conditions $(5)$ and $(6)$ are the two conditions which differ from G.~Janelidze's presentation (see $4.4(i)$ and $(j)$ in \cite{Jan2008}). We explain how to translate from his context to ours. In order to avoid confusion, let us point out a small typo in \cite{Jan2008}: the conclusions of Condition (g) in Definitions 4.1 and Condition (f) in Definition 4.4 should be that $p$ is in $F$ (rather than in $C$ -- see Condition \eqref{ItemRightCancelletion} below). Also in Condition (i) of Definition 4.4 we should read $[E\times E, \Eq(p)]$ instead of $[E \times E, \Ker(p)]$.

Now using our notations, G.~Janelidze merely considers the data of $\sigma_{\ttop}\colon C_{\ttop} \to A_{\ttop}$, $S \defeq \Eq(f_C)$, $T \defeq \Eq(\gamma_{\ttop})$, $S' \defeq \Eq(f_A)$ and $T' \defeq \Eq(\alpha_{\ttop})$ as well as \emph{induced morphisms} $s\colon S \to S'$ and $t\colon T \to T'$. From this data, we easily build the entire morphism $(\sigma,\beta)$ with no further assumptions. The only difference is then the assumption that $\alpha$ and $\gamma$ are not merely pushout squares of extensions, but also double extensions. Whenever $\C$ is a Mal'tsev category, which is the case in the examples considered by G.~Janelidze and others \cite{Ever2007,EvGrVd2008,EverVdL2011,EverVdL2012,Ever2012}, $\alpha$ and $\gamma$ are automatically double extensions. In our context, this is the ``natural'' extra requirement to work with (we work locally with congruences which commute since in our context, commutativity of congruences does not hold everywhere). Now when $s$ and $t$ are further required to be extensions (such as in $4.4(j)$), by the same reasoning, the natural generalization from Mal'tsev categories consists in requesting $(\sigma,\beta)$ to be a square of double extensions. Finally observe that $4.4(j)$ was already challenged in Remark $4.6$ of \cite{Jan2008}. Observe that under the restrictions suggested by T.~Everaert or G.~Janelidze, our square of double extensions $(\sigma,\beta)$ becomes a $3$-fold extension. Hence our choice of presentation is arguably an adequate and elegant variation from \cite{Jan2008}, which is coherent with the example we are interested in, as well as the examples considered in \cite{Jan2008} and related works.

\begin{example}\label{ExampleGaloisStructureWithAbstractCommutator}
From the results of Section \ref{SectionThinkingAboutCommutator}, and Lemma \ref{LemmaSurjectionBetweenParallelisticDERel}, we deduce that the Galois structure from Theorem \ref{TheoremGaloisStructureD2} together with the operation $[-,-]$ from Definition \ref{DefinitionCommutator} satisfies the conditions of Definition \ref{DefinitionGaloisWithCommutator}.
\end{example}

Since compatibility with unions is understood as an important property for commutators, we show the following result which may be used to study Example \ref{ExampleGaloisStructureWithAbstractCommutator} from the perspective of \cite{Jan2008}. Note that our hypotheses might not be optimal; we deduce the \emph{modular law} locally from the less general permutability conditions on our congruences. These are arguably more suitable for this context in which (repeatedly) we have been using, locally, some properties which are globally satisfied in Mal'tsev categories.
\begin{lemma}
Let $A$ be a quandle, $R$, $S$ and $T$ congruences on $A$ such that $S \leq R$. If $R$, $S$ and $T$ commute two by two, and moreover $R\cap T$ commutes with $S$ (for instance when $S$ commutes with all congruences), then $[R,S \cup T] = [R,S] \cup [R, T] = [S, A \times A] \cup [R, T]$.
\end{lemma}
\begin{proof}
First observe that $ [R,S] \cup [R, T] \leq [R,S \cup T]$ is an easy consequence of Lemma \ref{LemmaCommutatorCompatibleWithOrder}. Then consider a generator
\[  (x \qndop a \qndiop b \qndop c \qndiop d , x) \ \in \  [R,S \cup T] \quad  \text{ for some } \quad  x\in A \  \text{ and } \  \vcenter{\xymatrix @C= 15pt@R=10pt {
a \ar@{-}[r]|-{R}  \ar@{..}[d]|-{S \cup T}  & b  \ar@{..}[d]|-{S \cup T} \\
d \ar@{-}[r]|-{R}  & c
}} \ \in\  R\square (S \cup T).\]
Since $S$ and $T$ commute, there is $b_0 \in A$ such that $(b,b_0) \in S$ and $(b_0,c) \in T$. Moreover since $R$ commutes with $S$ and $T$, there are $a_0$ and respectively $d_0$ such that $(a,a_0) \in S$, $(a_0,b_0) \in R$, $(d,d_0) \in T$ and $(d_0,b_0) \in R$. Hence $(a_0,d_0) \in R \cap (S \cup T)$. Using the modular law $(a_0,d_0) \in S \cup (R \cap T)$ and thus there is $a_1 \in A$ such that $(a_0,a_1) \in S$ and $(a_1,d_0) \in (R \cap T)$. From there observe that $(a_1,b_0) \in R$ and $(a_1,d) \in T$ such that we obtain:
\[ \vcenter{\xymatrix @C= 15pt@R=10pt {
a \ar@{-}[d]|-{R}  \ar@{..}[r]|-{S} & a_0 \ar@{-}[d]|-{R}  \ar@{..}[r]|-{S}   & a_1 \ar@{-}[d]|-{R}  \ar@{..}[r]|-{T} & d \ar@{-}[d]|-{R}  \\
 b  \ar@{..}[r]|-{S}  & b_0  \ar@{..}[r]|-{S} & b_0  \ar@{..}[r]|-{T}  & c
}}\]
Considering each of these three squares separately, by definition of $[R,S]$ and $[R, T]$ we derive that $x$ is in relation by $[R,S] \cup [R, T]$ with the element 
\[x \qndop a \qndiop b \qndop b_0 \qndiop a_0 \qndop a_0 \qndiop b_0 \qndop b_0 \qndiop a_1 \qndop a_1 \qndiop b_0 \qndop c \qndiop d, \]
which reduces to $x \qndop a \qndiop b \qndop c \qndiop d$.
\end{proof}

\section*{Aknowledgments}
The main aims of this paper were reached during a research visit at the \emph{University of Cape Town}, which was funded by the \emph{Concours des bourses de voyage 2018} awarded by the \emph{Fédération Wallonie-Bruxelles}. I am very grateful to George Janelidze for our weekly conversations and lunches on the \emph{UCT} campus. His request for the definition of a commutator and his encouragements and advice for the completion of the main results of my Ph.D. project during my stay in South-Africa had a significant influence on the outcomes of my research. Let me add many thanks to my supervisors Tim Van der Linden and Marino Gran for their support, advice, contributions and careful proofreading of this work.

\bibliography{Part2.bib}

\providecommand{\bysame}{\leavevmode\hbox to3em{\hrulefill}\thinspace}
\providecommand{\MR}{\relax\ifhmode\unskip\space\fi MR }
\providecommand{\MRhref}[2]{%
  \href{http://www.ams.org/mathscinet-getitem?mr=#1}{#2}
}
\providecommand{\href}[2]{#2}
\begin{thebibliography}{10}

\bibitem{Bar1971}
M.~Barr, \emph{Exact categories}, Lecture Notes Math (Berlin, Heidelberg),
  Exact Categories and Categories of Sheaves, vol. 236, Springer, 1971,
  pp.~1--120.

\bibitem{Ber2008}
W.~Bertram, \emph{Differential geometry, {Lie} groups and symmetric spaces over
  general base fields and rings}, Memoirs of the American Mathematical Society,
  vol. 900, AMS, Providence, R.I., 2008.

\bibitem{BonSta2019}
M.~Bonatto and D.~Stanovský, \emph{Commutator theory for racks and quandles},
  2019, Preprint \url{https://arxiv.org/abs/1902.08980v3}.

\bibitem{BorBou2004}
F.~Borceux and D.~Bourn, \emph{Mal'cev, protomodular, homological and
  semi-abelian categories}, Math. Appl., vol. 566, Kluwer Acad. Publ., 2004.

\bibitem{BorJan1994}
F.~Borceux and G.~Janelidze, \emph{Galois theories}, Cambridge studies in
  advanced mathematics, Vol. 72, Cambridge University Press, 1994.

\bibitem{Bou2003}
D.~Bourn, \emph{The denormalized {$3\times 3$} lemma}, J.~Pure Appl. Algebra
  \textbf{177} (2003), 113--129.

\bibitem{Bou2015}
D.~Bourn, \emph{A structural aspect of the category of quandles}, Journal of
  Knot Theory and Its Ramifications \textbf{24} (2015).

\bibitem{BouGra2004bis}
D.~Bourn and M.~Gran, \emph{Regular, protomodular, and abelian categories},
  {C}ategorical Foundations: Special Topics in Order, Topology, Algebra and
  Sheaf Theory (M.~C. Pedicchio and W.~Tholen, eds.), Encyclopedia of Math.
  Appl., vol.~97, Cambridge Univ. Press, 2004, pp.~165--211.

\bibitem{BouMon2018}
D.~Bourn and A.~Montoli, \emph{The 3x3 lemma in the ${\Sigma}$-{Maltsev} and
  ${\Sigma}$-protomodular settings. {A}pplications to monoids and quandles},
  Homology, Homotopy and Applications \textbf{21} (2018).

\bibitem{Bou2020}
Dominique Bourn, \emph{Split epimorphisms as a productive tool in {U}niversal
  {A}lgebra}, Algebra Universalis \textbf{82} (2020), no.~1, 1.

\bibitem{Bri1988}
E.~Brieskorn, \emph{Automorphic sets and braids and singularities}, Braids
  (Providence, RI), Contemporary Mathematics, vol.~78, Amer. Math. Soc., 1988,
  pp.~45--115.

\bibitem{BroEll1988}
R.~Brown and G.~Ellis, \emph{Hopf formulae for the higher homology of a group},
  Bulletin of the London Mathematical Society \textbf{20} (1988).

\bibitem{BroJan2004}
R.~Brown and G.~Janelidze, \emph{Galois theory and a new homotopy double
  groupoid of a map of spaces}, Applied Categorical Structures \textbf{12}
  (2004), no.~1, 63 -- 80.

\bibitem{CKP1993}
A.~Carboni, G.~M. Kelly, and M.~C. Pedicchio, \emph{Some remarks on {Maltsev}
  and {Goursat} categories}, Appl. Categ. Struct. \textbf{1} (1993), 385--421.

\bibitem{CPP1992}
A.~Carboni, M.~C. Pedicchio, and N.~Pirovano, \emph{Internal graphs and
  internal groupoids in {M}al'cev categories}, Am. Math. Soc. for the Canad.
  Math. Soc., Providence, 1992, pp.~97--109.

\bibitem{CaKaSa2001}
J.~S. Carter, S.~Kamada, and M.~Saito, \emph{Geometric interpretations of
  quandle homology}, Journal of Knot Theory and Its Ramifications \textbf{10}
  (2001), no.~03, 345--386.

\bibitem{CJKLS2003}
J.S. Carter, D.~Jelsovsky, S.~Kamada, L.~Langford, and M.~Saito, \emph{Quandle
  cohomology and state-sum invariants of knotted curves and surfaces}, Trans.
  Amer. Math. Soc. \textbf{355} (2003), no.~4, 3947–3989.

\bibitem{Dri1992}
V.~G. Drinfeld, \emph{On some unsolved problems in quantum group theory},
  Quantum Groups (Berlin, Heidelberg) (Petr~P. Kulish, ed.), Springer Berlin
  Heidelberg, 1992, pp.~1--8.

\bibitem{DuEverGra2012}
M.~Duckerts, T.~Everaert, and M.~Gran, \emph{A description of the fundamental
  group in terms of commutators and closure operators}, J.~Pure Appl. Algebra
  \textbf{216} (2012), no.~8--9, 1837--1851.

\bibitem{DuEvMo2017}
M.~Duckerts-Antoine, V.~Even, and A.~Montoli, \emph{How to centralize and
  normalize quandle extensions}, Journal of Knot Theory and Its Ramifications
  \textbf{27} (2018), no.~02, 1850020.

\bibitem{Eis2014}
M.~Eisermann, \emph{Quandle coverings and their {Galois} correspondence}, Fund.
  Math. \textbf{225(1)} (2014), 103--168.

\bibitem{Eve2014}
V.~Even, \emph{A {Galois}-theoretic approach to the covering theory of
  quandles}, Appl. Categ. Struct. \textbf{22} (2014), no.~5, 817--831.

\bibitem{EveGr2014}
V.~Even and M.~Gran, \emph{On factorization systems for surjective quandle
  homomorphisms}, Journal of Knot Theory and Its Ramifications \textbf{23}
  (2014), no.~11, 1450060.

\bibitem{EveGr2016}
V.~Even and M.~Gran, \emph{Closure operators in the category of quandles},
  Topology and its Applications \textbf{200} (2016), 237 -- 250.

\bibitem{EveGrMo2016}
V.~Even, M.~Gran, and A.~Montoli, \emph{A characterization of central
  extensions in the variety of quandles}, Theory and Applications of Categories
  \textbf{31} (2016), 201--216.

\bibitem{Ever2007}
T.~Everaert, \emph{Relative commutator theory in varieties of
  {$\Omega$}-groups}, Journal of Pure and Applied Algebra \textbf{210} (2007),
  no.~1, 1 -- 10.

\bibitem{Ever2010}
T.~Everaert, \emph{Higher central extensions and {Hopf} formulae}, Journal of
  Algebra \textbf{324} (2010), no.~8, 1771 -- 1789.

\bibitem{Ever2012}
T.~Everaert, \emph{Effective descent morphisms of regular epimorphisms},
  Journal of Pure and Applied Algebra \textbf{216} (2012), no.~8, 1896 -- 1904.

\bibitem{Ever2014}
T.~Everaert, \emph{Higher central extensions in {Mal'tsev} categories}, Applied
  Categorical Structures \textbf{22} (2014).

\bibitem{EvGoeVdl2012}
T.~Everaert, J.~Goedecke, and T.~Van~der Linden, \emph{Resolutions, higher
  extensions and the relative {M}al'tsev axiom}, J.~Algebra \textbf{371}
  (2012), 132--155.

\bibitem{EvGrVd2008}
T.~Everaert, M.~Gran, and T.~Van~der Linden, \emph{Higher {Hopf} formulae for
  homology via {Galois} theory}, Advances in Mathematics \textbf{217} (2008),
  no.~5, 2231 -- 2267.

\bibitem{EverVdL2011}
T.~Everaert and T.~Van~der Linden, \emph{Galois theory and commutators},
  Algebra Universalis \textbf{65} (2011), no.~2, 161--177.

\bibitem{EverVdL2012}
T.~Everaert and T.~Van~der Linden, \emph{Relative commutator theory in
  semi-abelian categories}, J.~Pure Appl. Algebra \textbf{216} (2012),
  no.~8--9, 1791--1806.

\bibitem{FenRou1992}
R.~Fenn and C.~Rourke, \emph{Racks and links in codimension two}, J. Knot
  Theory Ramifications \textbf{1(4)} (1992), 343--406.

\bibitem{FeRoSa1995}
R.~Fenn, C.~Rourke, and B.~Sanderson, \emph{Trunks and classifying spaces},
  Applied Categorical Structures \textbf{3} (1995), no.~4, 321--356.

\bibitem{FreMck1987}
R.~Freese and R.~Mckenzie, \emph{Commutator theory for congruence modular
  varieties},  \textbf{125} (1987).

\bibitem{GraRo2004}
M.~Gran and V.~Rossi, \emph{Galois theory and double central extensions},
  Homology, Homotopy Appl. \textbf{6} (2004), no.~1, 283--298.

\bibitem{Hat2001}
A.~Hatcher, \emph{Algebraic topology}, Cambridge University Press, 2001.

\bibitem{Hel2012}
S.~Helgason, \emph{Differential geometry, {L}ie groups, and symmetric spaces},
  Graduate studies in mathematics, vol.~34, American Mathematical Society,
  Providence, RI, 2012.

\bibitem{ImKel1986}
G.~B. Im and G.M. Kelly, \emph{On classes of morphisms closed under limits}, J.
  Korean Math. Soc. \textbf{23} (1986), no.~1, 1 -- 18.

\bibitem{Jan1990}
G.~Janelidze, \emph{Pure {Galois} theory in categories}, Journal of Algebra
  \textbf{132} (1990), no.~2, 270 -- 286.

\bibitem{JanComo1991}
G.~Janelidze, \emph{Precategories and {G}alois theory}, Category {T}heory,
  {P}roceedings {C}omo 1990 (A.~Carboni, M.~C. Pedicchio, and G.~Rosolini,
  eds.), Lecture Notes in Math., vol. 1488, Springer, 1991, pp.~157--173.

\bibitem{Jan1991}
G.~Janelidze, \emph{What is a double central extension ? (the question was
  asked by {Ronald Brown})}, Cah. de Top. et Géom. Diff. Cat. \textbf{32}
  (1991), no.~3, 191--201 (eng).

\bibitem{Jan2008}
G.~Janelidze, \emph{Galois groups, abstract commutators, and {Hopf} formula},
  Applied Categorical Structures \textbf{16} (2008), 653--668.

\bibitem{Jan2016}
G.~Janelidze, \emph{A history of selected topics in categorical algebra {I}:
  {F}rom {G}alois theory to abstract commutators and internal groupoids},
  Categories and General Algebraic Structures with Applications \textbf{5}
  (2016), no.~1, 1--54.

\bibitem{JK1994}
G.~Janelidze and G.M. Kelly, \emph{Galois theory and a general notion of
  central extension}, Journal of Pure and Applied Algebra \textbf{97} (1994),
  no.~2, 135 -- 161.

\bibitem{JanMarTho2002}
G.~Janelidze, L.~M{\'a}rki, and W.~Tholen, \emph{Semi-abelian categories},
  J.~Pure Appl. Algebra \textbf{168} (2002), no.~2--3, 367--386.

\bibitem{JanPed2001}
G.~Janelidze and M.~C. Pedicchio, \emph{Pseudogroupoids and commutators},
  Theory Appl. Categ. \textbf{8} (2001), no.~15, 408--456.

\bibitem{JanSo2002}
G.~Janelidze and M.~Sobral, \emph{Finite preorders and topological descent.
  {I}. {S}pecial volume celebrating the 70th birthday of professor max kelly.},
  {J.~Pure} Appl. Algebra \textbf{175} (2002), no.~1--3, 187---205.

\bibitem{JanSo2011}
G.~Janelidze and M.~Sobral, \emph{Descent for regular epimorphisms in {B}arr
  exact goursat categories}, Applied Categorical Structures \textbf{19} (2011),
  271--276.

\bibitem{JanSoTho2004}
G.~Janelidze, M.~Sobral, and W.~Tholen, \emph{Beyond {B}arr exactness:
  Effective descent morphisms}, {C}ategorical Foundations: Special Topics in
  Order, Topology, Algebra and Sheaf Theory (M.~C. Pedicchio and W.~Tholen,
  eds.), Encyclopedia of Math. Appl., vol.~97, Cambridge Univ. Press, 2004,
  pp.~359--405.

\bibitem{JanTho1994}
G.~Janelidze and W.~Tholen, \emph{Facets of descent, {I}}, Applied Categorical
  Structures \textbf{2} (1994), no.~3, 245--281.

\bibitem{JanTho1997}
G.~Janelidze and W.~Tholen, \emph{Facets of descent, {II}}, Applied Categorical
  Structures \textbf{5} (1994), 229--248.

\bibitem{Joy1979}
D.~Joyce, \emph{An algebraic approach to symmetry and applications in knot
  theory}, Ph.D. thesis, University of Pennsylvenia, 1979.

\bibitem{Joy1982}
D.~Joyce, \emph{A classifying invariant of knots, the knot quandle}, Journal of
  Pure and Applied Algebra \textbf{23(1)} (1982), 37--65.

\bibitem{Kau2012}
L.H. Kauffman, \emph{Knots and physics (fourth edition)}, Series On Knots And
  Everything, World Scientific Publishing Company, 2012.

\bibitem{Loos1969}
O.~Loos, \emph{Symmetric spaces: General theory}, Mathematics Lecture Note
  Series, W. A. Benjamin, 1969.

\bibitem{Lor2020}
F.~Loregian, \emph{Rosen's no-go theorem for regular categories}, 2020,
  Preprint \url{https://arxiv.org/abs/2012.11648}.

\bibitem{McLane1997}
S.~Mac~Lane, \emph{Categories for the working mathematician}, second edition
  ed., Springer, Chicago, 1997.

\bibitem{Ren2020}
F.~Renaud, \emph{Higher coverings of racks and quandles -- {P}art {I}}, 2020,
  Preprint \url{https://arxiv.org/abs/2007.03385}.

\bibitem{Ren2021}
F.~Renaud, \emph{Re-writing the elements in the intersection of the kernels of
  two morphisms between free groups}, 2021, In preparation.

\bibitem{Smth1976}
J.~D.~H. Smith, \emph{{M}al'cev varieties}, Lecture Notes in Math., vol. 554,
  Springer, 1976.

\bibitem{Tak1943}
M.~Takasaki, \emph{Abstraction of symmetric transformations}, Tohoku
  Mathematical Journal, First Series \textbf{49} (1943), 145--207.

\end{thebibliography}
\bibliographystyle{amsplain-nodash}

\end{document}